\documentclass[]{amsart}
\usepackage{amssymb,amsmath} %theorem,
\usepackage[mathscr]{euscript}

\newcounter{sec}

\newcounter{punct}[sec]

\def\punct{\refstepcounter{punct}{\arabic{sec}.\arabic{punct}.  }}

\newtheorem{theorem}{Theorem}[sec]
\newtheorem{proposition}[theorem]{Proposition}

\newtheorem{lemma}[theorem]{Lemma}

\newtheorem{corollary}[theorem]{Corollary}
\newtheorem{observation}[theorem]{Observation}

\newtheorem{conjecture}[theorem]{Conjecture}

\newtheorem{convention}{Convention}[sec]

\def\COUNTERS{\addtocounter{sec}{1}
              \setcounter{punct}{0}
          \setcounter{equation}{0}
          \setcounter{theorem}{0}
          }

\def\SL{\mathrm {SL}}

\def\GL{\mathrm  {GL}}

\def\Aff{\mathrm  {Aff}}

\def\Gr{\mathrm{Gr}}

\def\OSp{\mathrm {OSp}}
\def\AOSp{\mathrm {AOSp}}

\def\B{\mathfrak B}

\def\phi{\varphi}
\def\epsilon{\varepsilon}
\def\kappa{\varkappa}

\def\aff{\mathrm{aff}}

\def\le{\leqslant}
\def\ge{\geqslant}

\renewcommand{\Im}{\mathop{\rm Im}\nolimits}

\def\la{\langle}
\def\ra{\rangle}

\newcommand{\im}{\mathop{\rm im}\nolimits}
\newcommand{\indef}{\mathop{\rm indef}\nolimits}
\newcommand{\dom}{\mathop{\rm dom}\nolimits}

\newcommand{\Ber}{\mathop{\rm Ber}\nolimits}
\newcommand{\graph}{\mathop{\rm graph}\nolimits}

\newcommand{\Lagr}{\mathop{\rm Lagr}\nolimits}

\def\bF{\mathbf F}

\def\cA{\EuScript A}
\def\cB{\EuScript B}
\def\cC{\EuScript C}
\def\cD{\EuScript D}

\def\cF{\EuScript F}

\def\cH{\EuScript H}
\def\cJ{\EuScript J}

\def\cL{\EuScript L}
\def\cM{\EuScript M}
\def\cN{\EuScript N}
\def\cO{\EuScript O}
\def\cP{\EuScript P}
\def\cQ{\EuScript Q}

\def\cU{\EuScript U}

\def\cW{\EuScript W}

\def\cZ{\EuScript Z}

\def\ccF{\boldsymbol {\cF}}

\def\frB{\mathfrak B}

\def\frD{\mathfrak D}

\def\frL{\mathfrak L}

\def\fra{\mathfrak a}

\def\frc{\mathfrak c}

\def\frg{\mathfrak g}
\def\frh{\mathfrak h}

\def\frk{\mathfrak k}
\def\frl{\mathfrak l}

\def\frn{\mathfrak n}
\def\fro{\mathfrak o}
\def\frp{\mathfrak p}
\def\frq{\mathfrak q}

\def\frs{\mathfrak s}

\def\fru{\mathfrak u}

\def\frx{\mathfrak x}

\def\bfF{\mathbf F}

\def\bfQ{\mathbf Q}
\def\bfR{\mathbf R}

\def\R {{\mathbb R }}
 \def\C {{\mathbb C }}
  \def\Z{{\mathbb Z}}
  
\def\K{{\mathbb K}}

\def\T{\mathbb T}

 \def\ov{\overline}
\def\wt{\widetilde}
\def\wh{\widehat}

\def\F{\mathbf F}

\def\b{\mathfrak b}

\def\sm{\smallskip}

\def\0{{\ov 0}}
\def\1{{\ov 1}}

\def\Diff{\mathrm{Diff}}
\def\SDiff{\mathrm{SDiff}}
\def\SCont{\mathrm{SCont}}
\def\vect{\mathfrak{vect}}
\def\ns{\mathfrak{ns}}
\def\sheis{\mathsf{sheis}}
\def\frheis{\mathfrak{sheis}}
\def\free{{\mathfrak{free}}}
\def\vir{\mathfrak{vir}}
\def\cont{\mathfrak{cont}}
\def\osp{\mathfrak{osp}}

\def\sfS{\mathsf S}

\def\Assoc{{\mathrm{Assoc}}}
\def\CONT{{\mathrm{Cont}}}
\def\Cont{{\mathrm{Cont}}}
\def\VECT{{\mathrm{Vect}}}
\def\Vect{{\mathrm{Vect}}}

\def\vel{{\check{\circ\circ}}}

\begin{document}

%$\frg_{\check{\text{\tiny$\circ\circ$}}}$

%$\frg_{\check{\circ\circ}}$

%$G_{\check{\circ\circ}}$

%$\frg_{\check{\circ\!\circ}}$

%\end{document}

%$\frg_{_o\<_(_)/(_)}}$

%$_o\<_(_)/(_)$
%$_o_\backslash<_(_)/(_)$

%$\frg_\velt$

%$\frg_\vel$

%$\frg_{\text{\tiny{$|$}}\!\!\top}\!\!\text{\tiny{$|$}}}$

%$\frg_\vell$

%$\frg_{\check{\boxed{\circ\circ}}}$

%$\check{\mbox{\text{\aeolicbii} }}$
%$\stackrel{\checkmark}{\mbox{\text{\aeolicbii} }}$
%$\frg_{\stackrel{\checkmark}{\mbox{\text{\scriptsize\aeolicbii} }}}$

%$\check{\circ\circ}$

%$\frg_{\check{\tiny\circ\circ}}$

\def\tto{\rightrightarrows}

%\hyphenation{Fr\'e-chet Po-in-ca-r\'e}

\def\Frechet{Fr\'echet \vphantom{.}}
\def\Poincare{Poincar\'e}

%well defined

%$vect^\R$, $\cont^R$

%$\Cont$ $\SCont$

%$L_\vir$

%of variables

%\newpage

\begin{center}
\bf\Large
On a super-Virasoro group, a semigroup of annuli, and Gauss--Berezin integral 
operators
\sc\large

\bigskip

Yury A. Neretin%
\footnote{Supported by the grants FWF, projects  P19064, P31591,
the grant  NWO.047.017.015, and the grant JSPS-RFBR-07.01.91209.}

\end{center}

{\small
Denote by ${\mathcal A}$ the Grassmann algebra with a countable number of generators, by ${\mathcal A}_{\overline 0}$, ${\mathcal A}_{\overline 1}$
its even and odd parts. We consider supergroups as groups over $\mathcal A$.
 Consider the Neveu--Schwarz Lie superalgebra
$\mathfrak{ns}=
\mathfrak{ns}_{\overline 0}\oplus \mathfrak{ns}_{\overline 1}$. Consider
its Grassmannization $\mathfrak{ns}(\mathcal{A}):=
(\mathfrak{ns}_{\overline 0}\otimes {\mathcal A}_{\overline 0})
\oplus (\mathfrak{ns}_{\overline 1}\otimes {\mathcal A}_{\overline 1})$
and the corresponding supergroup $\mathrm{NS}(\mathcal {A})$.
We describe this group explicitly in different ways.
% We show that each irreducible
%unitarizable representation of $\mathfrak{ns}$ admits the integration to a
%unitary representation of the supergroup $\mathrm{NS}({\mathcal A})$. Under some inequalities for parameters, we show
%that  such representation can be realized by Gaussian operators in a super-Fock space.
 Next, we define a
complexification $\Gamma(\cA)$ of $\mathrm{NS}({\mathcal A})$. It is a semigroup,
whose elements are superannuli of dimension $1|1$ equipped with contact structures; the multiplication is gluing of such superannuli. Under some inequalities for parameters,
we
show that  unitary representations of $\mathrm{NS}({\mathcal A})$ admit extensions to representations
of $\Gamma(\mathcal {A})$. We
do not assume that the reader has prior knowledge of 
Lie superalgebras and supergroups.}

\vspace{22pt}

{\sf

\noindent
1. Introduction \hfill \pageref{s:introduction}

\noindent
2. Preliminaries. Berezin--Kats Grassmanizations
%\\
%\phantom{1. }
 and surplace groups \hfill
\pageref{s:surplace}

\noindent
3. Diffeomorphisms and contactomorphisms 
%\\
%\phantom{1. }
of the supercircle $S^{1|1}$ \hfill
\pageref{s:cont}

\noindent
4. Invariant differential operations and an embedding 
\\
\phantom{1. }
of 
$\cont(S^{1|1}_\bullet)$ to an orthosymplectic superalgebra \hfill
\pageref{s:embedding}

\noindent
5. Unitary  representations of the Neveu--Schwarz supergroup \hfill
\pageref{s:integration}

\noindent
6. Proofs of Theorems \ref{th:2}, \ref{th:3} \hfill
\pageref{s:proof}

\noindent
7. Affine relations and the Potapov transform\hfill
\pageref{s:affine-relations}

\noindent
8. The semigroup of contact superannuli\hfill
\pageref{s:annuli}

\noindent
9. Logarithmic densities on super-annuli\hfill
	\pageref{s:logarithmic}

\noindent
10. Super-Fock space and Gauss--Berezin integral operators\hfill
	\pageref{s:super-fock}
	
\noindent
11. Representations of the semigroup of super-annuli
\hfill \pageref{s:last}	

\noindent
12. Final remarks \hfill\pageref{s:final}}

\section{Introduction%
\label{s:introduction}}

\COUNTERS

{\bf \punct The Virasoro algebra.%
\label{ss:virasoro}} Consider the Lie algebra
$\vect^\R(S^1)$
of $C^\infty$-smooth vector fields on the circle $S^1$, they have the form 
$f(\phi)\frac\partial{\partial \phi}$,
where $\phi\in \R/2\pi\Z$, and $f(\phi)$ is a real-valued function. 
The  'Lie group' corresponding to $\vect^\R(S^1)$ is the group
 $\SDiff(S^1)$
of orientation preserving diffeomorphisms of the circle. 

Denote by $\vect(S^1)$ the complexification of the Lie algebra
 $\vect^\R(S^1)$, it consists
of smooth vector fields $f(\phi)\frac\partial{\partial \phi}$
with complex-valued $f(\phi)$.
For vector fields 
$$
L_\alpha:=ie^{i\alpha \phi} \frac\partial{\partial \phi},\qquad
\text{where $\alpha\in \Z$},
$$
we have $[L_\alpha,L_\beta]=(\alpha-\beta)L_{\alpha+\beta}$.

The Virasoro algebra $\vir$ is the Lie algebra
with basis $L_\alpha$, where $\alpha$ ranges in $\Z$,
 and $\zeta$; the relations are 
 \begin{align}
& [L_\alpha,L_\beta]=(\alpha-\beta)L_{\alpha+\beta}+
\tfrac 1{12}(\alpha^3-\alpha)\delta_{\alpha+\beta,0}\cdot \zeta,
\label{eq:vir1}
\\
&[L_n,\zeta]=0.
\label{eq:vir2}
 \end{align}
 Virasoro met these relations  in the following context.
Consider the space of functions depending of infinite
number of complex variables $z_1$, $z_2$, \dots\vphantom{.}
(bosonic Fock space, see below Subsect. \ref{ss:super-Fock}).
We define {\it bosonic creation and annihilation operators} by
\begin{equation}
T_k f(z)=\sqrt k\, z_k f(z), \quad T_{-k} f(z)=\sqrt k \frac\partial{\partial z_k} f(z),\quad \text{where $k>0$,}
\label{eq:boson-cran}
\end{equation}
and set also $T_0:=0$. For fixed $\mu$, $\nu\in \C$
we define the following operators
\begin{align}
L_\alpha&:=\frac12 \sum_{m,n:\,m+n=\alpha} T_m T_n
+(\mu+i\nu \alpha) T_n \quad\text{for $\alpha\ne 0$,}
\label{eq:L-alpha}
\\
L_0&:=  \sum_{m>0} T_m T_{-m}+\frac12(\mu^2+\nu^2).
\label{eq:L-0}
\end{align}
%where $:\dots:$ denotes the so-called normal ordering of products,
%we write the creation operators before annihilation 
%operators (this is important only for $L_0$, for $\alpha \ne 0$
%operators $T_m$ and $T_n$ in \eqref{eq:L-alpha} commute).
 Then we get a representation of the Virasoro
algebra, the central element $\zeta$ acts as the multiplication
by $c=1+12 \nu^2$. 

\sm

Next, 
consider  Grassmann variables
$\xi_{1/2}$, $\xi_{3/2}$,  $\xi_{5/2}$, \dots,
$$
\xi_r\xi_s=- \xi_s\xi_r, \qquad \xi_r^2=0,
$$
and consider the space of functions in these variables
  (the fermionic Fock space, see Subsect. \ref{ss:super-Fock}).
 We define {\it fermionic creation and annihilation operators}
\begin{equation}
	A_r f(\xi)=\xi_r f(\xi),\quad A_{-r}
	=\frac\partial{\partial \xi_r} f(\xi), \quad \text{where $r>0$.}
	\label{eq:f-c-a}
\end{equation}

The following operators
\begin{align}
L'_\alpha&:=
\frac 14 \sum_{r,s:\,r+s=\alpha} (r-s) A_r A_s
\label{eq:L-alpha'}
\\
L'_0&:= \frac12 \sum_{r>0} r A_r A_{-r}
\label{eq:L-0'}
\end{align}
satisfy the relations \eqref{eq:vir1}--\eqref{eq:vir2}.

\sm

{\bf \punct The Neveu--Schwarz and Ramond Lie superalgebras.%
\label{ss:NS}} The {\it Neveu--Schwarz algebra} \cite{NS} $\frn\frs$ is the Lie superalgebra 
$\ns=\ns_\0\oplus \ns_\1$  with even generators  $\zeta$, $L_\alpha\in \frn\frs_\0$,
where $\alpha\in \Z$,   odd generators $M_r$, where $r\in \frac12+\Z$, and relations
\begin{align}
&[L_\alpha,L_\beta]_s=(\alpha-\beta)L_{\alpha+\beta}+
\tfrac 1{12}(\alpha^3-\alpha)\delta_{\alpha+\beta,0}\cdot \zeta,
\label{eq:NS1}
\\
&[L_\alpha,M_r]_s=(\tfrac m2-r) M_{\alpha+r},
\label{eq:NS1.5}
\\
&[M_r,M_t]_s=2L_{r+t}+\tfrac 13 (r^2-\tfrac14)\delta_{r+t,0}\cdot \zeta,
\label{eq:NS2}
\\
&[L_\alpha,\zeta]_s=0,\qquad [M_r,\zeta]_s=0
\label{eq:NS3}
\end{align}
(so the  even part $\ns_\0$ is the Virasoro algebra).
 
This algebra has a twin, the {\it Ramond algebra} \cite{Ram} --- now
additional generators $M_r$ are enumerated by 
$r\in \Z$ and relations are given by the same formulas \eqref{eq:NS1}--\eqref{eq:NS3}.

\sm

 These superalgebras naturally arise
(Ramond \cite{Ram}, Neveu, Schwarz, \cite{NeeS}, 1971, see also \cite{IK1})
as you try to write mixed bosonic-fermionic expressions similar to generators of
the Virasoro algebra
\eqref{eq:L-alpha} and \eqref{eq:L-alpha'}.
 Namely, we  
consider the space of functions of infinite number of complex variables $z_1$, $z_2$, \dots, and Grassmann variables
$\xi_{1/2}$, $\xi_{3/2}$, \dots.
We fix complex parameters $\mu$, $\nu$ and
 define  generators 
\begin{align}
 L_\alpha&:= \frac 12 \sum_{i,j:\,i+j=\alpha} T_i T_j+ (\mu+i\nu n) T_n
 +\frac 14 \sum_{r,s:\,r+s=\alpha}(r-s) A_r A_s,
 \label{eq:fock-L}
 \\
 L_0&:=  \sum_{j>0} T_j T_{-j}+ \frac12(\mu^2+\nu^2)
+\frac 12 \sum_{r>0} r A_r A_{-r}, 
\label{eq:fock-L0}
 \\
M_r &:= \sum_{\alpha,s: \, \alpha+s=r} T_\alpha A_s+  (\mu+i\nu n)A_r . 
\label{eq:fock-M}
\end{align}
%where $: \dots:$ means a normal order of factors, we write creation operators before annihilation operators. This is important only for
%$L_0$, in this case the expression without $: \dots:$ diverges.
 Then the operators $L_\alpha$, $M_t$ determine a representation
of the Lie superalgebra $\ns$, the central element $\zeta$ acts as a multiplication by $3/2+12\nu^2$. The supercommutator in our case
is
\begin{align*}
[L_\alpha,L_\beta]_s:&=[L_\alpha,L_\beta]=L_\alpha L_\beta-L_\beta L_\alpha,
\quad [L_\alpha,M_r]_s:=[L_\alpha,M_r],
\\
[M_r,M_t]_s:&=\{M_r,M_t\}=M_r M_t+M_t M_r.
\end{align*}
The new expression for $L_\alpha$ is the sum of \eqref{eq:L-alpha} and \eqref{eq:L-alpha'}.

 The Ramond and Neveu--Schwarz Lie superalgebras are substantial objects of  representation theory;
 numerous  deep      statements 
 about the Virasoro algebra have counterparts for the Ramond and Neveu--Schwarz algebras
see Kac \cite{Kac},
Friedan, Qiu, Shenker \cite{Fri},
 Meurman, Rocha-Caridi \cite{MR-C}, Goddard, Olive, Kent \cite{GOK}, Kac, Wakimoto \cite{KW}, Iohara, Koga \cite{IK1}. 
 By some technical reasons
(the correspondence with the preliminary paper  \cite{Ner-super}),  below 
we prefer $\ns$.

Feigin and Leites \cite{FL} observed that these algebras (without center $\zeta$) are algebras of contact vector fields
on $(1|1)$-dimensional supermanifolds, whose underlying even space is the circle $S^1$. On further super-analogs 
of the Virasoro algebra, 
see \cite{KvL}, \cite{GLS}; on their representation, see, e.g., \cite{IK2}, \cite{Dob}.

\sm

{\bf\punct Purposes of the paper.} 
In this paper we follow  DeWitt's \cite{DeW}  approach
(see also Rogers \cite{Rog}) and
consider a supergroups as  usual groups over a Grassmann algebra $\cA$.
% we prefer to consider 
%the Grassmann algebra $\cA$ with infinite number of generators. 
I explain my choice
in Subsect. \ref{ss:formalities}.

We superize work \cite{Ner-holom}.
According \cite{Ner-semigroup}, \cite{Ner-holom}, \cite{Seg2}, 
\cite{Ner-book}, Sect. VII.4-5, any highest weight representation of the group $\SDiff(S^1)$ 
extends to a projective holomorphic representation of the following semigroup $\Gamma$. An element of $\Gamma$ is 
a one-dimensional complex manifold equivalent to an annulus with fixed parametrizations of boundaries. The multiplication is gluing
of annuli according the parametrizations of components of the boundary. We define a similar semigroup
of superannuli and construct their representations corresponding 
the representations \eqref{eq:fock-L}-\eqref{eq:fock-M}
 (under some restrictions for parameters of $\mu$, $\nu\in\R$).

\sm

Difficulties overcoming in the present paper are of functional-analytic
nature (definitions, self-adjointness, and multiplications of unbounded operators
in Hilbert spaces).
On the other hand, our main super-objects are the semigroup of 
contact super-annuli and the semigroup of Gauss--Berezin operators;
they are outside the standard formalism of super-mathematics.
For this reasons, I
do not assume that the reader has prior knowledge of 
Lie superalgebras and supergroups.
 I tried to write a  text admissible for
functional analysts and readers familiar with the representation theory of Lie groups
in the classical sense. We describe  supergroups in different ways  as explicitly as we can (even more than this is formally necessary).

\sm 

{\bf \punct Structure of the paper.} Let $\frg$ be a Lie superalgebra. According Berezin and Kats \cite{BK} (see also \cite{Ber-super})
for any supercommutative algebra%
\footnote{An associative $\Z_2$-graded algebra $\cA=\cA_\0\oplus \cA_\1$ is {\it supercommutative}, if elements of $\cA_\0$ are 
contained in the center, and elements $u$, $v\in \cA_\1$ anticommute, $uv=-vu$.} 
$\cA$, we can construct a certain Lie algebra  $\frg(\cA)$; quite often  this allows to construct the corresponding 
group%
\footnote{Berezin and Kats considered formal groups, but a pass to 'Lie groups' in finite-dimensional case and many infinite-dimensional cases is automatical, see below Sect. \ref{s:surplace}-\ref{s:cont}.}
(which is a kind of 'Lie group over $\cA$'). 

%There are different ways to formalize the notions of  supergroups and supermanifolds.
%For instance, definitions in  \cite{BK}, \cite{DeW}, \cite{Schw} are not equivalent, 
%see comparison in \cite{Schm}. Of course, these approaches are not antipodes, and usually translations are possible.
 
 %Our main object (the semigroup of contact annuli) is outside any 
 %existing formalism of supermanifolds and Lie supergroups.
 %On the other hand, this paper is mainly a functional-analytic work.
 %For these reasons, we
 %do not assume that the reader has prior knowledge of 
 %Lie superalgebras and supergroups and try to
 %minimize formalism.

 Basically, our approach  is similar to  DeWitt \cite{DeW}.
 We prefer to think that $\cA$
 is fixed and is the Grassmann algebra with a countable number of generators  $\fra_j$.
 In particular, we think that a supergroup is a Lie group other $\cA$.
 We explain Berezin--Kats Grassmanizations and descriptions	 of groups over $\cA$
 in Section \ref{s:surplace}.

In Section \ref{s:cont}, we discuss
Lie superalgebra $\cont(S^{1|1}_\bullet)$ of contact vector fields on $(1|1)$-circle
and the corresponding infinite-dimensional  supergroup $\SCont(S^{1|1}_\bullet;\cA)$
of contact super-diffeomorphisms, The notation $S^{1|1}$ in this sentence 
means that we consider the usual circle with the angle coordinate $\phi$
and with additional odd coordinate $\theta$, where $\theta^2=0$. Here there are two
variants for $S^{1|1}$, the trivial strip and the M\"obius strip; they correspond to the Ramond and 
Neveu--Schwarz superalgebras, we indicate by $\bullet$ the Neveu--Schwarz case. 
Our main purpose in Section \ref{s:cont} is to present an explicit description
of the supergroup $\SCont(S^{1|1}_\bullet;\cA)$ corresponding to $\cont(S^{1|1}_\bullet)$
(it is easy and apparently is not new), we present an additional description 
in Subsect. \ref{ss:addendum-cont}.

In Section \ref{s:embedding}, we start from   Kirillov's description
\cite{Kir} of the superagebra $\cont(S^{1|1}_\bullet)$ in terms of
invariant differential operations.
Our purpose is to obtain an  embedding of the supergroup $\SCont(S^{1|1}_\bullet;\cA)$
 to an infinite-dimensional orthosymplectic Lie supegroup
 $\OSp(\cW;\cA)$ in a certain space $\cW[\cA]$. More precisely,
 we consider a group $\AOSp(\cW,\cA)$ of affine orthosymplectic transformations 
 and construct 
  embeddings $\SCont(S^{1|1}_\bullet;\cA)\to \AOSp(\cW,\cA)$ depending on
 two  parameters (corresponding to $\mu$, $\nu$
  in formula \eqref{eq:fock-L}-\eqref{eq:fock-M}).
 This is a starting point for our main construction (Sections \ref{s:annuli}-
 \ref{s:last}).

 In Sections \ref{s:integration}-\ref{s:proof}, we show that  {\it 
any unitary highest weight representation of the Neveu--Schwarz Lie superalgebra $\ns$ can be integrated to a projective representation of the supergroup
$\SCont(S^{1]1}_\bullet;\cA)$}. This statement is known, see
Neeb, Salmasian \cite{NeeS}, we present its independent proof,
which is relatively straightforward%
% This part is a superization of the work by Goodman, Wallach \cite{GW}) 
% about representations of the group of diffeomorphisms of the circle%
\footnote{The analog of the Bott cocycle for the group 
	$\SCont(S^{1]1};\cA)$ was obtained by Radul \cite{Rad}. 
	It is natural to believe that 
 cocycles in highest weight projective representations of $\SCont(S^{1]1}_\bullet;\cA)$ are reduced to this expression and  a cocycle determining
 the universal covering group of $\SCont(S^{1]1}_\bullet;\cA)$. 
 But we do not establish such formal statement.}. We recall that for infinite-dimensional representations of supergroups 
 operators of representations are not bounded in Hilbert spaces, and this requires a care for manipulations with unbounded operators. 
 
Section \ref{s:affine-relations} contains preliminaries
about super-linear and super-affine relations.
 Linear relations  are kind of \vphantom{.}`linear operators',
  their domains 
 can be proper subspaces and  these `operators' can be multi-valued.

In Section \ref{s:annuli} we construct the semigroup complexification
$\Gamma_\bullet(\cA)$
 of $\SCont(S^{1]1}_\bullet;\cA)$.
%{\it We 
%construct the holomorphic continuation
%for the representation of the Neveu--Schwarz algebra determined by the formulas  \eqref{eq:fock-L}--\eqref{eq:fock-L}}. An element of the 
Its elements are annuli equipped with an additional odd coordinate $\theta$,
a contact structure, and two  $\cA$-valued contactomorphisms 
from $S^{1|1}$ to boundaries of the annuli. A multiplication is gluing of annuli along components of boundaries according 
contactomorphisms. 

In Section \ref{s:logarithmic} we realize the semigroup $\Gamma_\bullet(\cA)$
as a semigroup of Lagrangian super-affine relations in the orthosymplectic space $\cW[\cA]$ mentioned above.

In Section \ref{s:super-fock} we define super-Fock space and Berezin symbols
of operators. Next, we define Gauss--Berezin integral operators,
which are super-hybrids of Gaussian operators in the boson Fock
space and  Berezin operators in the fermion Fock space. For each Gauss--Berezin
operator we assign an affine Lagrangian super-linear relation. Heuristically,
a product of Gauss--Berezin operators corresponds to a product
of Lagrangian relations. But formally we meet  difficulties related to
possible unboundedness of Gauss--Berezin operators and 
the operation of product of superaffine relations, which generally is not good (see \cite{Ner-AA}, Lemma 2.2, \cite{Ner-book},  Sect. IV.3). 
 
In Section \ref{s:last}, {\it we construct representations
of the semigroup $\Gamma_\bullet(S^{1|1})$ corresponding to
 the representations \eqref{eq:fock-L}-\eqref{eq:fock-M}
of the Neveu--Schwarz algebra}. By the construction of
Section \ref{s:logarithmic} we have an embedding of $\Gamma_\bullet(\cA)$
to a semigroup of Lagrangian super-affine relations. For such relations
we construct bounded Gauss--Berezin operators in the space of smooth vectors
of the operator $L_0$ defined by \eqref{eq:fock-L0}. In this case, product of operators
really corresponds to product of relations.

\section{Preliminaries. Berezin--Kats Grassmanizations and surplace groups%
\label{s:surplace}}

\COUNTERS

Here we discuss Grassmanizations of Lie superalgebras  in the sense
of Berezin and G.~Kats \cite{BK} and the corresponding
groups.

\sm

{\bf \punct The phantom algebra $\boldsymbol\cA$.} 
Denote by $\cA=\cA(\K)$
(where $\K=\R$ or $\C$)
 the associative algebra  with unit generated by
 a  countable family of elements $\fra_1$, $\fra_2$, \dots
 \vphantom{.} 
 satisfying the  relations
$$
\fra_k\fra_l=-\fra_l\fra_k, \qquad \fra_k^2=0.
$$
We call elements of $\cA$ {\it phantom constants}.
For collection $J=\{j_1,\dots,j_{m}\}$ of pairwise different
numbers we denote
$$
\fra^J:= \fra_{j_1}\dots \fra_{j_m},
$$
We also denote
$$
\dot\fra^J=\fra_{j_1}\dots \fra_{j_m}, \qquad \text{if $j_1<\dots<j_m$, $m\ge 0$.}
$$
Such products form a basis of $\cA$.

 A {\it monomial} is an element
of the form 
$s\cdot \fra^J$
where $s$ is a scalar and all
$\fra_{j_l}$ are pairwise distinct. We say that $m$ is the {\it degree} of this monomial. 

%We say that a monomial is {\it monic} if its has the form $\fra^I$,
%where $i_1<i_2<\dots$. Monic monomial form a basis in $\cA$.

%Denote the set of all monomials by $\cM=\cM(\cA)$ and the set of all monic
% monomials by $\cM^\mon=\cM^\mon(\cA)$.  

 Denote by:
 
 \sm
 
 --- $\cA^m$ the linear span of monomials of degree $m$;
 
 \sm

---  $\cA_\0$, $\cA_\1$ the linear span of monomials of even (respectively,
odd) degrees;

\sm

--- $\cA_+$ the linear span of all monomials of positive degree;

\sm

--- $\cM$, $\cM_\0$, $\cM_\1$  the set of all monomials, the set of all monomials of even degree,  of odd degree;

\sm

--- $\cM_{\0+}$ the set of all monomials of even positive degree.

\sm

We say that an element $\phi\in \cA$ is {\it pure}
if $\phi$ is contained in $\cA_\0$ or in $\cA_\1$.
For pure elements
$\phi$, $\psi$ of $\cA$, we have
\begin{align*}
\phi \in \cA_\0,\, \psi \in \cA
\quad
\Rightarrow \quad
\phi\psi=\psi\phi;
\\
\phi,\psi \in \cA_\1
\quad \Rightarrow \quad
\phi\psi=-\psi\phi.
\end{align*}

The {\it body map}
 (see \cite{DeW}) $\pi_\downarrow:\cA\to \K$ 
 is the homomorphism defined by
$$
\pi_\downarrow \Bigl(\sum_{I:i_1<i_2\dots} c_I\fra^I \Bigr):=c_\varnothing.
$$
Clearly, an element $\omega\in \cA$ is invertible if and only if
$\pi_\downarrow(\omega)\ne 0$. Also, $\ker \pi_\downarrow=\cA_+$.

We define  an involutive automorphism
$\omega\mapsto \omega^\circ$ of $\cA$, whose values
 on the generators are $\fra_k^\circ= -\fra_k$.
Equivalently, for $\omega\in \cA_{\ov j}$ we have
\begin{equation}
\omega^\circ=(-1)^{\ov j}\omega
\label{eqLomega-circ}
\end{equation}

{\bf\punct Superlinear spaces and their Grassmannizations.%
\label{ss:superspaces}}
We say that a {\it superlinear space} $V=V_{\0}\oplus V_{\1}$
is a direct sum of two linear spaces $V_{\0}$, $V_{\1}$. We say that 
$v\in V_{\0}\oplus V_{\1}$ is {\it pure} if $v\in V_\0$ or $v\in V_\1$.
For a pure element $v\in V_{\ov j}$, we define its {\it parity} by
$p(v)=\ov j$ (in particular $p(0)$ is both $\0$ and $\ov 1$).

\begin{convention} {\rm a)} The symbols $\0$ and $\1$ are added and multiplied 
as elements of the field with two elements.

\sm

{\rm b)} If a formula contains an element $v$ of a superspace {\rm(}or 
a superalgebra{\rm)} and its parity $p(v)$,
then all elements $v$, $w$, etc. of the superspace in this formula are pure.
\end{convention}

{\sc Remark.} We keep in mind different types of linear spaces 
$V_\0$, $V_\1$:

\sm

--- finite-dimensional linear spaces;

\sm

--- countable-dimensional linear spaces;

\sm

--- Fr\'echet spaces, i.e., complete metrizable locally convex topological
vector spaces (we need some  explicit examples
and avoid general considerations).
\hfill $\boxtimes$

\sm 

For a superspace $V$, we define its {\it Grassmannization}:
\begin{align}
V[\cA]&:=V\otimes \cA; 
\nonumber\\
 V_\0[\cA]&:=\bigl(V_\0\otimes \cA_\0\bigr)\oplus \bigl(V_\1\otimes \cA_\1\bigr),
 \qquad
  V_\1[\cA]:=\bigl(V_\0\otimes \cA_\1\bigr)\oplus \bigl(V_\1\otimes \cA_\0\bigr).
  \label{eq:VVVVa}
\end{align}

To avoid an ambiguity,
we define tensor products 
$$
V\otimes \cA:=\bigoplus_{i_1<\dots<i_m} W\fra^I
$$
as  sets of all formal finite linear combinations $h=\sum_I  \otimes \dot\fra^I v_I$.  
 We say that a sequence $h_j$ converges to $h$ 
 if all $h_j$ and $h$ are contained in some 
 finite direct sum $V \dot\fra^{J_1}\oplus \dots \oplus V \dot\fra^{J_k}$ and the sequence convergences  in this direct sum. 
 
 We consider $V[\cA]$ as a left $\cA$-module or as a right $\cA$-module 
 (and do not think that it is a bimodule).

\sm

{\bf\punct Lie superagebras.%
\label{ss:lie-superalgebras}}
A {\it Lie superalgebra} $\frg=\frg_\0\,\oplus\, \frg_\1$ is a super-linear space
 equipped with a bilinear operation (a {\it bracket}) $[\cdot,\cdot]_s: \frg\times \frg \to \frg$
 such that for any pure elements $x$, $y$, $z$ we have
 \begin{align*}
& [x,y]_s\in \frg_{p(x)+p(y)}\quad \text{($\Z_2$-grading)};
 \\
& [y,x]_s=-(-1)^{p(x)p(y)}[x,y]_s\quad \text{(super-anticommutativity)};
 \\
&(-1)^{p(x)p(z)}[x,[y,z]_s]_s+ (-1)^{p(y)p(z)}[y,[z,x]_s]_s+ (-1)^{p(z)p(y)}[z,[x,y]_s]_s=0
\\
& \qquad \qquad \qquad \qquad \qquad\qquad \qquad \qquad \qquad \qquad
\text{(super-Jacobi identity).} 
 \end{align*}
 
 If $\frg_\0$, $\frg_\1$ are Fr\'echet spaces, then the bracket 
 must be continuous.
 
 \sm
 
 In particular, 
 the even part $\frg_\0$ is a Lie algebra,
and $\frg_\1$
 is a module over $\frg_\0$. 
 
 \sm

{\sc Examples.} 
a) If $\frg_\1=0$, then we get the usual definition of a Lie algebra.
 
\sm 
 
b) The {\it general Lie superalgebra $\frg\frl(p|q)$}.
 Consider a superlinear space $\C^p\oplus \C^q$,
 we consider its elements as vectors-rows. 
 Assume that a parity of block operators of the form
$\begin{pmatrix}A&0\\0&D\end{pmatrix}$ is $\0$, and a parity of operators
$\begin{pmatrix}0&B\\C&0 \end{pmatrix}$ is $\1$. We define a bilinear bracket
on matrices of order $p+q$
assuming that for pure $X$, $Y$
$$
[X,Y]:=X Y-(-1)^{p(X)p(Y)}YX.
$$
The even part $\frg\frl(p|q)_{\ov 0}$ is $\frg\frl(p)\oplus \frg\frl(q)$.

\sm

c) Consider the block complex $(m+m)$-matrix 
$J_m:=\begin{pmatrix}0&1\\-1&0\end{pmatrix}$
and the block $(n+n)$-matrix $I_n:=\begin{pmatrix}0&1\\1&0\end{pmatrix}$.
The {\it orthosymplectic Lie superalgebra} $\osp(2m|2n)$
 is the subalgebra
in $\frg\frl(2m|2n,\C)$ consisting of matrices 
$\begin{pmatrix}A&B\\C&D\end{pmatrix}$ satisfying the condition
$$
A J_m+J_m A^t=0,\quad D I_n+I_n D^t=0,\quad B^t J_m-I_n C=0.
$$
The even part $\osp(2m|2n,\C)_{\ov 0}$ is a direct sum 
of the symplectic Lie algebra $\frs\frp(2m,\C)$ and
the orthogonal Lie algebra $\fro(2n,\C)$.

More generally, we can consider an arbitrary nodegenerate
skew symmetric matrix $J$ of size $2m$ and nodegenerate
symmetric matrix $I$ of size $l$. If $l=2n$ is even, then we get a
 Lie superalgebra isomorphic to $\frg\frl(2m|2n,\C)$. If $l=2n+1$,
 then again
 all such algebras are isomorphic and we get the algebra
 $\frg\frl(2m|2n+1,\C)$
 in the standard notation.
\hfill $\boxtimes$

\sm

{\bf \punct Grassmannization of Lie superalgebras.} Let $\frg=\frg_\0\oplus\frg_\1$
be a  Lie superalgebra. Consider the space 
\begin{equation*}
\frg(\cA):= (\frg\otimes\cA)_\0:=
\Bigl(\frg_\0\otimes \cA_\0\Bigr)\oplus\Bigl( \frg_\1\otimes \cA_\1\Bigr)
:=\Bigl(\bigoplus\limits_{J:\,p(J)=\0} \dot\fra^J \frg_\0\Bigr)
 \bigoplus
\Bigl(\bigoplus\limits_{J:\,p(J)=\1} \dot\fra^I \frg_\1\Bigr).
\end{equation*}
We define a
bilinear bracket $[\cdot,\cdot]:\frg(\cA)\times \frg(\cA)\to \frg(\cA)$ in the following way.
For pure $X$, $Y\in \frg$ and collections $J$, $I$ such that
$p(X)=p(I)$, $p(Y)=p(J)$ (so $\fra^I X$, $\fra^J Y\in \frg(\cA)$)
 we set
\begin{equation}
[\fra^J X, \fra^I Y]=\fra^I \fra^J [X,Y]_s.
\label{eq:grass}
\end{equation}
Clearly, we get a structure of a Lie algebra (not a superalgebra) 
on $\frg(\cA)$. It is called the {\it Grassmannization of $\frg$}.

\sm

{\bf\punct Remarks.}
1) We  emphasize that  $\frg(\cA)$ is not simply a Lie algebra over $\C$ or $\R$,
the bracket is superbilinear in the following sense:
 for $X\in\frg_{\ov i}$, $Y\in \frg_{\ov j}$ and 
$\omega\in\cA_{\ov i}$, $\theta\in \cA_{\ov j}$
we have
$$
[\omega X, \theta Y]=\omega\theta\,[X,Y],
$$
and we consider representations of such algebras in modules
over $\cA$ (otherwise, $\frg(\cA)$ is a tedious example of
a non-semisimple complex  Lie algebra).

 2) Apparently, structures of Lie superalgebras firstly appeared in the
  paper
 Milnor, Moore \cite{MM}, 1965. 
 
 \sm

3) Grassmannizations were defined by Berezin and
Kats \cite{BK}, see also \cite{Ber-super}, Sect 5.1.
%, they also were considered by DeWitt \cite{DeW}. 

\sm

4)
Specialists in super-mathematics have different points of view to
the supergroups, see \cite{BerL} (and \cite{Ber-super}, \cite{Lei}),
\cite{DeW} (and \cite{Rog}), \cite{Schw}, \cite{VV}, \cite{Mol} and a discussion in \cite{Schm}
and \cite{Rog} (Introduction).
Our approach is similar to DeWitt \cite{DeW}, I explain this choice
in Subsect. \ref{ss:formalities}.

\sm

5) We emphasize that `phantom constants' $\fra_j$ are not `fermions', 
see Sect. \ref{s:super-fock}, and are not odd variables, see Sect. \ref{s:cont}.

\sm

{\bf\punct Example. The Grassmanization of the algebra
$\frg\frl(p|q)$ and the corresponding
supergroup.%
\label{ss:gl-group}}
 The algebra $\frg\frl(p|q;\cA)$ is the Lie algebra of block
matrices $\begin{pmatrix}a& b\\ c&d\end{pmatrix}$, where
matrix elements of $a$, $a$ are contained in $\cA_\0$, and
matrix elements of $b$, $c$ are contained in $\cA_\1$.
The Lie bracket is the usual commutator of matrices.

The corresponding group 
$\GL(p|q;\cA)$ is the group
of invertible block matrices $g=\begin{pmatrix}A&B\\C&D\end{pmatrix}$
of size $p+q$, such that blocks $A$, $D$ are composed of elements 
of $\cA_0$, and $B$, $C$ of elements of $\cA_{\ov 1}$.
The invertibility of such matrix is equivalent to
the invertibility of blocks $A$, $D$ (see, e.g., \cite{Ber-super}, Sect. 3.1, \cite{Rog}, Subsect. 3.4). The matrices 
$A$, $D$ are invertible iff
the matrices  
$A_\downarrow$, $D_\downarrow$ over $\K$ are invertible.

In this section, starting Subsect. \ref{ss:envelope}
we discuss the general construction of groups corresponding to algebras $\frg(\cA)$.
Let us explain this correspondence in our case.

Denote by $\GL_\vel(p|q;\cA)\subset \GL(p|q;\cA)$
 the subgroup of matrices $g$
satisfying $A_\downarrow=1$, $D_\downarrow=1$.
Denote by  $\frg\frl_\vel(p|q;\cA)\subset\frg\frl_\vel(p|q;\cA)$
the subalgebra consisting of matrices $X$ such that $X_\downarrow=0$.
For elements of $\frg\frl_\vel(p|q;\cA)$ we have well-defined
exponents, $\exp X= \sum X^n/n!$. For elements $1+U\in \GL_\vel(p|q;\cA)$
we have well-defined logarithms, $\ln(1+U)=\sum (-1)^{n+1} X^n/n$
(actually, in the both cases sums are finite). Clearly,
these maps
$$
\exp:\,\frg\frl_\vel(p|q;\cA)\to \GL_\vel(p|q;\cA),\quad 
\ln:\, \GL_\vel(p|q;\cA)\to \frg\frl_\vel(p|q;\cA) 
$$
are inverse one to another.

\sm

Recall that the {\it Berezinian} (see \cite{Ber-super}, Sect. 3.1, \cite{Rog}
 (Sect. 3.4)) of a matrix $g\in \GL(p|q;\cA)$
is
$$
\Ber\begin{pmatrix}A&B\\C&D
\end{pmatrix}=
\det (A)\cdot \bigl(\det(D-CA^{-1}B)\bigr)^{-1}\,\in\cA_\0;
$$
the matrices $A$ and $D-CA^{-1}B$ are composed 
of elements of a commutative algebra $\cA_\0$,
hence their determinants are well defined%
\footnote{Recall that the usual formula for determinants 
of block matrices is $\det \begin{pmatrix}a&b\\c&d
\end{pmatrix}=
\det (a)\cdot \bigl(\det(d-ca^{-1}b)\bigr)$.}.
The Berezinian is a homomorphism $\GL(p|q;\cA)\to\cA$,
$$
\Ber(g_1 g_2)=\Ber(g_1)\,\Ber(g_2).
$$

\sm

For an operator $g=\begin{pmatrix}A&B\\C&D
\end{pmatrix}$  in the space of rows $\cA^p\oplus \cA^q$
we define the {\it supertranspose} by the following formula
$$
g^{st}:= 
\begin{cases}
\begin{pmatrix}A^t&C^t\\-B^t&D^t\end{pmatrix},\quad \text{if $p(g)=\0$};
		\\
\begin{pmatrix}A^t&-C^t\\B^t&D^t\end{pmatrix},\quad \text{if $p(g)=\1$}	
\end{cases}.
$$
We also set 
$$(g_1+g_2)^{st}=g_1^{st}+g_2^{st}.$$
This operation satisfies the condition
$$
(g_1 g_2)^{st}=(-1)^{p(g_1)p(g_2)} g_2^{st} g_1^{st}.
$$
In particular supertranspose determines 
an antihomomorphism 
from the group 
$\GL_\vel(p|q;\cA)$ to itself%
\footnote{The identity $(AB)^t=B^t A^t$ is valid for matrices over
a commutative rings $R$. If $R$ is not commutative but it has an
antiinvolution $r\mapsto \ov r$ (i.e., $\ov {r_1+r_2}=\ov r_1+\ov r_2$,
$\ov{\ov r}=r$, $\ov {r_1 r_2}=\ov r_1 \ov r_2$, as for quaternions),
then we have adjoint matrices and $(AB)^*=B^* A^*$. The supertransposition differs from these operations.} 

Notice, that $(g^{st})^{st}\ne g$.

\sm

{\bf\punct Example. The orthosymplectic supergroup.%
\label{ss:osp-group}}
    The Lie superalgebra $\osp(2m|2n,\cA)$ (see Subs. \ref{ss:lie-superalgebras})
is the subalgebra in $\frg\frl(2m|2n,\cA)$ consisting  of matrices $Q$ satisfying the condition
$$
Q\begin{pmatrix} J&0\\ 0&I\end{pmatrix}+
\begin{pmatrix} J&0\\ 0&I\end{pmatrix}Q^{st}=0.
$$    
The corresponding group $\OSp(2p|2q;\cA)$
 (see, e.g.,  \cite{Rog}. Sect 9.2) is the subgroup
in $\GL(p|q;\cA)$ consisting of matrices $g$ satisfying the 
identity  
$$
g\begin{pmatrix} J&0\\ 0&I\end{pmatrix}g^{st}=\begin{pmatrix} J&0\\ 0&I\end{pmatrix}.
$$
We can say this in another way. Consider the form
$\cJ$ in $\cA^{2p}\oplus \cA^{2q}$
defined by
$$
\cJ(v,w)=\begin{pmatrix} v_\0&v_\1\end{pmatrix}
\begin{pmatrix} J&0\\ 0&I\end{pmatrix}
\begin{pmatrix} v_\0^{st}\\v_\1^{st}\end{pmatrix},
$$
here $v_\0^{st}=v_\0^t$, $v_\0^{st}=(v_\1^\circ)^t$.
Then the group $\OSp(2p|2q;\cA)$ consist of elements 
$g\in \GL(p|q;\cA)$ satisfying 
$$
\cJ(v,w)=\cJ(vg,wg).
$$

{\bf\punct Example: the group of affine transformations.}
Consider the group  
$\Aff(p|q;\cA)\subset \GL(p+1|q;\cA)$
consisting of matrices
$$
\begin{pmatrix}
	A&0&B\\
	\phi&1&\psi\\
	C&0&D
\end{pmatrix}.
$$
This subgroup is the stabilizer of the vector
$$
(0,\dots,0,1\bigl|0,\dots,0).
$$ 
It induces affine transformations
$$
\begin{pmatrix}\xi&\eta \end{pmatrix}
\mapsto \begin{pmatrix}\xi&\eta \end{pmatrix}
\begin{pmatrix}a&b\\c&d\end{pmatrix}+
\begin{pmatrix}\phi&\psi\end{pmatrix}
$$

Next,  consider an infinite-dimensional Fr\'echet space
$V=V_\0\oplus V_\1$ and an affine action  a group
$G$ on $V[\cA]$, 
$$
v\mapsto v\rho(g)+\gamma(g).
$$
Notice that
in this case $\rho(g)$ is a linear action and 
 $\gamma(g)$ satisfies to the {\it cocycle condition}
\begin{equation}
\gamma(g_1g_2)=\gamma(g_1)\rho(g_2)+\gamma(g_2).
\label{eq:cocycle-identity}
\end{equation}

For a linear action $\rho(g)$ we can
can shift an origin to a point $r\in V_\0[\cA]$ and get action
\begin{equation}
v\mapsto v\rho(g)+r\rho(g)-r.
\end{equation}
We get an affine action, in particular,
\begin{equation}
\gamma(g):=r\rho(g)-r
\label{eq:trivial-cocycle}
\end{equation}
satisfies the cocycle condition.

\sm

\sm

{\bf\punct Universal enveloping algebras.%
\label{ss:envelope}} For a Lie superalgebra $\frg$ we define 
the {\it universal enveloping algebra} $\cU(\frg)$ as
the associative algebra with generators $i_X$, where
$X$ ranges in pure elements of $\frg$, and relations
\begin{align*}
&i_{s_1 X_1+s_2 X_2}=s_1i_{ X_1}+ s_2i_{ X_2},\quad \text{where 
$p(X_1)=p(X_2)$, $s_1$, $s_2\in \K$.}
\\
&i_X i_Y-(-1)^{p(X)p(Y)}i_Y i_X=i_{[X,Y]_s}.
\end{align*}

Below we write $X$ instead of $i_X$ and identify $\frg$ with a subspace
in $\cU(\frg)$. Recall (see, e.g. \cite{Mus}, \S6.1) that for a finite-dimensional or countable-dimensional
$\frg$  the {\it following '\Poincare--Birkhoff--Witt'
theorem holds}:

\sm

{\it Let $X_1$, $X_2$, \dots be a basis in $\frg_\0$, $Y_1$,  $Y_2$, \dots
be a basis in $\frg_\1$. Then the following elements 
form a basis in $\cU(\frg)$:
\begin{equation}
X_1^{k_1} X_2^{k_2}\dots Y_1^{\epsilon_1}Y_2^{\epsilon_2}\dots,
\label{eq:PBW}
\end{equation}
where
$k_i\ge 0$, $\epsilon_j=0,1$, and all but a finite number of $k_i$
and $\epsilon_j$ equal 0.}

\sm

{\sc Remark.} 
a) In supercase the analog of symmetric basis in the universal enveloping 
algebra also exists, see, e.g., \cite{Mus}, Sect. 6.4.

\sm

b) Formally, '\Poincare--Birkhoff--Witt theorem' is valid for superalgebras, whose dimension
has an arbitrary cardinality,
but it is reasonable only for algebras
having a finite or countable basis. Otherwise, we come to Hamel bases.

\sm

c) May be, for Fr\'echet spaces a definition of enveloping algebras
must be modified. A right definition is not clear, and the benefit of 
such object  is problematic. In any case, our considerations below include only finite sums of products $\sigma\cdot Z_1\dots Z_n$, where $Z_j\in \frg$, 
$\sigma\in\cA$ and do not depend on possible topologizations.
\hfill $\boxtimes$

\sm

Next, set
$$
\cU[\frg,\cA]:=\cU(\frg)\otimes \cA,
$$
its elements are finite sums $\sum \omega_j u_j$, where $u_j\in \cU(\frg)$,
$\omega_j\in\cA$, 
we assume that elements $\omega\in\cA$ and $u\in\cU(\frg)$ commute, $\omega u=u\omega$
We have the tautological embedding  $\frg(\cA)\to \cU[\frg;\cA]_\0$, it sends a commutator to
a commutator.

\sm

{\sc Remark.}
The $\cU[\frg,\cA]$ is not $\cU(\frg(\cA))$.
\hfill $\boxtimes$

\sm

 {\bf \punct Surplace groups $\boldsymbol{G_\vel(\cA)}$.}  
  Consider the Lie algebra
 $$
 \frg_\vel(\cA):=\Bigl(\frg_\0\otimes \cA_{\ov 0+}\Bigr)
 \oplus \Bigl(\frg_\1\otimes \cA_\1\Bigr)\subset  \frg(\cA).
 $$
 Clearly, $\frg_\vel(\cA)$ is an ideal in $\frg(\cA)$ 
and
$$\frg(\cA)/\frg_\vel(\cA)=\frg_{\ov 0},$$ 
  $\frg(\cA)$ is a semidirect product
 $\frg_\0\ltimes \frg_\vel(\cA)$.
 
We write elements of $\frg_\vel(\cA)$ as  finite sums
\begin{equation}
u=\sum\lambda_i X_i+\sum_j\mu_j Y_j,
\label{eq:u}
\end{equation}
 where $X_i\in \frg_\0$, $Y_j\in \frg_\1$ and
monomials $\lambda_i\in \cM_{\0+}$, $\mu_j\in \cM_{\1}$ are pairwise
 non-proportional.

\begin{observation}
Any finitely generated Lie subalgebra in $\frg_\vel(\cA)$ is
nilpotent.
\end{observation}

For any $u\in \frg_\vel(\cA)$ we define it exponential
$$
\exp(u):=\sum_{N=0}^\infty \frac 1{k!} u^k \in \cU[\frg;\cA]_\0,
$$
actually, this sum is finite and has parity $\0$.

\begin{theorem}
\label{th:1}
{\rm a)} The map $\exp: \frg(\cA)\to  \cU[\frg;\cA]_\0$ is injective.
  Its image  $G_\vel(\cA)$
  is a group, say the {\rm surplace group},  with respect to the multiplication
 in $\cU[\frg;\cA]$.

 \sm

{\rm b)} The group $G_\vel(\cA)$ is generated by all elements of the form
\begin{align}
&1+\lambda X,\qquad \text{where $X\in \frg_\0$, $\lambda\in \cM_{\0+}$;}
\label{eq:lambda-X}
\\
&1+\mu Y,\qquad \text{where $Y\in \frg_\1$, $\lambda\in \cM_{\1}$.}
\label{eq:mu-Y}
\end{align}
%Moreover, we can write elements of the type \eqref{eq:lambda-X},
%i.e., any element of $G_+[\cA]$, in the form
%$$
%\prod_p(1+\lambda_p X_p)\cdot \prod_q (1+\mu_q Y_q).
%$$

 {\rm c)} Each product of elements \eqref{eq:lambda-X} has the form
 $\exp(w)$, where $w\in \frg_{\ov 0}(\cA)$.
 The set $\exp\bigl(\frg_\0(\cA)\bigr)$ is a subgroup in $G_\vel(\cA)$.
 
 \sm

{\rm d)} Any element of the group $G_\vel(\cA)$
admits a unique representation in  the form
\begin{equation}
\exp\Bigl(\sum \nu_k X_k\Bigr) \cdot \exp\Bigl(\sum \kappa_m Y_m\Bigr),
\label{eq:exp-exp}
\end{equation}
where $X_k\in \frg_\0$, $\nu_k\in \cM_{\0+}$, $Y_m\in \frg_\1$, $\kappa_m\in\cM_\1$.
\end{theorem}

{\sc Remark.} We use the French sportive term
 '{\it surplace}' (stand track, standstill). Informally, our groups $G_\vel(\cA)$ are groups of actual infinitesimal
transformations, see the next section.
\hfill $\boxtimes$

 \sm

We  formulate some additions to this theorem.

First, for pure elements  $Z_1$, $Z_2\in\frg$, and  $\nu_1$,
 $\nu_2\in\cM^+$ such that
 $p(\nu_j)=p(Z_j)$, we have
\begin{align}
(1+\nu_1 Z_1)(1+\nu_2 Z_2)&=(1+\nu_2 Z_2)(1+\nu_1 Z_1)\,\,
(1+\nu_1\nu_2[Z_1,Z_2]_s)=
\label{eq:reorder-1}
\\
&=(1+\nu_1\nu_2[Z_1,Z_2]_s)\,\, (1+\nu_2 Z_2)(1+\nu_1 Z_1).
\label{eq:reorder-2}
\end{align}
Since $\nu^2=0$, we have
\begin{equation}
(1+\nu Z_1)(1+\nu Z_2)=1+\nu(Z_1+Z_2).
\label {eq:relation-2}
\end{equation}

\begin{proposition}
\label{pr:1}
	{\rm a)} The group with generators \eqref{eq:lambda-X}--\eqref{eq:mu-Y} and relations
	\eqref{eq:reorder-1}--\eqref{eq:relation-2}
	 is $G_\vel(\cA)$.
	
	\sm
	
	{\rm b)} Fix arbitrary linear order $\prec$ on the set of monomials $\dot\fra^I$.
	Then any element of $G_\vel(\cA)$ admits a unique representation
	\begin{equation}
	(1+s_1 \dot \fra^{I_1} Z_1)(1+s_2 \dot \fra^{I_2}  Z_2)\dots (1+s_N \dot \fra^{I_N}  Z_N),
	\label{eq:canonical-decomposition-vel}
	\end{equation}
	such that $s_j Z_j\ne 0$ and $\dot\fra^{I_k}\prec \dot\fra^{I_l}$ for $k<l$,
	$p(Z_m)=p(\fra^{I_m})$.
\end{proposition}

{\bf\punct Proof of Theorem \ref{th:1}.}
 First, recall the {\it Cambell--Hausdorff theorem} (see, e.g., \cite{Ser},
  \cite{Reu}).
   Consider the free Lie algebra 
 $\free_2$ with two generators $P$, $Q$. Its universal enveloping algebra
 $\cU(\free_2)$
 is the free associative algebra $\Assoc_2$ with the same generators.
Denote by  $\ov\Assoc_2$ its completion consisting of formal associative series 
in $P$, $Q$.
  Then 
 there exist elements $\gamma_{m,n}\in\free_2$ of degree $m$ in $P$
  and $n$ in $Q$ such that the following identity in  $\ov\Assoc_2$
  holds:
 $$
 \exp(P)\,\exp(Q)=\exp\Bigl(P+Q+\sum_{m\ge 1, n\ge 1} \gamma_{m,n}(P,Q)\Bigr).
 $$
 
{\sc Statement a).} Let $u$, $v\in \frg_\vel(\cA)$.
 We   substitute $P=u$, $Q=v$ to the Cambell--Hausdorff formula.
 The sum 
 $$w=P+Q+\sum_{m\ge 1, n\ge 1} \gamma_{m,n}(u,v)$$
 is finite and is an element of $\frg(\cA)$,
  and therefore $\exp(u)\exp(v)=\exp(w)$, i.e., $G_\vel(\cA)$
  is closed with respect to multiplications.
  
  \sm
  
 {\sc Statement b).} We have
 $$
 1+\lambda X=\exp(\lambda X)\in G_\vel(\cA),
 \qquad
 1+\mu Y=\exp(\mu Y)\in G_\vel(\cA).
 $$ 
 
  So it is sufficient to prove that any $\exp(u)$ can be decomposed 
  as a product of elements of the form \eqref{eq:lambda-X}--\eqref{eq:mu-Y}.
  Decompose $u$ into a sum $u_1+u_2+\dots$, where $u_k$ is homogeneous
  in $\fra_1$, $\fra_2$, \dots of degree $k$. 
   Let
  $u_r$ be the  first nonzero term,
  $$
  u_r=\sum_{I:\# I=r} c_{I} \dot\fra^I Z_I,
  $$
  where $Z_I$ are pure elements of $\frg$.
  Let $c_J\ne 0$. Then
  $$
 (1- c_J\, \dot\fra^J Z_J) \exp\Bigl(\sum_{I}  c_{I} \,\dot \fra^I Z_I +u_{r+1}+\dots\Bigr)
  $$
  has the form
  $$
  1+\sum_{I:I\ne J}  c_{I}\, \dot\fra^I Z_I
  +\Bigl\{\text{terms of degree $\ge r+1$} \Bigr\}.
  $$
  By a), this expression has the form $\exp(u')$, and $u'$ has 
  the form
  $$
  u'=\sum_{I:I\ne J}  c_{I}\, \dot\fra^I Z_I+u'_{r+1}+u'_{r+2}\dots
  $$
  where  $u'_l$ has degree $l$. Repeating the same argument, we come to
  $$
  \exp(u)=\prod_I (1+c_I\dot\fra^I Z_I)\cdot \exp\bigl(\wt u_{r+1}+\wt u_{r+2}+\dots\bigr),
  $$
where $\wt u_l$ are homogeneous of degree $l$. 
 
We repeat the same argument to $\wt u_{r+1}$, etc., and come to a desired 
decomposition.

\sm

{\sc Statement  c).} We apply the statement a) to the Lie algebra $\frg_\0$.

\sm

{\sc Statement d).} {\it Existence.}
Consider an element 
$$
R=\exp\Bigl(\sum_{\lambda_j\in \cM_{\0+}:\, |\deg \lambda_j|\ge 2N} \lambda_j
 X_j+\sum \mu_k Y_k \Bigr)
\in G_\vel(\cA);
$$
here $X_j\in \frg_\0$, $Y_k\in \frg_\1$ as above.
By the Campbell--Hausdorff formula,
$$
\exp\Bigl(-\sum_{\lambda_s\in \cM_{\0+}: |\deg \lambda_j|= 2N}\lambda_j
 X_j  \Bigr)\cdot R
$$
has the form
$$
\exp\Bigl(\sum_{\lambda'_r\in \cM_{\0+}: |\deg \lambda'_j|> 2N} 
\lambda'_r  X'_r+\sum_{\deg \mu_k<2N} \mu_k Y_k+
 \sum_{\deg \mu'_k>2N} \mu'_k Y'_k\Bigr)
$$

Repeating this transformation, we represent $R$
as
$$
R=\exp(w_2)\exp(w_4)\dots \exp\Bigl
(\sum%_{\deg \mu'_k>2N}
 \wt\mu_t \wt Y_t
\Bigr),
$$
where each $w_{2N}$ is a sum of elements $\lambda X$ with $\deg\lambda=2N$.
It remains to refer to statement c).

\sm

{\it Uniqueness.} It is sufficient to show that a product
\begin{equation}
\exp\Bigl(\sum \lambda_k X_k\Bigr)\cdot\exp \Bigl(\sum \kappa_m Y_m\Bigr)
\label{eq:exp-X-exp-Y}
\end{equation}
can not have the form $\exp \bigl(\sum \kappa_m Y_m\bigr)$.

Denote the minimal degree in $\fra$ of non-zero summands
$\lambda_k X_k$ by $2p$.
 We reorder this sum,
$$
\sum \lambda_k X_k =
\sum_{\lambda_s\in \cM_{\0+}: |\deg \lambda_j|= 2p}\lambda_s X_s+
\sum_{\lambda_t\in \cM_{\0+}: |\deg \lambda_t|\ge 2p+2}\lambda_t X_t
$$
assuming that the first sum is non-zero.
 Applying the Campbell--Hausdorff formula to \eqref{eq:exp-X-exp-Y}
 we get an expression of the form 
 \begin{equation*}
\exp\Bigl( \sum_{ |\deg \lambda_j|= 2p}\lambda_s X_s
+ \sum_{|\deg \delta_\tau|\ge 2p+2}\delta_\tau X_\tau+\sum \mu_\sigma Y_\sigma\Bigr),
\end{equation*}
and the first sum in brackets is nonzero.

\sm

{\bf\punct Proof of  Proposition \ref{pr:1}.}
{\sc Statement b).} 
Decompose an element $L\in G_{\vel}(\cA)$ into a product 
of the form $\prod (1+t_k \dot\fra^{J_k} Q_{J_k})$. 
We say that the {\it level} of a factor $(1+\nu Z)$ is the degree 
of $\nu$.   If we transpose two factors of our product  according (\ref{eq:reorder-1})
with levels $\alpha$ and $\beta$, then the additional factor has
level $\alpha+\beta>\max (\alpha,\beta)$. Consider all factors of level 1 and put them into a correct $\succ$-order. Then additional
factors arising after this operation have order $\ge 2$.  Next, consider factors of level 2 and put them into a correct $\succ$-order 
between elements of order 1. Etc.

Conversely, consider two representations of $L$ in the desired form. Then  factors of level 1  are fixed. If  factors of level 1 are fixed, then factors of level
2 are fixed, etc.

\sm

{\sc Statement a).} In proof of Statement b) we used only transformations \eqref{eq:reorder-1}, \eqref{eq:relation-2}.
In particular, this allow to reduce a product $L$, which equals 1, to the canonical form.

\sm

{\bf \punct Functoriality.} Consider a homomorphism
$\frg\to \frh$ of Lie  superalgebras. It generates a homomorphism
of the corresponding enveloping algebras $\cU(\frg)\to \cU(\frh)$,
a homomorphism of their Grassmannizations 
$\cU[\frg;\cA]\to \cU[\frh;\cA]$, and a homomorphism 
$$
G_\vel(\cA)\to H_\vel(\cA).
$$

In particular, a representation $\rho$ of a superalgebra  $\frg$
in a superspace  $V$ generates a representation of the surplace group $G_\vel(\cA)$ in $V[\cA]$.
Namely, we represent each $g\in G_\vel(\cA)$ as a product of generators 
of form \eqref{eq:lambda-X}, \eqref{eq:mu-Y} and set
$$
\rho(1+\lambda X)=1+\lambda\rho(X);
\qquad
\rho(1+\mu Y)=1+\mu \rho(Y).
$$

\sm

{\bf \punct Central extensions and projective representations.} Consider a superalgebra $\frg=\frg_\0\oplus \frg_\1$, let $\tau\in \frg_\0$
be a central element. Denote $\frg':=\frg/\C \xi$.
Then we can consider $G_\vel(\cA)$ as a center
extension of $G'_\vel(\cA)$ in the following sense.

Fix a linear section $\upsilon:\frg_\0'\to \frg_\0$ 
and extend this map to $\frg_\1$
assuming $\upsilon(Y)=Y$ for $Y\in \frg_\1$.
 So, $\frg$ is a direct sum of the subspaces $\upsilon(\frg')\oplus \C\xi$.  
\begin{align*}
[\upsilon(X_1), \upsilon(X_2)]_s&=\upsilon([X_1, X_2]_s)+ c(X_1,X_2)\xi,\quad  [\upsilon(X),\upsilon(Y)]_s=\upsilon([X,Y]_s),
\\
 [\upsilon(Y_1),\upsilon(Y_2)]_s&=\upsilon([Y_1.Y_2]_s)+\sigma(Y_1,Y_2)\xi,
\end{align*}
where
 $c:\frg_\0'\times \frg_\0'\to \C$, $\sigma:\frg_\1\times \frg_\1\to \C$ are certain bilinear functions.
Keeping in mind Proposition \ref{pr:1}, we fix canonical decompositions  \eqref{eq:canonical-decomposition-vel}
for all elements of $G'_\vel(\cA)$. Applying $\upsilon$ to these decompositions, we
define the map $\Upsilon: G_\vel'(\cA)\to G_\vel(\cA)$ 
 \begin{multline*}
 	\Upsilon\Bigl((1+s_1 \dot \fra^{I_1} Z_1)(1+s_2 \dot \fra^{I_2}  Z_2)\dots (1+s_N \dot \fra^{I_N}  Z_N)\Bigr)=
 	\\=
 	(1+s_1 \dot \fra^{I_1} \upsilon(Z_1))(1+s_2 \dot \fra^{I_2} 
 	 \upsilon(Z_2))\dots (1+s_N \dot \fra^{I_N}  \upsilon(Z_N)).
 \end{multline*}

\sm

Denote by $\frx$ the one-dimensional Lie algebra generated
by $\xi$, by $X_\vel(\cA)$ the corresponding surplace group.

 \begin{proposition}
 	$$\Upsilon(g)\Upsilon(h) =C(g,h) \Upsilon(gh),$$
 	where 
 	$C(g,h)$ is an polynomial expression of the form
 	$$
 	C(g,h)=1+ \sum_{j=1}^N K_j(g,h)\xi^j.\qquad \text{where}\qquad K_j(g,h)\in \cA_{\0+}.
 	$$ 
 	Equivalently, 	$C(g,h)$ is a product of elements of the form $1+\lambda \xi$, where $\lambda\in \cM_{\0+}$.
 	
 So, $G_\vel(\cA)$ is a central extension of $G_\vel'(\cA)$ by
the group $X_\vel(\cA)$. 
 \end{proposition}
 
 {\sc Proof.}  Let us reduce $gh$ to the canonical form \eqref{eq:canonical-decomposition-vel}
 using the relations \eqref{eq:reorder-1}-\eqref{eq:reorder-2} and repeat all steps for 
 $\Upsilon(g)\Upsilon(h)$. For $X_1$, $X_2\in \frg_\0'$, $\lambda_1$, $\lambda_2\in\cM_{0+}$,
    we have 
 \begin{multline*}
 	\bigl(1+\lambda_1\upsilon(X_1)\bigr)\bigl(1+\lambda_2\upsilon(X_2)\bigr)  =\bigl(1+\lambda_2\upsilon(X_2)\bigr)\bigl(1+\lambda_1\upsilon(X_1)\bigr)
 	\times\\\times 
 \bigl(1+\lambda_1\lambda_2[\upsilon(X_1),\upsilon(X_2)]\bigr).
 \end{multline*}
 We transform the last factor to
 \begin{multline*}
 \Bigl(1+\lambda_1\lambda_2\bigl(\upsilon ([X_1,X_2])+c(X_1,X_2)\xi\bigr)\Bigr)=
  \bigl(1+\lambda_1\lambda_2 \upsilon([X_1,X_2])\bigr) 
  \times\\ \times \bigl(1+c(X_1,X_2)\lambda_1\lambda_2\xi\bigr),
 \end{multline*}
and we split off the cental factor $\bigl(1+c(X_1,X_2)\lambda_1\lambda_2\xi\bigr)$.
Examination of other cases of the relations  \eqref{eq:reorder-1}-\eqref{eq:reorder-2}
is similar.
\hfill $\square$

\sm

Let $\rho$ be a representation of $\frg$, where $\xi$ acts by a scalar operator $s\cdot 1$. 
Then we can  consider  
$\rho$ as a projective representation $\rho'$ of the surplace group $G'_\vel(\cA)$,
$$
\rho'(g)\rho'(h)=C(g,h)\rho'(gh),
$$ 
where $C(g,h)\in \cA_\0$, $C(g,h)_\downarrow=1$.

\sm

{\bf\punct Groups $\boldsymbol{G(\cA)}$.%
\label{ss:supergroups}} If $\frg$ is finite-dimensional, then we 
have a connected Lie group $G_\0$ corresponding to $\frg_\0$
(recall that his group
is defined up to a covering). 
 We have the adjoint action of the Lie group $G_\0$  by automorphisms of 
 $\frg$, therefore it acts on $\frg(\cA)$, $\cU[\frg;\cA]$,
 and
  $G_\vel(\cA)$. We consider
the semidirect product
$$
G(\cA):=G_\0\ltimes G_\vel(\cA)
$$
(for different $G_\0$ we get different 
$G[\cA]$).

\sm

{\sc Examples.} a) For the Lie superalgebra
$\frg=\frg\frl(p|q)$, we have $\frg_\0=\frg\frl(p)\oplus \frg\frl(q)$.
We set%
\footnote{Notice  that $\SL(n,\R)$ has a two-sheeted covering group.
A covering group also exists for $\GL(n,\C)$, since we can cover 
the quotient group
$\GL(n,\C)/\SL(n,\C)$, which is the multiplicative group
$\C^\times$ of the complex field. It seem that in the  both cases
the resulting groups $G(\cA)$ 
are not interesting. But in the next example a covering group
is a natural object.}
 $G_\0=\GL(p,\K)\times \GL(q,\K)$ and come to the
 the group $\GL(p|q;\cA)$
considered in Subsect. \ref{ss:gl-group}.
Indeed, by functoriality, we have a homomorphism from
the abstract surplace group $\GL_\vel(p|q,\cA)$
to the group $\GL_\vel(p|q,\cA)$ defined above. But the exponential
map establishes of one-to-one correspondence of both groups
with the Lie algebra $\frg\frl_\vel(p|q,\cA)$. So the homomorphism is
an isomorphism. 

\sm

b) It is to show that  the similar statement is valid
for the group $\OSp(2p|2q,\cA)$,
see Subsect. \ref{ss:osp-group}.
 \hfill $\boxtimes$
 
 \sm
 
 Next, let $\frg=\frg_\0\oplus \frg_\1$ be a Fr\'echet Lie superalgebra,
 let a group $G_\0$ corresponds to $\frg_\0$ in any reasonable sense%
 \footnote{Below we discuss explicit cases, and possible general formalities
 are not interesting for us.},
 and $G_\0$ has a well-defined adjoint action on $\frg$.
 Then we can repeat the same construction.

\sm

{\bf\punct A partial complexification.}
Below we need the following situation, which slightly differs from
Subsect. \ref{ss:supergroups}.
 We have an infinite-dimensional real
  Lie superalgebra $\frg=\frg_\0\oplus \frg_\1$,
  and and
an infinite-dimensional Lie group $G_\0^\R$
corresponding to the Lie algebra $\frg_\0$.
We consider the complexification $\frg^\C$ of $\frg$
and the corresponding surplace group $G^\C_\vel(\cA)$.
 Then we have a semidirect
product
$$
G_\0\ltimes G^\C_\vel(\cA).
$$

\vspace{22pt}

\begin{center}
\bf \large Addendum to Section \ref{s:surplace}. 
Characterization of surplace groups in terms of the coproduct
\end{center}

Here we present an additional description of surplace groups
(cf. \cite{Reu}, Theorem 3.2, \cite{Ner-braid}, Theorem 2.3),
which actually is not used in the present work.

\sm

{\bf\punct The coproduct in $\boldsymbol{\cU(\frg)}$.}
Consider the linear space 
$\cU(\frg)\otimes \cU(\frg)$ equipped with a structure of an associative 
algebra defined by the formula:  
$$
(u\,\ov\otimes\, w)\cdot (u'\,\ov\otimes\, w')=(-1)^{p(w)p(u')} (uu'\,\ov\otimes\, ww'),
\qquad u,v,u,v \in \cU(\frg).
$$
Denote this algebra by $(\cU\,\ov\otimes\, \cU)(\frg)$

The coproduct $\Delta:\cU(\frg)\to (\cU\,\ov\otimes\, \cU)(\frg)$ 
is the homomorphism of associative algebras
defined on generators by
$$
\Delta (X):=X\,\ov\otimes\,  1+1\,\ov\otimes\, X, \qquad\text{where $X\in \frg$,}
$$
see \cite{Mus}, Subsect. A.2.2.

 Consider an element
\begin{equation}
E(I;J):=
X_{i_1} \dots X_{i_p} Y_{j_1}\dots Y_{j_q}
\label{eq:EEE}
\end{equation}
of the \Poincare--Birkhoff--Witt basis,
$i_1\le i_2\le \dots\le i_p$, $j_1<j_2<\dots<j_q$.
Then we have the following shuffle formula:
$$
\Delta E(I;J)=\sum E(I';J')\,\ov\otimes\, E(I'';J'')
$$
where the summation is taken over all collections $(I';J')$ having the following form:
$$
(i_{\alpha_1},\dots, i_{\alpha_r}; j_{\beta_1},\dots, \beta_s),
$$
where
$$
\alpha_1<\dots <\alpha_r,\,\, \beta_1<\dots<\beta_s,
\qquad\text{and $0\le \alpha_r\le p$, $0\le \beta_s\le q$,} 
$$
a collection $(I'',J'')$ is complementary to $(I',J')$
in $(I,J)$. 

Since some $j_\mu$ can coincide, the sum \eqref{eq:EEE}
can have similar summands. But

\sm

---  {\it for different
$(I_1;J_1)$, $(I_2;J_2)$ there are no similar summands
in $\Delta E(I_1;J_1)$ and $\Delta E(I_2;J_2)$}.

\sm 
 
{\bf \punct A characterization of surplace groups.}
Next, we
  define the algebra 
  $$(\cU\,\ov \otimes\, \cU)[\frg,\cA]:=
  \bigl(\cU\,\ov \otimes\, \cU(\frg)\bigr)\otimes \cA,$$
  assuming that elements of $\cA$ commute with elements of
  $(\cU\ov \otimes \cU)[\frg;\cA]$. We define 
  a parity of $\mu\cdot Q\,\ov\otimes\, R\in (\cU\,\ov\otimes\, \cU)[\frg;\cA]$
  as $p(\mu)+p(Q)+p(R)$.

 We extend $\Delta$ to a homomorphism
 from $\cU[\frg,\cA]$ to 
  $(\cU\,\ov \otimes\, \cU)[\frg,\cA]$.

\begin{proposition}
 Consider   $r\in \cU[\frg;\cA]_\0$ such that $r_\downarrow=1$. Then $r\in G_\vel(\cA)$ if and only if 
 $\Delta(r)=r\,\ov\otimes\, r$.
\end{proposition}

{\sc Proof.}
Consider the set $G_\vel'$ consisting of even $r$ satisfying $\Delta r=r\,\ov\otimes\, r$. Clearly, this set is closed with respect to multiplication, 
and generators $1+\nu Z$ of $G_\vel(\cA)$ are contained
in $G_\vel'$. Therefore,
$G_\vel(\cA)\subset G_\vel'$.

Let us show that any $r\in G_\vel'$ is contained in $G_\vel(\cA)$.
For a monomial $\dot\fra^I$ denote by $\cN(I)$ the ideal
in $\cA$ generated by all monomials of degree $\ge |I|$
except $\dot \fra^I$. Consider  monomials $\dot\fra^J$
which are actually  contained in the expression for $r$,
consider a monomial of the minimal degree, say $\dot\fra^I$,
and consider the corresponding term 
$\dot\fra^I\sum c(I;J)E(I;J)$, there $c(I,J)\ne 0$ 
are complex coefficients.
Then 
\begin{multline}
\Delta(r)=\\
=\Delta(1+\dot\fra^I\sum c(I;J)E(I;J))=
1\ov\otimes 1+\dot\fra^I\sum c(I;J)\, \Delta E(I;J)
\quad \bigl(\mathrm{mod} \,\cN(I)\bigr);
\label{eq:Delta1}
\end{multline}
\begin{multline}
\Delta (r)
=r\ov\otimes r
=\\ =1\ov\otimes 1+\dot\fra^I
\bigl(\sum c(I;J)E(I;J)\bigr)\ov\otimes 1+1\ov\otimes \bigl(\sum c(I;J)E(I;J)\bigr)\quad \bigl(\mathrm{mod} \,\cN(I)\bigr).
\label{eq:Delta2}
\end{multline}
Split   $\sum c(I;j)E(I;J)$ as $S_1+S_2$, where 
$S_1=\sum c_\alpha Z_\alpha$ is a subsum
consisting of terms of degree 1.

First, assume that $S_1\ne 0$.
We have $\exp(-S_1)\in G_\vel(\cA)$
and 
$$
\exp(-S_1)r=\exp(S_2) \quad \bigl(\mathrm{mod} \,\cN(I)\bigr).
$$
So, we can avoid this case.

Second assume that $S_1=0$, $S_2\ne 0$.
Since we have no cancellations in the sum
$\sum c(I;J)\Delta E(I;J)$, the expression \eqref{eq:Delta1}
contains terms of the form $\dot\fra^I Z\ov\otimes E(K,L)$,
where $Z=X_i$ or $Y_j$ and $|K|+|L|>0$. But such terms
can appear in \eqref{eq:Delta2}. We come to a contradiction.
 \hfill $\square$

\section{Diffeomorphisms and contactomorphisms 
of the supercircle $S^{1|1}$\label{s:cont}}

\COUNTERS

Here we discuss various descriptions of the group of
contactomorphisms of a supercicle $S^{1|1}$.

\sm 

{\bf\punct  Group of diffeomorphisms of the circle.}
Consider the circle $S^1=\R/2\pi \Z$ with a coordinate $\phi\in\R/2\pi\Z$. 
Denote $\partial_\phi:=\frac \partial{\partial\phi}$.

Denote by $\Diff(S^1)$ the group  of $C^\infty$-smooth  diffeomorphisms of the circle, by $\SDiff(S^1)$ its connected subgroup consisting of orientation preserving diffeomorphisms.  
The universal covering group $\SDiff^\sim(S^1)$ of $\SDiff(S^1)$
can be realized as the group
of diffeomorphisms of the line $\R$ satisfying the condition
\begin{equation}
q(\psi+2\pi)=q(\psi)+2\pi.
\label{eq:Diff-covering}
\end{equation}
Sometimes, we prefer to think 
that $\SDiff^\sim(S^1)$ consists of diffemorphisms
of $\R$ satisfying $q(\psi+\pi)=q(\psi)+\pi$.

Its center $\cZ$ of $\SDiff^\sim(S^1)$
 is isomorphic to $\Z$, it consists of shifts $\psi\mapsto \psi+2\pi k$;
the quotient $\SDiff^\sim(S^1)/\cZ$ is $\SDiff(S^1)$.  

We also consider the two-sheeted covering $\SDiff^{(2)}(S^1)$ of $\SDiff(S^1)$, 
we realize it as the  subgroup  
in $\SDiff(S^1)$ consisting of diffeomorphisms satisfying
$$
q(\phi+\pi)=q(\phi)+\pi.
$$
%So, we also consider $\SDiff^\sim(S^1)$ as the group of diffeomorphisms
%of $\R$ satisfying $q(\psi+\pi)=q(\psi)+\pi$.

\sm

Denote by $\vect^\R(S^1)$ the {\it Lie algebra of
$C^\infty$-smooth
real vector fields} $a(\phi) \partial_\phi$
on $S^1$, this algebra is the Lie algebra of the group $\SDiff(S^1)$
(and of its coverings $\SDiff^{(2)}(S^1)$, $\SDiff^{\sim}(S^1)$).

Denote by $\vect(S^1)$ the {\it complexification of the Lie
algebra $\vect^\R(S^1)$}.
%it consists of differential operators $a(\phi) \partial_\phi$, where 
%$a(\phi)$ are $C^\infty$=smooth complex-valued functions on the circle
 Denote 
 $$L_n:=e^{in\phi}i\partial_\phi.$$
  Then
$$
[L_n,L_m]=(n-m) L_{m+n}.
$$

Recall that the  cotangent vector bundle $T^*(S^1)$ over $S^1$ admits  tensor powers $(T^*)^{\otimes \lambda} (S^1)$ of arbitrary complex degree $\lambda$.
 More precisely, we consider formal expressions of the form
 $u(\phi)\, (d\phi)^\lambda$, where $u(\phi)$ is a smooth function. The group
 $\SDiff(S^1)$  acts on such expressions  
by the formula
$$
T_\lambda (q)\, u(\phi) (d\phi)^\lambda=u(q(\phi))\,q'(\phi)^\lambda \,
(d\phi)^\lambda.
$$ 
Vector fields $r(\phi)\partial_\phi$ act on densities by
$$
\bigl(r(\phi) \partial_\phi\bigr)\, \bigl(u(\phi) (d\phi)^\lambda\bigr)=
\bigl(r(\phi) u'(\phi)+\lambda r(\phi)' u(\phi)\bigr)(d\phi)^\lambda.
$$   
Densities of degree $\lambda=-1$ are vector fields.

On the other hand, we can consider densities $f(\phi)+\ln d\phi$.
defined up an addition of a constant function.
Precisely, fix real $\mu$, $\nu$, consider an element  $\wt q\in \SDiff^\sim(S^1)$
covering a given $q\in \SDiff(S^1)$. Then the formula
$$
A(q)f(\phi)=f(q(\phi))+\mu(\wt q(\phi)-\phi)+\nu \ln q'(\phi) 
$$
determines an affine action of $\SDiff(S^1)$ in $C^\infty(S^1)/\R$.

\sm

{\bf \punct Surplace diffeomorphisms of the  circle.}
Now we apply the construction of the previous section 
to the Lie algebra $\vect(S^1)$.
 
 Consider the space $C^\infty[S^1;\cA]:=C^\infty(S^1)\otimes \cA$, we regard it as a right $\cA$-module.
 Elements of $C^\infty[S^1;\cA]$ are finite sums
$$
f(\phi;\fra)=\sum_J f_J(\phi)\dot\fra^J,
$$
where $f_J(\phi)$ are smooth complex-valued functions on the circle.

An $\cA$-valued diffeomorphism is a finite expression of the form
\begin{equation*}
q(\phi;\fra)=q_\varnothing(\phi)+\sum_{I:\, p(I)=\0,\,\#I>0}  q_I(\phi) \dot\fra^I,
\end{equation*}
where $q_\varnothing \in\Diff(S^1)$, other coefficients $q_I(\phi)$ are complex-valued smooth functions on the circle.
For $f\in C^\infty[S^1;\cA]$ and an $\cA$-valued diffeomorphism we define
\begin{multline}
f(q(\phi;\fra);\fra)=f(q_\varnothing(\phi);\fra)+ \\ +
\sum_{n\ge 1}\frac 1{n!}
 \partial_\phi^n f(\phi;\fra)\Bigr|_{\phi\mapsto q_\varnothing(\phi)} 
 \Bigl(\sum_{I:\, p(I)=\0,\#I>0}  q_I(\phi) \dot\fra^I \Bigr)^n
\label{eq:q-phi}
\end{multline}
(so, we consider $\fra_j$ as a kind of actual infinitesimals). In the same way, we define
a composition of $\cA$-valued diffeomorphisms $r$ and $q$: we write
$$
r\bigl(q(\phi;\fra);\fra\bigr)
$$
and decompose this expression  according \eqref{eq:q-phi}.

We say that  {\it surplace diffeomorphisms} are transformations of the form:
\begin{equation}
q(\phi;\fra)=\phi+\sum_{I:\, p(I)=\0, \#I>0}  q_I(\phi) \dot\fra^I
\label{eq:surplace}
\end{equation}
with complex-valued $q_I\in C^\infty(S^1)$. We denote the group
of all $\cA$-valued diffeomorphisms by $\Diff(S^1;\cA)$,
the subgroup of surplace diffeomorphisms by $\SDiff_\vel(S^1;\cA)$.
We have
$$
\Diff(S^1;\cA)/\SDiff_\vel(S^1;\cA)\simeq \Diff(S^1).
$$

Next, consider the surplace group $\VECT_\vel(\cA)$ corresponding to the Lie algebra $\vect(S^1)$ (it is considered as a superalgebra with trivial odd part). The homomorphism of $\vect(S^1)$ to the Lie algebra 
of bounded operators in $C^\infty(S^1)$ generates the action 
of the group $\VECT_\vel(S^1;\cA)$  in $C^\infty(S^1;\cA)$ by operators
\begin{equation}
f\mapsto \bigl(1+\lambda_1 a_1(\phi)\partial_\phi\bigr)\dots \bigl(1+\lambda_k a_k(\phi)\partial_\phi\bigr)  f(\phi;\fra),
\label{eq:vect-surplace}
\end{equation}
where $\lambda_j\in \cM_{\0+}$.
We have
$$
\bigl(1+\lambda a(\phi)\partial_\phi\bigr) f(\phi;\fra)= f(\phi;\fra)+\lambda a(\phi) \partial_\phi f(\phi;\fra),
$$
This transformation corresponds to the surplace diffeomorphism
\begin{equation}
q(\phi;\fra)=\phi+ \lambda a(\phi).
\label{eq:elementary}
\end{equation}

\begin{lemma}
\label{l:surplace}
 $\VECT_\vel(S^1;\cA)=\SDiff_\vel(S^1;\cA)$.
\end{lemma}

\sm

{\sc Proof.} We must verify that the homomorphism $\VECT_\vel(S^1;\cA)
\to \Diff_\vel(S^1;\cA)$ is bijectivite.

\sm

1) {\it Surjectivity.} Let us show that any surplace diffeomorphism 
$q(\phi;\fra)$
can be realized as an element of $\Vect_\vel(S^1;\cA)$. 
Consider a minimal degree $2k=\#I$, which actually appears in an expression \eqref{eq:surplace},  
$$
q(\phi;\fra)=\phi+\sum_{\#I=2k}  q_I(\phi) \dot\fra^I+ \sum_{\#J>2k}  q_J(\phi) \dot\fra^J.
$$
We multiply this diffeomorphism by
$$
\prod_{\#I=2k} \bigl(1-\dot\fra^I q_I\partial_\phi \bigr),
$$
the expression for the product does not contain terms of degree $2k$. We repeat the same argument, etc.

\sm

2) {\it Injectivity.} Consider an element \eqref{eq:vect-surplace}
of $\VECT_\vel(S^1;\cA)$. Without loss of generality we can assume 
that degrees of monomials $\lambda_j$ in \eqref{eq:vect-surplace} (nonstrictly) increase.
Let $\lambda_1$, \dots, $\lambda_p$ be (pairwise distinct) monomials of minimal degree.
Let $f(\phi)\in C^\infty(S^1)$, 
Then our transformation sends
$f$ to
$$ f(\phi)+
\sum_{j=1}^p  \lambda_j a_j(\phi) f'(\phi)
+\text{terms of higher degree in $\fra$.}
$$  
If $f'(\phi)$ has only finite number of zeroes, then the last expression $\ne f(\phi)$,
and so our transformation is not identical.
\hfill $\square$

\sm

{\bf \punct The supercircle $S^{1|1}$ and its diffeomorphisms.}
Now consider $(1|1)$-circle $S^{1|1}$ with an even coordinate 
$\phi\in \R/2\pi \Z$  and odd coordinate $\theta$, $\theta^2=0$.
We assume that $\theta$ commutes with phantom generators $\fra_j$.

This means that we consider the space $C^\infty[S^{1|1};\cA]$ of 'functions' of the form
$$
\sum_{I}  f_I(\phi)\fra^I+\sum_J   g_J(\phi)\fra^J \theta ,  
$$
where $f_I$, $g_J$ are smooth functions on the circle.

The Lie superalgebra 
$\vect(S^{1|1})=\vect(S^{1|1})_\0\oplus \vect(S^{1|1})_\1$ of vector
fields is defined in the following way:
$\vect(S^{1|1})_\0$ consists of vector fields
$$
a(\phi)\,\partial_\phi+b(\phi)\theta\,\partial_\theta,
$$
the odd part
$\vect(S^{1|1})_\1$ consists of vector fields
$$
c(\phi)\theta\, \partial_\phi+ d(\phi)\, \partial_\theta.
$$

A  {\it change of variables}  is a transformation
of the form
\begin{align}
\phi\mapsto \bfQ(\phi,\theta;\fra)= Q(\phi;\fra)+ q(\phi;\fra) \theta;
\label{eq:diff-1}
\\
\theta\mapsto \bfR(\phi,\theta;\fra)= R(\phi;\fra)\theta+r(\phi;\fra),
\label{eq:diff-2}
\end{align}
where parities with respect to $\fra_j$ satisfy the conditions 
\begin{equation}
 p(Q)=\0,\quad p(R)=\0,\quad p(q)=\1,\quad p(r)=\1,
 \label{eq:diff-plus1}
 \end{equation}
 and the functions $Q(\phi;\fra)_\downarrow$,
$R(\phi;\fra)_\downarrow$
take real values.
 Such substitution determines a linear transformation
of $C^\infty[S^{1|1};\cA]$. 

If   $\partial_\phi Q(\phi;\fra)_\downarrow\ne 0$,  
$R(\phi;\fra)_\downarrow\ne 0$ for all $\phi$,
then this transformation is invertible, we say that it is a
{\it (super)diffeomorphism of $S^{1|1}$}. We denote the group of all diffeomorphisms 
 by $\Diff(S^{1|1};\cA)$. We also consider the
  subgroup $\SDiff(S^{1|1};\cA)$ of {\it orientation preserving
  diffeomorphisms}, i.e.,
  diffeomorphisms satisfying the conditions
\begin{equation} 
 \partial_\phi Q(\phi;\fra)_\downarrow> 0,\quad
 R(\phi;\fra)_\downarrow>0
 .
 \label{eq:diff-plus2}
 \end{equation}

For $\SDiff(S^{1|1};\cA)$ we define the universal covering 
group $\SDiff^\sim(S^{1|1};\cA)$. We consider the group orientation preserving
diffeomorphisms of the superline $\R^{1|1}$, i.e., changes of variables \eqref{eq:diff-1}--\eqref{eq:diff-2}, where $\phi\in\R$, satisfying
\eqref{eq:diff-plus1}, \eqref{eq:diff-plus2}.  
The subgroup $\SDiff^\sim(S^{1|1};\cA)$ consists of such diffeomorphisms
commuting with the shift
$$
(\phi,\theta)\mapsto (\phi+2\pi,\theta).
$$
Equivalently,
\begin{align}
 Q(\phi+\pi;\fra)&=Q(\phi;\fra)+2\pi,& \quad   R(\phi+2\pi;\fra)&=R(\phi;\fra);
 \label{eq:symmetry1}
 \\
q(\phi+2\pi;\fra)&=q(\phi;\fra),& \quad r(\phi+2\pi;\fra)&=r(\phi;\pi).
\end{align}

\sm

%The corresponding Lie superagebra is generated by even vector fields of the form
%$a(\phi)\partial_\phi$, $b(\phi)\theta \partial_\theta$ and odd vector fields of the form
%$c(\phi) \theta \partial_\phi$,  $e(\phi)\partial_\theta$.

{\bf\punct  Affine actions of the group 
$\boldsymbol{\SDiff^\sim(S^{1|1};\cA)}$.%
\label{ss:affine-diff}}
For an element $g=(\bfQ,\bfR)\in \Diff^{\sim}(S^{1|1},\cA)$ 
given by \eqref{eq:diff-1}-\eqref{eq:diff-2}
we define two `cocycles'
\begin{align*}
\gamma_1[g]&:=	\bfQ(\phi,\theta;\fra)-\phi;\\
\gamma_2[g]&:=\ln \Ber
\begin{pmatrix}
\partial_\phi \bfQ& \partial_\theta \bfQ\\	
\partial_\phi \bfR& \partial_\theta \bfR
\end{pmatrix}.
\end{align*} 

\begin{lemma}
	Functions  $\gamma_1[g](\phi,\theta;\fra)$, $\gamma_2[g](\phi,\theta;\fra)$ are well defined, $2\pi$-periodic in $\phi$ (so, they are elements
of $C^\infty[S^{1|1},\cA]_\0$),
	 and 
	 satisfy
	the identity
\begin{equation}
\gamma[g_1\circ g_2](\phi,\theta;\fra)=\gamma[g_1](g_2(\phi,\theta;\fra))+ 
\gamma[g_2](\phi,\theta;\fra).
\end{equation}
\end{lemma}

{\sc Proof.}  For   $\bfQ=Q(\phi;\fra)+q(\phi;\fra)\theta$,
we represent
$$
Q(\phi;\fra)=Q_\downarrow(\phi)+\sum_{I:  p(I)=\0,\#I>0} Q_I(\phi) \dot\fra^I,
$$
here $Q_\downarrow(\phi)\in \SDiff^\sim(S^1)$ and
other $Q_I(\phi)$ are $2\pi$-periodic. By \eqref{eq:Diff-covering},
$Q(\phi)-\phi$ is $2\pi$-periodic. The expression $q(\phi;\fra)$
also is $2\pi$-periodic. 

The expression $\gamma_1$  has the form of a trivial cocycle%
\footnote{It is trivial in the space of functions on $\R$ and became nontrivial
on the space of functions on the circle.}
\eqref{eq:trivial-cocycle}  and satisfies the cocycle identity \eqref{eq:cocycle-identity}.

\sm

Notice that  
$$
\Ber
\begin{pmatrix}
\partial_\phi \bfQ& \partial_\theta \bfQ\\	
\partial_\phi \bfR& \partial_\theta \bfR
\end{pmatrix}_\downarrow
=\Ber
\begin{pmatrix}
	\partial_\phi Q_\downarrow(\phi)&0\\0&R_\downarrow
\end{pmatrix}.
$$
Therefore,
$$
\Ber(\cdot)= Q'_\downarrow(\phi) R_\downarrow(\phi)^{-1}
 \bigl(1 + S(\phi,\theta;\fra)\bigr),
$$ 
where $Q_\downarrow'(\phi)>0$, $R_\downarrow(\phi)>0$ and $S_\downarrow=0$. Therefore 
$$\ln(\Ber(\cdot))=\ln Q_\downarrow'(\phi)-\ln R_\downarrow(\phi)+
\ln (1 + S(\phi,\theta;\fra))
$$
is well defined.

The Jacobi matrix 
$$
J[\gamma]:=
\begin{pmatrix}
	\partial_\phi \bfQ& \partial_\theta \bfQ\\	
	\partial_\phi \bfR& \partial_\theta \bfR
\end{pmatrix}
$$
satisfies the usual chain identity, see, e.g., \cite{Ber-super}, formula (2.2.4),
$$
J[g_1\circ g_2]= (J[g_1]\circ g_2)\cdot J[g_2]
$$
The Berezinian is multiplicative, so $\ln\Ber(\cdot)$ satisfies
the cocycle identity.
\hfill $\square$

\sm

{\bf \punct Contact  diffeomorphisms.}
Consider {\it the standard contact form} on $S^{1|1}$:
$$
\Omega:= d\phi +\theta d\theta,
$$

We say that a diffeomorphism $g$ is {\it contact} if 
it sends the form $\Omega$ to an expression of the form
  $h(\phi,\theta;\fra)\cdot \Omega$. 

\begin{proposition}
 A diffeomorphism \eqref{eq:diff-1}--\eqref{eq:diff-2} is contact if and only if 
 \begin{equation}
 \begin{cases}
  R^2=\partial_\phi Q+ r\cdot \partial_\phi r;
  \\
  q+r R=0.
  \end{cases}
  \label{eq:system}
 \end{equation}
\end{proposition}

{\sc Proof.}
Substituting \eqref{eq:diff-1}--\eqref{eq:diff-2} to $\Omega$, we get
%\begin{align*}
% dt\mapsto \partial_\phi Q\, dt+ \partial_\phi q \theta \,dt+ q\,d\theta;
% \\
% \theta d\theta\mapsto (R\theta+r) (\partial_\phi R \theta \,dt+ \partial_\phi r\,dt+R\,d\theta;
%\end{align*}
\begin{equation*}
 \Bigl[ \partial_\phi Q+ \partial_\phi q\cdot \theta + R\cdot \partial_\phi r\cdot \theta+ 
 r \cdot \partial_\phi R\cdot\theta + r\cdot \partial_\phi r\Bigr]\,d\phi 
 +
 \bigl(q+R^2\theta +r R\bigr)\,d\theta.
\end{equation*}
On the other hand, we wait an expression
$$
h(\phi,\theta;\fra) \,d\phi+ h(\phi,\theta;\fra)\theta\,d\theta.
$$
Therefore, $[\dots]=h$ and 
$h(\phi,\theta;\fra)\theta=\bigl(q+R^2\theta +r R\bigr)$.
Hence $q+rR=0$.
So, $h(\phi,\theta;\fra)\theta =R^2\theta$, i.e.,
$$ (\partial_\phi Q+r\cdot \partial_\phi r)\theta=R^2\theta,$$
and this implies our statement.
\hfill $\square$

\sm

System \eqref{eq:system} of equations is easy to solve in terms of $Q$, $r$,
\begin{equation}
\begin{cases}
 R=\bigl(\partial_\phi Q+r \partial_\phi r\bigr)^{1/2};
 \\
 q=-r\bigl(\partial_\phi Q+r \partial_\phi r\bigr)^{1/2}.
\end{cases}
\label{eq:solution}
\end{equation}

Recall that $q_\downarrow=0$, $r_\downarrow=0$ 
and
$\partial_\phi Q_\downarrow$, $\partial_\phi R_\downarrow$ are  real and non-vanishing.
Therefore 
$$(\partial_\phi Q)_\downarrow= (R_\downarrow)^2>0.$$
  Therefore for each $Q$ we have two variants of a choice of $R$.
Denote by
$$\SCont(S^{1|1};\cA)$$
 the connected component of the group of all contactomorphisms of
  $S^{1|1}$, it is defined by the conditions
$$ R_\downarrow>0.$$
Under this condition, the square root is canonically defined.
Indeed, denote $U:=\partial_\phi Q_\downarrow$. Then
$$
\bigl(\partial_\phi Q+r \partial_\phi r\bigr)^{1/2}
:=\sqrt U\cdot
\Bigl[1+ 
\Bigl\{(U^{-1/2}\cdot\partial_\phi Q-1)+
(U^{-1/2}\cdot r \partial_\phi r\Bigr\}\Bigr]^{1/2}.
$$
We have $\{\dots\}_\downarrow=0$, and therefore we 
can apply the binomial formula,
$$[1+H]^{1/2}:=\sum_{k>0} \frac{(-1/2)_k}{k!}(-H)^k.$$

\sm

{\bf\punct Surplace contactomorphisms.}
A vector field $\upsilon$ is called {\it contact} if the Lie derivative $\cL_\upsilon \Omega$ has the form $H\cdot \Omega$
for some function $H$. It is easy to see that the Lie algebra 
$\cont(S^{1|1})$ 
has the  even 
part $\cont_\0(S^{1|1})$ consisting of vector fields
\begin{equation}
L(a):=
a(\phi)\partial_\phi+\tfrac 12 (\partial_\phi a)\,\theta \partial_\theta,
\label{eq:L(a)}
\end{equation}
 the odd part $\cont_\1(S^{1|1})$ consists of vector fields
\begin{equation}
M(b):=
b(\phi) \theta \partial_\phi- b(\phi) \partial_\theta.
\label{eq:M(b)}
\end{equation}
The Lie algebra $\cont_\0(S^{1|1})$ is nothing more than  $\vect(S^1)$.

We define the group $\SCont_\vel(S^{1|1};\cA)$ of %surplace
contact
diffeomorphisms satisfying the condition $Q(\phi,\fra)_\downarrow=\phi$, $R(\phi,\fra)_\downarrow=1$.

\begin{theorem}
 The group  $\SCont_\vel(S^{1|1};\cA)$
 of surplace contactomorphisms 
  coincides with the surplace  group 
 $\CONT_\vel(S^{1|1};\cA)$ corresponding to the Lie superalgebra $\cont(S^{1|1})$. 
\end{theorem}

{\sc Proof.} We must show the bijectivity
of the homomorphism $\CONT_\vel(S^{1|1};\cA)\to \SCont_\vel(S^{1|1};\cA)$.

\sm

1){\it The surjectivity.}
Elements $(1+\lambda L(a))$ correspond to diffeomorphisms
$$
\gamma_\0(\lambda a):\,(\phi,\theta)\mapsto \bigl(\phi +\lambda a,\theta+ \tfrac12 \lambda\, (\partial_\phi a) \, \theta\bigr),
$$
elements $(1+\mu  M(b))$ correspond to
$$
\gamma_{\1}(\mu b):\, (\phi, \theta)\mapsto \bigl(\phi+ \mu b\,\theta, \theta-\mu b\bigr).
$$

Let $g\in \SCont_\vel(S^{1|1};\cA)$. Repeating the proof of Lemma \ref{l:surplace} we multiply $g$ by elements $\gamma_\0(\lambda a)$
we can get an element of $\SCont_\vel(S^{1|1};\cA)$ satisfying $Q(\phi)=\phi$. This supercontactomorphism 
satisfies the  equations
\begin{equation}
 R=\bigl( 1+r \partial_\phi r\bigr)^{1/2}, \qquad
 q=-r\bigl( 1+r \partial_\phi r\bigr)^{1/2}.
 \label{eq:Rq11}
\end{equation}
We expand 
$$
r(\phi;\fra)=\sum_{m\ge 1} \Bigl(\sum _{I:\,\#I=2m+1} r_I(\phi) \dot\fra^I\Bigr),
$$
take the first $m$, say $l$, for which the corresponding summand is nonzero. We multiply
$r$ by
$$
\prod_{I:\, \#I=2l+1} \gamma_\1(-\fra^I r_I(\phi)) 
%\bigl(1+ r_I \dot\fra^I \theta \bigr)
$$
and come to a superdiffeomorphism with $Q(\phi)=\phi$ and new $r$ of the form $r=\sum_{m>l}\bigl(\dots\bigr)$.
Repeating this argument, we come to $r=0$. By \eqref{eq:Rq11}, $R=1$, $q=0$.

\sm

2) {\it The injectivity.}
Consider a nontrivial element
$$
H=\prod_p \gamma_\0(\dot\fra^{J_p} a_p(\phi))\cdot \prod_q \gamma_\1(\dot\fra^{K_q} b_q(\phi))\,\,\in \CONT_\vel(S^{1|1};\cA)
$$
(all monomials $\dot\fra^{J_p}$, $\dot\fra^{K_q}$
are different, and all $a_p$, $b_q$ are non-zero, such decomposition is possible by Proposition \ref{pr:1}).
Let $f(\phi)$ does not depend on $\theta$, $\fra$ and has a finite numbers of zeroes.
 Consider the minimal degree $m$
of monomials $\dot\fra^{J_p}$, $\dot\fra^{K_q}$, which appear in this product. If $m$ is even, then
$$
Hf(\phi)=f(\phi)+ \Bigl(\sum_{I_t:\,\#I_t=m} \dot\fra^{I_t} a_t(\phi)\Bigr)
\, \partial_\phi f(\phi) + 
\Bigl\{\text{terms of degree $>m$ in $\fra$}\Bigr\}. 
$$
So we get a non-identity operator $C^\infty(S^1)\to C^\infty[S^{1|1};\cA]$.

If $m$ is odd, then
$$
Hf(\phi)=f(\phi)+ \Bigl(\sum_{J_s:\,\#I_s=m} \dot\fra^{J_s} b_s(\phi)\Bigr) \theta
\, \partial_\phi f(\phi) + \Bigl\{\text{terms of degree $>m$ in $\fra$}\Bigr\}, 
$$
and again we get a non-identity operator.
\hfill $\square$

\sm 

The {\it Ramond algebra} (see Subsect. \ref{ss:NS}) is a center extension of $\cont(S^{1|1})$. 
Indeed, in notation \eqref{eq:L(a)}--\eqref{eq:M(b)},
 we take  the following vector fields
 \begin{align*}
 L_\alpha:=L(e^{i\alpha\phi})\in \cont_\0(S^{1|1}),\qquad \text{ where $\alpha\in\Z$};
 \\
 M_r:=M(e^{ir\phi})\in \cont_\1(S^{1|1}), \qquad
 \text{where $r\in \Z$.}
  \end{align*}
  Then we get commutation relations 
 \eqref{eq:NS1}--\eqref{eq:NS2} with $\zeta=0$.
 More precisely, we consider the subalgebra in
 $\cont(S^{1|1})$ consisting of finite sums $\sum a_k L_k+\sum b_n M_n$,
 the Ramond algebra is its central extension. 
 
 Our topic is the Neveu--Schwarz algebra, 
we consider the corresponding twin $\cont(S^{1|1}_\bullet)$
of $\cont(S^{1|1}_\bullet)$ 
 in the next subsection.
 
 \sm

{\bf\punct The superalgebra $\boldsymbol{\cont(S^{1|1}_\bullet)}$ and the group
$\boldsymbol{\SCont(S^{1|1}_\bullet;\cA)}$.%
}
Consider the contactomorphism $\sigma$ of $S^{1|1}$ defined by
\begin{equation}
\sigma: (\phi,\theta)\mapsto (\phi+\pi,-\theta)\in \Cont(S^{1|1};\cA).
\label{eq:cont-bullet}
\end{equation}
The super-group $\SCont(S^{1|1}_\bullet;\cA)$ consists of elements of 
$\SCont(S^{1|1}_\bullet;\cA)$ commuting with $\sigma$. Let us describe it 
in more details.

 We consider the subalgebra $\cont(S^{1|1}_\bullet)$
in $\cont(S^{1|1})$, whose even part $\cont_\0(S^{1|1}_\bullet)$
 consists of vector fields
\begin{equation}
L(a):=
a(\phi)\partial_\phi+\tfrac 12 (\partial_\phi a)\,\theta \partial_\theta,
\quad \text{where $a(\phi+\pi)=a(\phi)$,}
\label{eq:L(a)-sim}
\end{equation}
and the odd part $\cont_\1(S^{1|1}_\bullet)$ consists of vector fields
\begin{equation}
M(b):=
b(\phi) \theta \partial_\phi- b(\phi) \partial_\theta,
\quad \text{where $b(\phi+\pi)=-b(\phi)$.}
\label{eq:M(b)-sim}
\end{equation}

We choose the following basis  in $\cont(S^{1|1}_\bullet)$:
\begin{equation}
L_n=L\bigl(\tfrac i 2 e^{2in\phi} \bigr), \qquad M_r:=M\bigl(\tfrac i2 e^{2ir\phi} \bigr),\quad\text{where $n\in \Z$, $r\in \tfrac 12+\Z$.}
\label{eq:LMbasis}
\end{equation}
It is easy to see that these vector fields satisfy  the supercommutation relations \eqref{eq:NS1}--\eqref{eq:NS3} of the 
{\it Neveu--Schwarz
algebra} $\ns$ without $\zeta$,
i.e.,  we get the Lie superalgebra $\ns/\C\zeta$. The Lie superalgebra 
$\cont(S^{1|1};\cA)$ is a completion of $\ns/\C\zeta$.

\sm

Next, we  consider the corresponding group
$\SCont(S^{1|1}_\bullet;\cA)$, it consists
of elements
$$\phi\mapsto Q(\phi;\fra)+ q(\phi;\fra) \theta,
\quad
\theta\mapsto R(\phi;\fra)\theta+r(\phi;\fra).
$$
of
$\SCont(S^{1|1};\cA)$ commuting with $\sigma$.
They satisfy  the following symmetry conditions:
\begin{align}
 Q(\phi+\pi;\fra)&=Q(\phi;\fra)+\pi,& \quad   R(\phi+\pi;\fra)&=R(\phi;\fra);
 \\
q(\phi+\pi;\fra)&=-q(\phi;\fra),& \quad r(\phi+\pi;\fra)&=-r(\phi;\pi).
\end{align}
Such diffeomorphisms preserve the space of functions of the form
$$
f(\phi;\fra)+ g(\phi;\fra)\theta, \quad 
\text{where $f(\phi+\pi;\fra)=f(\phi;\fra)$, 
 $g(\phi+\pi;\fra)=-g(\phi;\fra)$.}
$$
In particular, our group contains the
$\SDiff^{(2)}(S^1)$.

\sm

{\sc Remark.}
We do not discuss the notion of supermanifolds,
see \cite{Lei}, \cite{Ber-super}, \cite{DeW}, \cite{Rog}.
According Batchelor theorem (see \cite{Rog}, Sect. 8.1), 
any $C^\infty$-supermanifold of dimension $m|n$ can be obtained in the following
way. We consider a $C^\infty$-manifold $M$, $n$-dimensional real vector bundle $L$
over  $M$, and take the direct sum of bundles $\oplus_{j=0}^n \Lambda^j L$,
where $\Lambda^j$ denotes the $j$-th exterior power.
For $M=S^1$ we have two bundles, trivial one and the M\"obius bundle.
We prefer to consider sections of the M\"obius bundle as
function on the circle satisfying $f(\phi+\pi)=-f(\phi)$
and this leads to our group $\SCont(S^{1|1}_\bullet;\cA)$.
\hfill $\boxtimes$

\section{Invariant differential operations and an embedding of 
$\cont(S^{1|1}_\bullet)$ to an orthosymplectic superalgebra%
	\label{s:embedding}}

\COUNTERS

{\bf \punct The description of $\boldsymbol{\cont(S^{1|1}_\bullet)}$ in  terms of invariant differential operations.}
Denote by $C^\infty_\pm(S^1)$ the space of $C^\infty$-smooth
functions on the circle $S^1=\R/2\pi\Z$ satisfying
$$
f(\phi+\pi)=\pm f(\phi).
$$

We need the following realization of the 
 Lie superalgebra 
 $$\cont(S^{1|1}_\bullet)=\cont_\0(S^{1|1}_\bullet)\oplus\cont_\1(S^{1|1}_\bullet)$$  
 proposed by Kirillov \cite{Kir}:
 
\sm 
 
 ---  the even part $\cont_\0(S^{1|1}_\bullet)$
  consists of smooth vector fields $a(\phi)\partial_\phi$
on the circle with $a(\phi)\in C^\infty_+(S^1)$, i.e., $a(\phi+\pi)=a(\phi)$; 

\sm

--- the odd part $\cont_\1(S^{1|1}_\bullet)$ consists of smooth densities $b(\phi)(d\phi)^{-1/2}$
 such that $b(\phi)\in C^\infty_-(S^1)$.

\sm

The supercommutator
 $[\cdot,\cdot]_s:\cont(S^{1|1}_\bullet)\times \cont(S^{1|1}_\bullet) \to \cont(S^{1|1}_\bullet)$ is defined 
as follows:

\sm

--- the bracket $\cont_\0\times \cont_\0 \to \cont_\0$
 is the commutator of vector fields: %\eqref{eq:skobka}, $[\cdot,\cdot]_s:=[\cdot,\cdot]$;
\begin{equation*}
[a_1 \partial_\phi,a_2\partial_\phi]_s=(a_1 \cdot \partial_\phi a_2- \partial_\phi a_1\cdot a_2)\,\partial_\phi;
%\label{eq:skobka}
\end{equation*}

\sm

--- the bracket $\cont_\0\times \cont_\1 \to \cont_\1$ 
is the natural action of vector fields on densities of
degree
$-1/2$,
\begin{multline*}
[a(\phi)\partial_\phi,b(\phi)(d\phi)^{-1/2})]_s=-[b(\phi)(d\phi)^{-1/2}), a(\phi)\partial_\phi]_s
=\\=
\bigl(b_1(\phi)\cdot \partial_\phi p(\phi)-\tfrac12 \partial_\phi a(\phi)\cdot  b(\phi)  \bigr)(d\phi)^{-1/2};
\end{multline*}

--- the bracket $\cont_\1\times \cont_\1 \to \cont_\0$
 is the product of densities
$$
[b_1(\phi)(d\phi)^{-1/2}),b_2(\phi)(d\phi)^{-1/2})]_s:=
2 b_1(\phi)b_2(\phi) (d\phi)^{-1}= 2 b_1(\phi)b_2(\phi) \partial_\phi.
$$

This bracket determines a structure of a Lie superalgebra on $\cont(S^{1|1}_\bullet)$,
 the commutation relations correspond to commutation relations
between vector fields $L(a)$, $M(b)$ of forms  \eqref{eq:L(a)-sim}, \eqref{eq:M(b)-sim}.

\sm

{\sc Remark.} These operations are 'natural', i.e., they are invariant
with respect to the two-sheeted covering $\SDiff^{(2)}(S^1)$
of $\SDiff(S^1)$.
\hfill $\boxtimes$

\sm 

The correspondence 
with the definition  above is given by
\begin{align*}
	\cont_\0 (S^{1|1}_\bullet): \qquad 
a(\phi)\partial_\phi+ \frac 12 (\partial_\phi a(\phi))\theta\partial_\theta
\,\,\longleftrightarrow \,\,
a(\phi)\partial_\phi;
\\
\cont_\1 (S^{1|1}_\bullet): \qquad 
b(\phi)\theta\partial_\phi-b(\phi)\partial_\theta
\,\,\longleftrightarrow \,\,
b(\phi) (d\phi)^{-1/2}.
\end{align*}

\sm

{\bf \punct The  orthosymplectic form.%
\label{ss:cont-osp}}
Consider a superlinear space $\cW:=\cW_\0\oplus \cW_1$
whose elements have the form
\begin{equation*}
f(\phi)\oplus g(\phi)(d\phi)^{1/2},
%\label{eq:f-oplus-g}
\end{equation*}
where
 $f\in C^\infty_+(S^1)$ is defined up to addition of constant functions, $g\in C^\infty_-(S^1)$.
We equip this space with the orthosymplectic form
\begin{align}
\Bigl\{ f_1(\phi)\oplus g_1(\phi) (d\phi)^{1/2}, \, f_2(\phi) \oplus g_2(\phi) (d\phi)^{1/2}\Bigr\}
=\\=
\frac 1{8\pi}\int_{S^1}(f_1\,df_2-f_2\,df_1)+ \frac{1}{2\pi} \int_{S^1}g_1(\phi) g_2(\phi) d\phi.
\label{eq:orthosymp}
\end{align}
In particular, the function $1\oplus 0$ is contained in the kernel of the form. This form 
is invariant with respect to the natural action of the group $\SDiff^{(2)}(S^1)$.
Expanding our functions into Fourier series, we get
\begin{multline*}
\Bigl\{
\sum p_n e^{2in\phi}\oplus \sum q_n e^{(2m+1)i\phi} (d\phi)^{1/2},\,
\sum p_n' e^{2in\phi}\oplus \sum q_n' e^{(2m+1)i\phi} (d\phi)^{1/2} 
\Bigr\}
=\\=
\sum_{n\in\Z} n p_n p'_{-n} +\sum_{m\in\Z} q_m q'_{-m}.
\end{multline*}

\sm 

For any $f(\phi)\oplus g(\phi)(d\phi)^{1/2}\in \cW$ we assign the `function' $F(\phi,\theta)=f(\phi)+g(\phi)\theta$.
In this language, the orthosymplectic form \eqref{eq:orthosymp}  can be written as
$$
\{F_1,F_2\}=\iint F_1(\phi,\theta)\cdot \bigl(\partial_\theta+\theta\partial_\phi\bigr) F_2(\phi,\theta)\,
\,d\phi\,d\theta,
$$
where $\int R(\theta)d\theta$ is the Berezin integral
(see \cite{Ber-super}, Sect 2.2, \cite{Rog}, Sect. 11.1)
$$
\int  \bigl(u(\phi)+v(\phi)\theta)\,d\theta:=v(\phi).
$$

We also define the inner product in $\cW$:
\begin{multline}
\bigl\la
\sum p_n e^{2in\phi}\oplus \sum q_n e^{(2n+1)i\phi} (d\phi)^{1/2},\,
\sum p_n' e^{2in\phi}\oplus \sum q_n' e^{(2n+1)i\phi} (d\phi)^{1/2}
\bigr\ra
=\\=
\sum_n |n|\, p_n \ov p'_{n} +\sum q_m \ov q'_{m},
\label{eq:inner}
\end{multline}
notice that $\la 1\oplus 0,1\oplus 0\ra=0$, and the form
is positive definite on the space $\bigl(C^\infty_+(S^1)/\C\cdot 1\bigr)
\oplus C^\infty_-(S^1)$. 

We also can write this inner product in the integral form,
\begin{multline}
\Bigl\la  f_1(\phi)\,\oplus\, g_1(\phi) (d\phi)^{1/2},
 f_2(\phi)\,\oplus\, g_2(\phi) (d\phi)^{1/2}
\Bigr\ra
=\\=
\frac 1{4\pi}\int_0^{2\pi} \mathrm {v.p.} \int_0^{2\pi} 
\cot\Bigl(\frac{\phi-\psi}2\Bigr) f'_1(\psi)\,d\psi\, \ov {f_2(\phi)} \,d\phi+\\
+\frac 1{2\pi}\int_0^{2\pi} g_1(\phi)\ov{g_2(\phi)}\,d\phi,
\label{eq:inner-2}
\end{multline}

{\sc Remark.} The inner product is invariant with respect 
to the natural action of the group $\SL(2,\R)\subset \Diff^{(2)}(S^1)$.
It is not $\Diff^{(2)}$-invariant, but it is 'almost invariant', see
\cite{Ner-book}, Sect. VII.2-3, and this allows our further story.
\hfill $\boxtimes$

\sm

{\bf\punct The embedding $\boldsymbol{\cont(S^{1|1}_\bullet)\to\osp(\cW)}$.}
We consider the action of the Lie superalgebra  $\cont(S^{1|1})$
 in $\cW$.
Vector fields $a(\phi)\partial_\phi$ act on
pairs \{a function, a density\} in the natural way:
\begin{multline*}
\bigl(a(\phi)\partial_\phi\oplus 0\bigr)\cdot \bigl(f(\phi)\oplus g(\phi) (d\phi)^{1/2}\bigr):=\\:=
a(\phi)\partial_\phi\cdot f(\phi)\,\oplus\, a(\phi)\partial_\phi\cdot g(\phi) (d\phi)^{1/2}
=\\=a(\phi)  f'(\phi)\oplus \Bigl(a(\phi)  g' (\phi)+\tfrac12  a'(\phi) g(\phi)\bigr)(d\phi)^{1/2}. 
\end{multline*}
An element $0\oplus b(\phi) (d\phi)^{-1/2}$ acts
as
\begin{multline*}
\bigl(0\oplus b(\phi) (d\phi)^{-1/2}\bigr)
\cdot
\bigl( f(\phi)\oplus g(\phi)(d\phi)^{1/2}\bigr):=\\:=
(b(\phi) (d\phi)^{-1/2}) \, \bigl(g(\phi)(d\phi)^{1/2}\bigr)\,\oplus\,
\bigl(df(\phi)\bigr) \, \bigl(b(\phi) (d\phi)^{-1/2}\bigr)
=\\=
  b(\phi) g(\phi)\oplus b(\phi)f'(\phi) (d\phi)^{1/2}.
\end{multline*}
It is easy to verify that these operators preserve the orthosymplectic form $\{\cdot,\cdot\}$ and so 
we have a homomorphism of Lie superalgebras 
$$\cont(S^{1|1}_\bullet)\to \osp(\cW).$$

\sm 

If we regard $\cW$ as the space of functions $F(\phi,\theta)$, then our action
is the action of $\cont(S^{1|1}_\bullet)$ discussed above

\begin{corollary}
The group 
$\SCont(S^{1|1}_\bullet;\cA)=\SDiff^{(2)}\ltimes \SCont_\vel(S^{1|1}_\bullet;\cA)$
 is contained
in the orthosymplectic group $\OSp(\cW;\cA)$ of the space $\cW[\cA]$.
\end{corollary}

According Subsect. \ref{ss:affine-diff} we also get a two parametric family of affine orthosymplectic
actions of $\SCont(S^{1|1}_\bullet;\cA)$ in $\cW[\cA]$,
\begin{multline*}
T_{\mu,\nu}(g)F(\phi,\theta;\fra)=
F\bigl(\bfQ(\phi,\theta;\fra),\bfR(\phi,\theta;\fra);\fra)+
\\
+
\mu \bigl(\bfQ(\phi,\theta;\fra)-\phi\bigr)
+\nu \ln\Ber\begin{pmatrix}
	\partial_\phi \bfQ& \partial_\theta \bfQ\\	
	\partial_\phi \bfR& \partial_\theta \bfR
\end{pmatrix},
\end{multline*}
where $g$ is given by \eqref{eq:diff-1}--\eqref{eq:diff-2}.

\sm

{\bf \punct Addendum. A description of  
$\boldsymbol{\CONT_\vel(S^{1|1}_\bullet)}$ as a subgroup  in 
$\boldsymbol{\OSp_\vel(\cW,\cA)}$.%
\label{ss:addendum-cont}} 
We use description given in  Theorem \ref{th:1}.d (keeping its notations).

\sm 

1) The group consisting of elements $\exp(\sum \lambda_j X_j)$ 
is $\SDiff^{(2)}_\vel(S^{1};\cA)$, subspaces $\cW_\0$, $\cW_\1$ are invariant.
In $\cW_\0$ the group $\SDiff_\vel(S^{1};\cA)$ acts as in the space of functions
$f(\phi;\fra)$. In the subspace $\cW_\1$ we have the action
$$
g(\phi;\fra)\,\theta\,\mapsto\, g\bigl(q(\phi;\fra)\bigr)\, q'(\phi;\fra)^{-1/2}\, \theta.
$$ 
The group $\Diff^{(2)}(S^1)$ acts by the same formulas.

\sm

2) The case $\exp(\sum \mu_j Y_j)$,
 where $Y_j\in \cont_\1$ and $\mu_j\in \cM_\1$  is more interesting.
We write elements of $Y_k\in \cont_\1^\sim$ as matrices
$$
Y_k=\begin{pmatrix}
   0& b_k(\phi)\\b_k(\phi)\partial_\phi&0
  \end{pmatrix},
$$
which are applied  to columns $\begin{pmatrix}
                           f\\g
                          \end{pmatrix}$.
                          
\begin{proposition}
Let $\mu_k\in \cM_{\ov1}$. Denote
\begin{equation*}
S:=\sum_{k<l} \mu_k\mu_l
\bigl(b_k\cdot\partial_\phi b_l-\partial_\phi b_k\cdot b_l\bigr),
\qquad 
 T:=\sum \mu_j b_j.
\end{equation*}
Then
\begin{equation}
\exp
\bigl(\sum \mu_k Y_k\bigr)
=\begin{pmatrix}
1&T\\
A(S)\,T\,\partial_\phi&B(S)
\end{pmatrix},
\label{eq:odd-contact}
\end{equation}
where 
\begin{align}
A(S)=\sum_{m=0}^\infty \frac 1{(2m+1)!} S^m= \frac {\sinh\sqrt S}{\sqrt S}
\label{eq:sh}
\\
B(S)=\sum_{m=0}^\infty \frac 1{(2m)!} S^m=\cosh \sqrt S.
\label{eq:ch}
\end{align}
\end{proposition}                          

Notice that $\sqrt S$ makes no sense but actually it is not present in
the series.

\sm            
                          
{\sc Proof.}
We have
$$
Y_1 Y_2-Y_2 Y_1=\begin{pmatrix}
                 0&0\\0& b_1\partial_\phi b_2-\partial_\phi b_1\cdot b_2
                \end{pmatrix}.
$$
Keeping in mind that
$$
\Bigl(\sum \mu_j Y_j\Bigr)^2=\sum_{k<l} \mu_k \mu_l\, (Y_k Y_l-Y_l Y_k),
$$
we get
\begin{multline*}
 \exp\Bigl(\sum \mu_j Y_j\Bigr)=
 \sum_{m=0}^\infty\frac 1{2m!} \Bigl(\sum_{k<l} \mu_k \mu_l (Y_k Y_l-Y_l Y_k)\Bigr)^{m}+\\+
  \sum_{m=0}^\infty\frac 1{(2m+1)!} \sum_{k<l} \bigl(\mu_k \mu_l (Y_k Y_l-Y_l Y_k)\bigr)^{m}\cdot \Bigl(\sum \mu_j Y_j\Bigr),
\end{multline*}
and we come to our statement.
\hfill $\square$

\section{Unitary  representations of the Neveu--Schwarz supergroup\label{s:integration}}

\COUNTERS

{\bf \punct  Unitary highest weight representations of the Virasoro algebra
$\boldsymbol\vir$.} For generalities on representations of Virasoro algebra,
see \cite{KR}, \cite{Ner-aff}. Initial formalities are similar
to the theory of Verma modules over semisimple Lie algebra,
see \cite{Dix}, Sect. 7.1. 

We keep the notation of Subsect. \ref{ss:virasoro}.
Let $(h,c)\in \C^2$. 
Recall that a  module $M$ over $\vir$ has highest weight 
$(h,c)$ if there exists a vector $v\in V$ such that

\sm

1) $L_j v=0$ for $j>0$;

\sm

2) $L_0 v=hv$, $\zeta v =c v$;

\sm

3) $v$ is cyclic, i.e., $M$ has no proper submodule containing $v$.

\sm

This implies that vectors of the form 
\begin{equation}
L_{-1}^{k_1} \, L_{-2}^{k_2}\dots L_m^{k_m} v
\label{eq:basis-verma}
\end{equation}
%(only finite number of powers are nonzero) 
span $M$.
A highest weight module is $\R$-graded, the degree of a
vector \eqref{eq:basis-verma} is $h+\sum jk_j$,
and we have
$$
L_0\cdot L_{-1}^{k_1} L_{-2}^{k_2}\dots L_m^{k_m} v
= \bigl(h+\sum_j jk_j\bigr) L_{-1}^{k_1} L_{-2}^{k_2}\dots L_m^{k_m} v
.$$

 For each highest weight
$(h,c)$ there is a unique module $M(h,c)=M_\vir(h,c)$ with highest weight $(h,c)$ ({\it Verma module}) such that vectors 
\eqref{eq:basis-verma} are linearly independent. All $\vir$-modules with highest weight $(h,c)$ are quotients
of $M(h,c)$. 
For $(h,c)\in\C^2$ being in a general position the module $M(h,c)$ is irreducible, the exceptional set
is a union of a countable family of quadrics given by Kac determinant formula, see \cite{Kac}, \cite{KR}, Sect. 8.3, see also some comments
in \cite{Ner-aff}.

For each $(h,c)$ there exists a unique  
 irreducible 
module $L(h,c)=L_\vir(h,c)$ with highest weight $(h,c)$, it is the quotient of $M(h,c)$ by the maximal
proper submodule, 

\sm

If $h$, $c$ are real, then there is a unique nondegenerate Hermitian form
({\it Shapovalov form}) $\la\cdot,\cdot\ra$ on $L(h,c)$ such that
 $L_k^*=L_{-k}$ with respect to this form.
 A module $L(h,c)$ is {\it unitarizable} if this form is positive definite.
 The domain of unitarizability is the quadrant
 $$h\ge 0,\qquad c\ge 1$$
  and a finite collection of points
 on each level 
 $$c=1-\frac{6}{(p+2)(p+3)},$$
  where $p=1$, 2, 3, \dots (discrete series),
 see%
 \footnote{For $c=0$ ($p=1$) we have only one-dimensional 
module $L(0,0)$.} 
 \cite{Ner-disser}, \cite{Fri0}, \cite{GOK}.
 If  a unitarizable module $L(h,c)$ satisfies $c>1$ and $h>0$  or $c=1$, $h\ne k^2$, then 
 $L(h,c)=M(h,c)$. Otherwise,
 we have degenerations. 
 
\sm 
 
 Any module $L(h,c)$ can be integrated to a unitary projective 
 representation
 $\rho_{h,c}$
 of the group $\SDiff(S^1)$. So, for operators of representation
we have
\begin{equation}
\label{eq:diff-projective}
\rho_{h,c}(g_1)\rho_{h,c}(g_2)=\sigma(g_1,g_2)\, \rho_{h,c}(g_1g_2), 
\end{equation}
where $g_1$, $g_2$ range in  $\SDiff(S^1)$, and
the  cocycle  $\sigma(g_1,g_2)$ takes values in the 
unit circle 
$\T$ in $\C$.
See \cite{Seg1}, \cite{Ner-highest}, \cite{GW}, \cite{Ner-book},
Sect. VII.2-3.

Below we need the following simple remarks about our cocycles.
We can  choose another operators $\rho'_{h,c}(g)= \gamma(g)\rho_{h,c}(g)$,
where $\gamma(g)\in \T$. Then we get an {\it equivalent cocycle}
\begin{equation}
\sigma'_{h,c}(g_1,g_2)= \frac{\gamma(g_1)\gamma(g_2)}{\gamma(g_1g_2)}
\,\sigma_{h,c}(g_1,g_2).  
\label{eq:sigma-prime}
\end{equation}

\begin{lemma}
{\rm a)}
Consider a unitary projective representation of $\SDiff(S^1)$.
 Fix a representative of an equivalence class
of cocycles. Then this representative  
    uniquely determines operators $\rho(g)$. 

\sm

{\rm b)} If $L(h,c)$ and $L(h',c)$ are unitary and $h-h'\in\Z$,
 then the corresponding cocycles are equivalent. 
\end{lemma}
 
{\sc Proof.}  a) The group $\SDiff(S^1)$
is simple as an abstract group (W.~Thurston, see \cite{Ban}, Theorem 2.1.1). In particular, it has no homomorphisms
to  $\T$.
Therefore, for any function $\gamma:\SDiff(S^1)\to \T$ the expression
 $\frac{\gamma(g_1)\gamma(g_2)}{\gamma(g_1g_2)}$ is not the identical
 1, and $\sigma'$ and $\sigma$ in \eqref{eq:sigma-prime}
 are different.
 
 \sm 

b) Let $\pi$ and $\tau$ are unitary projective representation of 
$\SDiff(S^1)$ with cocycles $\mu$ and $\nu$. Then their tensor product
 $\pi \otimes \tau$ has cocycle $\mu\nu$.

Now let $L(h_1,c_1)$, $L(h_2,c_2)$ be unitarizable
$\ne L(0,0)$.
Clearly, 
$$
L(h_1,c_1)\otimes L(h_2,c_2)= \bigoplus_{n=0}^\infty 
\tau_{h_1,c_1,h_2,c_2}(n) \,
L(h_1+h_2+n, c_1+c_2),
$$
with non-negative multiplicities $\tau_{h_1,c_1,h_2,c_2}(n)$ while
$\tau_{h_1,c_1,h_2,c_2}(0)=1$

 Clearly, for nondegenerate  $L(h_1,c_1)$, $L(h_2,c_2)$ we have
 $\tau_{h_1,c_1,h_2,c_2}(n)=p(n)$,
where  $p(n)$ is the number of partitions of $n$, these numbers are positive.
Therefore, all modules $L(h_1+h_2+n, c_1+c_2)$ have the same cocycles.
Hence, for $c>2$ and for any $h>0$ all modules $L(h,c)$, $L(h+1,c)$,
$L(h+2,c)$, \dots have the same cocycles. Taking products
$L(0,c_1)\otimes L(0,c_2)$ we see that $\sigma_{0,c_1+c_2}$
is equivalent to some $\sigma_{N,c_1+c_2}$. So, the statement
is proved for $c>2$.

Now, let $c\le 2$. Let $h'>0$, $c'>2$.
Then $L(h,c)\otimes L(h',c')$ and $L(h,c)\otimes L(h'+m,c')$
expand as certain direct sums of modules $L(h+h'+j,c+c')$ and
$L(h+h'+m+j,c+c')$ respectively. All such summands have the same cocycles,
therefore $L(h,c)$ and $L(h+m,c)$ have the same cocycles.
\hfill $\square$

\sm

{\bf \punct Unitarizable representations of Lie superalgebras.}
Let $\frg=\frg_\0\oplus \frg_\1$ be a real Lie superalgebra,
let $G_\0$ be  the real simply connected group
corresponding $\frg_\0$. 
Let $H=H_\0\oplus H_\1$ be a superspace, and summand be  Hilbert spaces.
Let we have a unitary representation
$\Lambda=\Lambda_\0\oplus \Lambda_\1$ of $G_\0$ in $H_\0\oplus H_\1$.
We say that we have an {\it unitarizible representation
$\lambda$ of the superalgebra} $\frg$ if we have dense subspaces
$\cH_\0\subset H_\0$,  $\cH_\1\subset H_\1$,
and a representation $\lambda$ of $\frg$ in $\cH_\0\oplus \cH_\1$
such that

\sm 

--- for all $X\in \frg_\0$ operators $i\lambda(X)$ are essentially self-adjoint in $H_\0\oplus H_\1$
and correspond to the representation $\Lambda$ of $G_\0$;

\sm

--- for all  $Y\in \frg_\1$ the operators $\lambda(Y)$
send $H_\0\to H_\1$, $H_\1\to H_\0$
and operators $e^{\pi i/4} \lambda(Y)$
are essentially self-adjoint. 

\sm 

{\sc Remarks.} a) We term `{\it unitary representation}' 
is commonly recognized, but it is strange since representations are not 
unitary. I think that the term {\it `$*$-representations'}
is better.

\sm

b) Let $Y\in\frg_\1$, $\lambda(Y)\ne 0$. Then 
$$i\lambda([Y,Y]_s)=2i \lambda(Y)^2$$
is selfadjoint and postive (semi)definite, so its spectrum is contained
oh a half-line. On the other hand, $[Y,Y]_s\in \frg_\0$.
Semiboundedness of spectra of generators of Lie algebra is a rare phenomenon,
 for reductive Lie algebras such spectra appear only
for highest weight
representations (this remark is contained at least in \cite{FrH}). For Lie superalgebras
there is a nice theory of highest weight unitary representations 
(see Kac \cite{Kac}, Furutsu, Nishiyama \cite{FuNi}, Jakobsen \cite{Jak}
with numerous continuations). This theory
is
parallel to  the theory of holomorphic seies of representations of semisimple Lie groups
(Harish-Chandra, F.~A.~Berezin, R.~Howe, N.~Wallach, M.~Vergne, S.~G.~Gindikin,
G.~I.~Olshanski,   et  al.). 
The present paper with \cite{Ner-super} is a continuation of these parallels.

\sm

c) The remaining part of theory of unitary representations
of semisimple groups does not have counterparts for Lie superalgebras
(as far as I know, no one can invent a definition, which allows something else, for instance, unitary principal series).
\hfill $\boxtimes$
 
\sm

{\bf \punct Highest weight modules over $\boldsymbol\ns$.}
Theory of Verma modules over Lie superalebras are 
parallel to the classical theory of semisimple Lie algebras,
see \cite{Mus}, Chapter 9. 

Consider the Neveu--Schwarz superagebra $\ns$, see Subsect. \ref{ss:NS}. 
Let $h$, $c\in \C$.
We say that a module  $V$ over $\ns$ is a
  {\it module with a highest weight $(h,c)$}
(see \cite{KW})
if there is a nonzero vector $v\in V$ such that

\sm

1) $L_{j} v=0$ for $j>0$, $M_{r} v=0$ for $r>0$;

\sm

2) $L_0 v= hv$, $\zeta v=cv$.

\sm

3) The vector $v$ is cyclic, i.e., $M$ has no proper submodules containing $v$.

\sm

This implies that vectors
of the form 
\begin{equation}
L_{-1}^{k_1} L_{-2}^{k_2}\dots M_{-1/2}^{\epsilon_{1/2}} M_{-3/2}^{\epsilon_{3/2}}\dots v,
\label{eq:LLLv}
\end{equation}
where $k_j\ge 0$, $\epsilon_r=0$, $1$ (only finite number of powers are nonzero) generate space $W$. A highest weight module
is graded by eigenvalues of $L_0$, the degree of a vector \eqref{eq:LLLv} is
$$h+\sum_{j>0} j k_j + \sum_{r>0} r \epsilon_r .$$
Dimensions of homogeneous subspaces are finite.
 The central element $\zeta$ acts as a multiplication by $c$.
 There is a unique module $M(h,c)=M_\ns(h,c)$ ({\it Verma module}) with highest weight $(h,c)$ 
 such that vectors \eqref{eq:LLLv} are linear independent, all other highest weight modules
 are quotients of the Verma module. Modules $M(h,c)$ are irreducible for all $(h,c)\in\C^2$
 except a countable family of quadrics, which was written by Kac \cite{Kac}.

For each $(h,c)$ there is a unique irreducible module
%\footnote{Here the usual arguments related to Verma modules (which also were used in Subsect. \ref{ss:virasoro})
%	 remain to be valid, see, e.g., \cite{Dix}, \cite{Mus}.}
 $L(h,c)=L_\ns(h,c)$ with highest weight $(h,c)$.
For  real $(h,c)$ there is a unique Hermitian form ({\it Shapovalov form}) $\la\cdot,\cdot\ra$ on $L(h,c)$ such that 
generators  satisfy the condition 
$$
L_k^*=L_{-k}, \qquad M^*_r=M_{-r}.
$$
This form is nondegenerate. We say that
 $L(h,c)$ is {\it unitary}, if the Shapovalov form is 
positive definite. Conditions of unitarity were obtained in
\cite{Fri}, \cite{GOK}. The domain of unitarity includes the angle
$h\ge 0$, $c\ge 3/2$ and a finite number of points $(h,c)$
on each level
 $$c=\frac 32\Bigl(1-\frac{8}{(N+2)(N+4)}\Bigr)$$
  ('discrete series'). For $h>0$, $c>3/2$ representations
$L(h,c)$ are nondegenerate in the following sense: vectors \eqref{eq:LLLv} are linear independent. On conditions
of nondegeneracy, see \cite{KW}, \cite{MR-C}.
For a unitary module $L(h,c)$ we consider its completion $\cH=\cH(h,c)$ with respect to the Shapovalov
form.

\sm

Restricting unitary module $L_\ns(h,c)$ to the Virasoro  algebra $\vir$ (whose generators are $L_\alpha$ and $\zeta$) we get a direct
sum of highest weight modules of $\vir$
of the form $L_\vir(h+j/2,c)$, multiplicities in this sum are finite.

\sm

{\bf \punct Integration of $\boldsymbol{L_\ns(h,c)}$%
\label{ss:integration}.}
Denote by $\cH=\cH(h,c)$ the completion of an unitarizable  module
$L_\ns(h,c)$ by the positive definite inner product.
Consider the  operator 
$$\cL:=1+L_0$$
 in 
$\cH$ defined on $L(h,c)$. Homogeneous subspaces of $L(h,c)$ are the eigenspaces of $\cL$, so
$\cL$ is essentially self-adjoint on the domain $L(h,c)$. Denote by $\cL$ its closure.
For $k\ge 0$ denote by $\cH_k$ the
set of $v\in L(h,c)$ such that
$\la \cL^N v,v\ra<\infty$.
We equip $\cH_k$ with the norm
\begin{equation}
 \| v\|^2_k=\la \cL^k v,v\ra,
 \label{eq:seminorm}
\end{equation}
so $\cH_{2k}$ is the domain of $\cL^k$.
We get a scale of Hilbert spaces,
\begin{equation}
\dots \supset \cH_{k-1}\supset  \cH_{k} \supset \cH_{k+1}\supset\dots 
\label{eq:cH-k} 
\end{equation}
 Denote
\begin{equation}
\cH_\infty:=\cap_{k=0}^\infty \cH_k
\label{ss:cH-infty}
\end{equation}
and equip it with the seminorms \eqref{eq:seminorm}.
Clearly, $\cH_\infty$ is a Fr\'echet space.

According Goodman, Wallach \cite{GW}, Corollary 2.2, the representation of 
$\vir$ in $L(h,c)$ lifts to a unitary projective representation
$\rho_{h,c}(g)$ of $\SDiff(S^1)$ in $\cH$. The subspace 
$\cH_\infty$ is invariant with respect to $\SDiff(S^1)$.
The representation of $\vir$ extends to a projective representation
of $\vect(S^1)$ in $\cH_\infty$ (the subspace $\cH$ is $\vect(S^1)$-invariant,
and we have a continuous representation of $\vect(S^1)$
in this space). For  $X\in \vect^\R$ operators $iX$  are essentially
self-adjoint on $\cH_\infty$.

\begin{theorem}
\label{th:2}
{\rm a)} The representation of $\ns$ in $L_\ns(h,c)$ extends to a continuous projective representation of the Lie superalgebra
algebra $\cont(S^{1|1}_\bullet)$ in $\cH_\infty$.

\sm

{\rm b)} For $M(b)= b(\phi) (\theta\partial_\phi-\partial_\theta)
\in \cont(S^{1|1}_\bullet)$   
with real $b(\phi)$, the operator $e^{i\pi/4}M(b)$ is essentially 
self-adjoint
on $\cH_\infty$.
\end{theorem}

This statement follows from  \cite{NeeS}, see Definition 7.2, Lemma 7.6,
Theorem 7.14, Corollary 4.17.

Therefore we get a representation of the group 
$\Cont_\vel^\sim(S^{1|1}_\bullet;\cA)$
 in the space  $\cH_\infty[\cA]$.

\begin{theorem}
	\label{th:3}
 The projective representation $\rho_{h,c}$ of $\SDiff(S^1)$ and the representation of $\Cont_\vel(S^{1|1}_\bullet;\cA)$ in $\cH(\cA)$
 determines a projective representation of the semidirect product 
 $\SCont(S^{1|1}_\bullet;\cA)$.
\end{theorem}

\section{Proofs of theorems \ref{th:2}, \ref{th:3}\label{s:proof}}

\COUNTERS

{\bf\punct Unitary highest weight representations of the Lie superalgebra 
$\boldsymbol{\osp(2|1)}$.} We consider the 5-dimensional Lie superalgebra $\fru_0$  
with even
generators
$L_{-1}$, $L_0$, $L_1$, odd generators $M_{1/2}$, $M_{-1/2}$, and relations
\begin{align}
 [L_0,L_{-1}]&=L_{-1},\, [L_0,L_1]=-L_1, \, [L_{-1},L_1]=L_{-2},
 \label{eq:osp21-1}
 \\
 [L_{-1}, M_{1/2}]&=-M_{-1/2}, \, [L_1,M_{-1/2}]=M_{1/2}, \\
 [L_0,M_{1/2}]&=-\tfrac12 M_{-1/2},\, [L_0,M_{-1/2}]=\tfrac 12 M_{1/2},\\
 [M_{1/2},M_{1/2}]&=2L_1,\, [M_{1/2},M_{-1/2}]=2L_0,\, [M_{-1/2},M_{-1/2}]=2L_{-1}.
 \label{eq:osp21-4}
\end{align}
It is a subalgebra in the Neveu--Schwarz algebra (see relations \eqref{eq:NS1}, \eqref{eq:NS2}). On the other hand, 
this Lie superagebra is isomorphic to  $\osp(2|1)$, on its representations, see, e.g., \cite{FH}. The even part is the Lie algebra $\frs\frl(2)$.

We present a description of its irreducible unitary highest
weight modules $V(\lambda)$. A highest weight module  is a  module with a cyclic vector $v$ satisfying 

\sm

--- $L_0 v=\lambda v$, where $\lambda\in\C$;

\sm

--- $L_{1} v=0$, $M_{1/2} v=0$.

\sm

The usual arguments on Verma modules (\cite{Dix}, \cite{Mus}) imply that an irreducible highest weight module is uniquely determined by $\lambda$.

A representation is '{\it unitary}' if it is equipped with a non-degenerate positive definite Hermitian form
$\la\cdot,\cdot\ra$ such that the following conditions of formal adjointness hold:
$$L_0^*=L_0,\qquad  L_{-1}^*=L_1, \qquad M^*_{1/2}=M_{-1/2}.$$
 Then $\lambda$ is real and $\ge 0$. If $\lambda=0$, then the module is one-dimensional. Other modules can be realized
in the space of polynomials in an even variable $z$ and an odd variable $\omega$: $\omega^2=0$, $\omega z=z\omega$. The generators act by
\begin{align*}
 L_1=\partial_z,\qquad L_0=z\partial_z+\lambda+ \tfrac12 \omega\partial_\omega, \qquad L_{-1} = z^2\partial_z+2\lambda z+z\omega\partial_\omega
 \\
 M_{1/2}=\partial_\omega +\omega \partial_z,\qquad M_{-1/2}=z\partial_\omega +\omega z\partial_z+2\lambda\omega.
\end{align*}
Vectors $z^n$, $\omega z^m$ form an orthogonal basis, their inner squares are
\begin{equation}
 \la z^n, z^n\ra=\frac{n!}{(2\lambda)_n}, \qquad \la \omega z^n, \omega z^n\ra= \frac{n!}{(2\lambda)_{n+1}}.
 \label{eq:squares}
\end{equation}
A verification of commutation relations and adjointness conditions is trivial. The restriction
of this representation to $\frs\frl(2)$ is a direct sum of two highest weight representations with bases $z^n$ and $\omega z^n$ 
respectively and with 
 highest vectors $1$, $\omega$.
 
Action of $M_{\pm 1/2}$ on basis vectors is given by
\begin{align}
 M_{1/2} z^n&=n\,\omega z^{n-1},&\qquad M_{1/2}\omega z^n&=z^n;
 \label{eq:1/2}
 \\
 M_{-1/2} z^n&=(n+2\lambda)\omega z^n ,&\qquad M_{-1/2}\omega z^n&= z^{n+1}.
 \label{eq:-1/2}
\end{align}

{\bf \punct Estimates of norms of $\boldsymbol{M_{\pm(p+1/2}}$.}
Let $A$ be an operator $\cH_\infty\to \cH_\infty$. Denote
$$\|A\|_{s\to t}:=\sup\limits_{\|v\|_s=1} \|Av\|_t
$$
(such a supremum can be infinite).

\begin{lemma}
\label{l:norms}
We have the following estimates for operators $M_{\pm(p+1/2}$, where $p>0$, in $\cH_\infty$:
\begin{align}
 \|M_{p+1/2}\|_{N+1\to N}&\le A(c,h,N) p;
 \label{eq:norm-1}
 \\
  \|M_{-p-1/2}\|_{N+1\to N}&\le B(c,h,N) p^{N/2+1},
  \label{eq:norm-2}
\end{align}
where $A(c,h,N)$, $B(c,h,N)$ are  constants independent on $p$.
\end{lemma}

{\sc Proof.}
1) {\it Reduction of the question to $\osp(2|1)$.}
 Fix $p=1$, 2, \dots. Denote
$$
\sigma(p)=\frac c6 \frac{p(p+1)}{2p+1}.
$$
Consider the 5-dimensional subalgebra $\frg^{(p)}$ in $\ns$ with generators
\begin{align*}
L^{(p)}_{\pm 1}&:=\frac{1}{2p+1} L_{\pm(2p+1)}, \qquad M_{\pm 1/2}^{(p)}:=\frac 1{(2p+1)^{1/2}} M_{\pm(p+1/2)},
\\
L^{(p)}_0&:= \frac{1}{2p+1} L_0+\sigma(p).
\end{align*}
It is easy to see that these generators satisfy to commutation relations
\eqref{eq:osp21-1}--\eqref{eq:osp21-4}.

Restricting our $L(h,c)$ to $\frg^{(p)}$, we get a countable direct sum of highest weight representations of 
$\osp(2|1)\simeq \frg^{(p)}$, their highest weights $\lambda$ have the form
\begin{equation}
\lambda=\sigma(p)+ \frac{h+\xi}{2p+1}, \qquad\text{where $\xi=0$,
$\frac12$, 1, $\frac32$, 2, $\dots$.}
\label{eq:lambda-in}
\end{equation}
So we need uniform estimates of norms $\|\cdot\|_{N+1\to N}$ of operators $M_{\pm 1/2}^{(p)}$ in $\frg^{(p)}$-modules $V(\lambda)$. 
The restriction of $\cL$ to such module has the form
$$
\cL=1+L_0=1+(2p+1)\bigl( L_0^{(p)}-\sigma(p)\bigr).
$$

We realize each module of this type as in the previous subsection

\sm

{\it Estimates of norms for $M_{1/2}^{(p)}$.} These generator
 sends
$$
\dots \to \C \omega z^n \to \C z^n\to \C \omega z^{n-1}\to\dots,
$$
so it is a weighted shift. One dimensional subspaces $\C z^n$, $\C\omega z^n$ are eigenspaces of $\cL$. 
Therefore, it is sufficient to obtain  upper estimates  for
$$
\frac{\|M_{1/2}^{(p)} z^n\|_N} {\| z^n\|_{N+1}},
 \qquad
 \frac{\|M_{1/2}^{(p)}\omega z^n\|_N}{ \|\omega z^n\|_{N+1}}.
 $$
These estimates must be uniform in 
$\xi$ (see formula \eqref{eq:lambda-in}) and $n$. 
 
Keeping in  mind \eqref{eq:1/2}, \eqref{eq:squares}, we get
\begin{multline*}
\frac {\|M_{1/2}^{(p)} z^n\|_N^2} {\|z^n\|_{N+1}^2}= \frac {n^2\, \|\omega z^{n-1}\|_N^2} {\| z^n\|_{N+1}^2}
=\\=
\frac{\bigl(1+(2p+1)(n+\lambda-\frac12 -\sigma(p)\bigr)^N}{\bigl(1+(2p+1)(n+\lambda- \sigma(p))^{N+1}} 
\cdot n^2 \cdot \frac{(n-1)!}{(2\lambda)_n} \cdot \frac {(2\lambda)_n}{n!}
=\\=
\Bigl(\frac{1+(2p+1)(n+\lambda-\frac12 -\sigma(p)}{1+(2p+1)(n+\lambda- \sigma(p))}\Bigr)^N
\cdot \frac{1}{2p+1}\cdot\frac {n}{n+ \lambda-\sigma(p)+\frac{1}{2p+1}}.
\end{multline*}
For $n=0$ we have zero. So, let $n\ge 1$.
By \eqref{eq:lambda-in}, the first factor and the third factor are $<1$ and the whole expression
is $\le 1/(2p+1)$.

Next, 
\begin{multline*}
 \frac {\|M_{1/2}^{(p)}\omega z^n\|_N^2} {\|\omega z^n\|_{N+1}^2}= 
  \frac {\| z^n\|_N^2} {\|\omega z^n\|_{N+1}^2}=
 \\=\frac{\bigl(1+(2p+1)(n+\lambda -\sigma(p)\bigr)^N}
 {\bigl(1+(2p+1)(n+\lambda+\frac12- \sigma(p))^{N+1}} 
% \frac {\|M_{1/2}^{(p)} z^n\|_N^2} {\|z^n\|_{N+1}^2}
 \cdot \frac{n!}{(2\lambda)_n}\cdot\frac{(2\lambda)_{n+1}}{n!}
=\\= 
 \Bigl(\frac{1+(2p+1)(n+\lambda -\sigma(p)}
 {1+(2p+1)(n+\lambda+\frac12- \sigma(p))}\Bigr)^N
 \cdot
 \frac{n+2\lambda}
 {1+(2p+1)(n+\lambda+\boxed{1/2}- \sigma(p))}.
\end{multline*}
The first factor $< 1$. Let us estimate the second factor.
Let us remove boxed $1/2$. This operation enlarges our expression.
Then we estimate
\begin{multline}
 \frac{n+2\lambda}
 {1+(2p+1)(n+\lambda- \sigma(p))}
 = \frac{n+ 2\bigl(\frac{h+\xi}{2p+1}+\sigma(p)\bigr)}
 {1+(2p+1)\bigl(n+\frac{h+\xi}{2p+1} \bigr) }
 =\\ 
 \frac{n}{1+(2p+1)\bigl(n+\frac{h+\xi}{2p+1} \bigr)\bigr] }
 +
 \frac{2\cdot\frac{h+\xi}{2p+1}}{1+(2p+1)\bigl(n+\frac{h+\xi}{2p+1} \bigr) }
 +
  \frac{\sigma(p)}{1+(2p+1)\bigl(n+\frac{h+\xi}{2p+1} \bigr) } 
  \le\\\le 
  \frac{n}{(2p+1)n}+\frac{2\cdot\frac{h+\xi}{2p+1}}{h+\xi}+
  \frac{\sigma(p)}{1}\le C\cdot p.
  \label{eq:estimates}
\end{multline}

%The last factor is monotone (decrease or increase) in the interval $n\in [0,+\infty)$.
%At $+\infty$ we have 1, for $n=0$ it is $>2\lambda/\lambda=2$. So the whole expression
%is $\le 2/(2p+1)$.

Since $M_{p+1/2}=(2p+1)^{1/2} M^{(p)}_{1/2}$, we get \eqref{eq:norm-1}.

\sm

{\it Estimates of norms for
 $M_{-1/2}^{(p)}$.}
\begin{multline*}
  \frac {\|M_{-1/2}^{(p)} z^n\|_N^2} {\|z^n\|_{N+1}^2}=  \frac {(n+2\lambda)^2\|\omega z^n\|_N^2} {\|z^n\|_{N+1}^2}
  =\\=
    \frac{\bigl(1+(2p+1)(n+\lambda+\frac12 -\sigma(p)\bigr)^N}
 {\bigl(1+(2p+1)(n+\lambda- \sigma(p))\bigr)^{N+1}}
 \cdot \frac{(n+2\lambda)^2\cdot n! \cdot (2\lambda)_{n}}
 {(2\lambda)_{n+1}\cdot n!}
   =\\= 
   \Bigl(\frac{1+(2p+1)(n+\lambda+\frac12 -\sigma(p)}
 {1+(2p+1)(n+\lambda- \sigma(p))}\Bigr)^N
%\Bigl(\frac{n+\lambda+\frac12 +\frac 1{2p+1}-\sigma(p)}
%{n+\lambda +\frac 1{2p+1}-\sigma(p)} \Bigr)^{N}
\cdot
 \frac{n+2\lambda}{{1+(2p+1)(n+\lambda- \sigma(p))} }.
\end{multline*}
The second factor was estimated in \eqref{eq:estimates}.
 The expression $S(n)$ in the big brackets $(\dots)^{N}$ decreases in $n$. So its maximum for fixed $\lambda$ is achieved 
at $n=0$, i.e., we have 
\begin{multline}
S(0)=
\frac{1+(2p+1)(n+\lambda+\frac12 -\sigma(p)}
 {1+(2p+1)(n+\lambda- \sigma(p))}
%\Bigl(\frac{n+\lambda+\frac12 +\frac 1{2p+1}-\sigma(p)}
%{n+\lambda +\frac 1{2p+1}-\sigma(p)}
%\frac{1/(2p+1)+\lambda+\frac12 -\sigma(p)}
% {1/(2p+1)+\lambda- \sigma(p)}
%\frac{\lambda+\frac12 +\frac 1{2p+1}-\sigma(p)}{\lambda +\frac 1{2p+1}-%\sigma(p)}
=
\frac{1+(2p+1)(\frac12 +\frac{h+\xi}{2p+1})}
 {1+(2p+1)\frac{h+\xi}{2p+1}} 
 \le \\\le
\frac{1+(2p+1)(\frac12 +\frac{h}{2p+1})}
 {1+(2p+1)\frac{h}{2p+1}} 
 =\frac{p+\frac32+h}{1+h}\le p\cdot \frac{h+\frac52}{h+1}, 
 \label{eq:estimates2}
\end{multline}
and we get an uniform estimate.
%This function in $\lambda$ is decreasing. The minimal possible value of $%\lambda$ is $\lambda=\sigma(p)+h/(2p+1)$.
%So, the expression in the big brackets is
%$$
%+(\dots)^N\le\biggl(\frac{\frac12 + \frac{h+A}{2p+1}}{\frac{h+A}{2p+1}} %\biggr)^N\le (2p+1)^N \Bigl(\frac{\frac12+h+1}{h+1}\Bigr)^N. 
%$$

It remains to  estimate
\begin{multline*}
  \frac {\|M_{-1/2}^{(p)}\omega z^n\|_N^2} {\|\omega z^n\|_{N+1}^2}=\frac {\| z^{n+1}\|_N^2} {\|\omega z^n\|_{N+1}^2}
=\\=
   \Bigl(\frac{1+(2p+1)(n+\lambda+1 -\sigma(p)}
 {1+(2p+1)(n+\lambda+\frac12- \sigma(p))}\Bigr)^N
%\Bigl(\frac{n+\lambda +1-\sigma(p)}{n+\lambda +\frac12+\frac 1{2p+1}-\sigma(p)} \Bigr)^{N}
\cdot \frac 1{2p+1}\cdot \frac{n+1}{n+\lambda +\frac12+\frac 1{2p+1}-\sigma(p) }.
\end{multline*}
The last factor is monotonic in $n\in[0,\infty)$. The value at $\infty$ is 1, the value at $n=0$ is 
$$
\frac{1}{\lambda +\frac12+\frac 1{2p+1}-\sigma(p) }
\le\frac{1}{1/2}  = 2.
$$
We estimate the first factor as in \eqref{eq:estimates2}
and come to the desired statement. 
%The factor $(\dots)^N$ decreasing as a function in $n$. For $n=0$ we have 
%$$
%\Bigl(\frac{\lambda +1-\sigma(p)}{\lambda +\frac12+\frac 1{2p+1}-\sigma(p)} %\Bigr)^{N}\le 2^N.
%$$
%This implies the estimate \eqref{eq:norm-2}.
\hfill $\square$

\sm

{\bf\punct Proof of Theorem \ref{th:2}.} %Denote
%by $\cH(N)$ the completion of $\cH$ with respect to a norm $\|\cdot\|_N$.
%For each $N$ we have natural embedding $\cH(N+1)\to \cH(N)$ with a dense image and $\cH_\infty=\cap \cH(N)$.
{\sc Statement a.}
Let 
$$b(\phi)=\sum c_k e^{i(2k+1)\phi}\in C^\infty_-(S^1),
$$
therefore, for each $N$ we have $c_k=o(k^{-N})$.

 By Lemma \ref{l:norms}, for each $N$
the series $ \sum c_k M_k$ uniformly converges as a series of operators $\cH_{N+1}\to\cH_N$ between Hilbert spaces. 
The sum continuously depends on $b\in C^\infty_-(S^1)$.
Therefore, this series %pointwise 
converges in $\cH_\infty$.  By the continuity, we get a continuous projective  representation of the Lie superalgebra 
$\cont(S^{1|1}_\bullet)$. 

\sm

{\sc Statement b.}
Next, let $b$ be real. Then for any $N$ the corresponding operator $e^{i\pi/4} M(b)$ is continuous as an operator $\cH_{N+1}\to \cH_{N}$, 
therefore it is continuous as an operator $\cH_{N+1}\to \cH_{N-1}$. The same holds for the commutator $[\cL, e^{i\pi/4} M(b)]=[L_0,e^{i\pi/4} M(b)]=i e^{i\pi/4} M(b')$.
Indeed, such a commutator has the same form with another
$b$: 
$$\bigl[i\partial_\phi, 
b(\phi)(\theta\partial_\phi-\partial_\theta)\bigr]=
b'(\phi)(\theta\partial_\phi-\partial_\theta).
$$

By  Nelson's commutator theorem, see \cite{RS}, Theorem X.36, we get self-adjointness of $e^{i\pi/4} M(b)$.

\sm

{\bf \punct Proof of Theorem \ref{th:3}.}
Let $Z$ range in  $\cont(S^{1|1}_\bullet)$, $g$ in $\SDiff^{(2)}(S^1)$. Denote by $Z\mapsto Z^{[g]}$ the natural action of  $\Diff^{(2})(S^1)$
on $\ns$. Denote operators of our representation corresponding to $Z$ by $\rho(Z)$ (instead of more usual $Z$).
  Our statement
reduces to the following lemma: 

\begin{lemma}
For any $g\in \SDiff(S^1)$, $Z\in \ns$,
the operators $\rho(g)$,  $\rho(Z)$ in the space   $\cH$ satisfy
the property 
 \begin{equation}
 \rho(g) \rho(Z) \rho(g)^{-1}=\rho(Z^{[g]})+\tau(g,Z),
 \label{eq:gZ}
 \end{equation}
  where $\tau(g,Z)$ is a scalar operator
\end{lemma}

{\sc Proof.}
Denote by $[\cH_\infty]_n\subset \cH_\infty$ and $[L(h,c)]_n\subset L(h,c)$ (finite-dimensional) subspaces
of vectors $w$ satisfying $Lw=(h+n)w$. By definition, $[\cH_\infty]_n\simeq [L_\ns(h,c)]_n$.

By Goodman, Wallach \cite{GW}, Theorem 4.2, the property \eqref{eq:gZ}
 holds for
$Z\in \cont_\0(S^{1|1}_\bullet)\simeq \vect(S^1)$. 
 So we must prove the similar statement for $Y$ ranging in $\cont_\1^\sim$.
The map $Z\mapsto \rho(Z^{[g]})$ is a unitary  representation of $\ns$.

Consider the  projective representation
$$
\rho^{[g]}(Z):=\rho(g)^{-1} \rho(Z^{[g]}) \rho(g)
$$
of $\cont(S^{1|1}_\bullet)$.
%For 
%$X\in \cont_\0^\sim(S^{1|1})=\vect(S^1)$
% we have $\rho^{(g)}(X)=\rho(X)+\tau(g,X)$.
Its restriction to $\vect(S^1)$
is the the same projective representation of $\vect(S^1)$.
 Therefore, 
 $\rho^{[g]}$ is a   representation
 of $\ns$ with highest weight $(h,c)$ and the same highest vector
$v$. In particular, $\rho^{(g)}$ is a representation of $\ns$. So, 
we have a map, say $\tau_g$, from
 $L(h,c)$ to $\cH_\infty$
defined by conditions 
\begin{multline*}
 L_{-1}^{k_1} L_{-2}^{k_2}\dots M_{-1/2}^{\epsilon_{1/2}} M_{-3/2}^{\epsilon_{3/2}}\dots v\mapsto
 \\
 \mapsto
 L_{-1}^{k_1} L_{-2}^{k_2}\dots \rho^{[g]}(M_{-1/2})^{\epsilon_{1/2}} \rho^{[g]}(M_{-3/2})^{\epsilon_{3/2}}\dots v.
\end{multline*}
It sends $[L(h,c)]_n\to [\cH_\infty]_n$. By the construction, it preserves the Shapovalov form. So we have a family of unitary
maps $[L(h,c)]_n\to [L(h,c)]_n$ and therefore $[\cH_\infty]_n \to [\cH_\infty]_n$. So, we have a unitary operator,
say $U(g)$, intertwining the representations $\rho_{h,c}^{[g]}$ and $\rho_{h,c}$, it is defined up to a scalar factor.
Therefore, for any $g_1$, $g_2\in \Diff^{(2)}(S^1)$, we have
$$
U(g_1) U(g_2)= c(g_1,g_2) U(g_1g_2), \text{where $c(g_1,g_2)$ is a scalar,}
$$
i.e., we get a projective unitary representation of $\Diff^{(2)}(S^1)$. The finite-dimensional subspaces
$[\cH_\infty]_n$ are $\Diff^{(2)}$-invariant. But $\Diff^{(2)}(S^1)$ has no nontrivial finite-dimensional representations%
\footnote{The group $\Diff^{(2)}(S^1)$ contains $\SL(2,\R)$. According the general Bargmann's theorem \cite{Bar} about projective unitary
representations of Lie groups, any projective unitary representation of $\SL(2,\R)$ is a linear representation of its universal covering
group $\SL(2,\R)^\sim$. Classification of unitary representations of $\SL(2,\R)^\sim$ is well known after Puka\'nszky \cite{Puk}. A unique finite-dimensional unitary representation is the trivial one-dimensional representation.}.
Therefore, $U(g)$
are scalar operators.
\hfill $\square$

%%%%%%%%%%%%%%%%%%%%%%%%%%%%%%%%%%%%%%%%%%%%%%%%%%%%%

%%%%%%%%%%%%%%%%%%%%%%%%%%%%%%%%%%%%%%%%%%%%%%%%%%%%%

%%%%%%%%%%%%%%%%%%%%%%%%%%%%%%%%%%%%%%%%%%%%%%%%%%%%%

%%%%%%%%%%%%%%%%%%%%%%%%%%%%%%%%%%%%%%%%%%%%%%%%%%%%%

%%%%%%%%%%%%%%%%%%%%%%%%%%%%%%%%%%%%%%%%%%%%%%%%%%%%%

%%%%%%%%%%%%%%%%%%%%%%%%%%%%%%%%%%%%%%%%%%%%%%%%%%%%%

\section{Affine relations and the Potapov transform%
	\label{s:affine-relations}}

\COUNTERS

For examination of the semigroup of super-annuli
we need some preliminaries on superlinear and superaffine relations.
{\it In this section, we discuss relations in  finite-dimensional
modules}, keeping in mind the space $\cW[\cA]$, which was
considered in Sect. \ref{s:embedding}.

\sm

{\bf\punct  Relations.} Let $X$, $Y$ be sets. 
A {\it relation} $H: X\tto  Y$
is the subset in $X\times Y$. For relations $H: X\tto  Y$,
 $G: Y\tto  Z$
we define their product $G\circ H: X\tto  Z$
as set of all $(x,z)\in X\times Z$, for which there exists $y\in Y$ such that $(x,y)\in H$,
$(y,z)\in Y\times Z$.

\sm 

{\sc Example.} For a map $\phi:X\to Y$ its graph $\graph(\phi)$ is a relation.
Product of maps corresponds to product of relations.
\hfill $\boxtimes$

\sm

{\bf\punct Relations in modules.} 
Let $B$ be a ring, $V$, $W$ be $B$-modules
(below $B=\C$ or $B=\cA$). A {\it $B$-relation} $V\tto W$ is a submodule 
in $V\oplus W$. 

For a $B$-relation $L:V\tto W$ we define the following $B$-modules.

\sm 

--- the {\it kernel} $\ker L\subset V$ is the intersection $V\cap L$;

\sm 

--- the {\it image} $\im L\subset W$ is the image of the projection $L\to W$;

\sm 

--- the {\it domain} $\im L\subset V$ is the image of the projection $L\to V$;

\sm 

--- the {\it indefinity} $\ker L\subset V$ is the intersection $W\cap L$;

\sm 

A product of $B$-relations is a $B$-relation.

%For  relations $L:V\tto W$,
%$M:W\tto Y$ we define their {\it  product} $M\circ L:V\tto Y$ as the set of all
%$v\oplus y$ such that there exists $w\in W$ such that $v\oplus w\in L$, $w\oplus y\in Y$.
%This product always is welld efined.

\sm

{\bf \punct Linear relations.}
Let $V$, $W$ be complex spaces.
A {\it linear relation} is a $\C$-relation between finite-dimensional complex linear spaces.  
Below we consider linear relations $L:V\tto W$ between topological vector spaces, then
$L$ is assumed to be closed.

\sm

%$L:V\tto W$ is a linear subspace in $V\oplus W$. If $V$, $W$ are topological vector spaces, then
%$L$ is assumed to be closed.
A product of linear relations $(L,M)\mapsto M\circ L$, where $L:V\tto W$, $M:W\tto Y$, 
is a continuous operation at points $(L,M)$
satisfying the {\it conditions of transversality}:
\begin{equation}
	\indef L\cap \ker M=0,\qquad \im L+\dom M=W.
	\label{eq:transversality}	
\end{equation}
see \cite{Ner-AA}, Lemma 2.2.

\sm 

{\bf \punct Affine relations.}	An {\it affine relation} $P:L\tto M$ is an affine subspace
in $L\oplus M$. We write such subspaces as $P=p+ L$, where $L:V\to W$ is a linear
relation and $p\in V\oplus W$ is a vector.
We say that $L$ is a {\it directrix} of $P$. 
It can be replaced by any vector $p'=p+ (v\oplus w)$.
where $v\oplus w\in L$. We write a vector $p$ in the form $p=p_V\oplus p_W\in V\oplus W$.

\begin{lemma}
	{\rm a)}  Under the  conditions of transversality \eqref{eq:transversality} 
	a product of $(p+L):V\tto W$ and  $(q+M):W\tto Y$ has the form $r+L\circ M$.
	
	\sm
	
	{\rm b)} 
	Precisely, we choose elements $l_W\in \Im L$, $m_W\in \dom M$ such that 
	\begin{equation}
		p_W+ l_W= q_W+ m_W.
		\label{eq:plqw}
	\end{equation}
	Choose $l_V\in V$ such that
	$l_V\oplus l_W\in L$ and  $m_Y\in Y$ such that $m_W\oplus m_Y\in M$. Then
	\begin{equation}
		r=r_V\oplus r_Y:=(p_V+l_V)\oplus (q_Y+ m_W).
		\label{eq:rr}
	\end{equation}
\end{lemma}

{\sc Proof.} By \eqref{eq:transversality} we can find $l_W$, $m_W$
satisfying \eqref{eq:plqw}. Set 
$$p':=(p_V+l_V)\oplus (p_W+ l_W),   \qquad q'=(q_W+ m_W)\oplus (q_Y+ m_Y).$$ 
Therefore  \eqref{eq:rr} is contained in $(q+M)\circ(p+L)$. 
Let $v\oplus y\in M\circ L$. We take $w\in W$ such $v\oplus w\in L$, $w\oplus y\in M$.
Then 
\begin{align*}
	p'+v\oplus w=(p'_V+v)\oplus(p'_W+w)  \in p+L;
	\\
	q'+w\oplus y=(q'_W+w)\oplus(q'_Y+y)\in q+M.
\end{align*} 
But $p'_W=q'_W$, $p'_V=r_V$, $q'_Y=r_Y$. Therefore, $r+M\circ L\subset (q+M)\circ(p+L)$. 
The inverse inclusion follows from inversion of our considerations.
\hfill $\boxtimes$

\sm

\sm

{\bf \punct Super-Grassmannians.} Let $V=V_\0\oplus V_\1$ be a  superlinear space of dimension $p|q$.
A subspace $L$ of dimension $r|s$ in $V[\cA]$ is a submodule generated by vectors 
$u_1$, \dots, $u_r\in V[\cA]_\0$ and $v_1$, \dots, $v_s\in V[\cA]_\1$ such that

\sm

--- $(u_1)_\downarrow$, \dots, $(u_r)_\downarrow\in V_\0$ are linear independent;

\sm 

--- $(v_1)_\downarrow$, \dots, $(v_s)_\downarrow\in V_\1$ are linear independent.

\sm

We denote the set of all such subspaces 
by $\Gr_{p|q}^{r|s}(\cA)=\Gr^{r|s}(V;\cA)$ ({\it super-Grassmannian}). The map $\pi_\downarrow:L\mapsto L_\downarrow$ sends $\Gr_{p|q}^{r|s}[\cA]$
to the product of the usual Grassmannians $\Gr_p^r(V_\0)\times \Gr_q^s(V_\1)$.
We also have
maps
$$
\pi_{\0\downarrow}: \Gr^{r|s}(V;\cA)\to \Gr^r(V_\0),
\qquad 
\pi_{\1\downarrow}: \Gr^{r|s}(V;\cA)\to \Gr^s(V_\1).
$$ 

\sm 

{\it Let $L\in \Gr^{r|s}(V;\cA)$, $\Gr^{\rho|\sigma}(V;\cA)$ satisfy the following conditions 
	of transversality}  
$$
\pi_{\0\downarrow} (L)+\pi_{\0\downarrow}(M)=V_\0,
\qquad
\pi_{\1\downarrow} (L)+\pi_{\1\downarrow}(M)=V_\1.
$$  
{\it Then}
$$
L\cap M\in \Gr_{p|q}^{r+\rho-p|s+\sigma-q}(V;\cA).
$$
See \cite{Ner-super}, Lemma 7.1.

\sm 

{\sc Remark.} Generally, an intersection of two subspaces is an $\cA$-submodule
but  it is not necessary a subspace.
\hfill $\boxtimes$

\sm

{\bf \punct Lagrangian super-Grassmannian.} Consider the superspace 
$V= \C^{2p|2q}$ and decompose it as
\begin{multline}
	V=V_\0\oplus V_\1=(V_\0^-\oplus V_\0^+)\oplus (V_\1^-\oplus V_\1^+)
	\simeq\\ \simeq \C^{2p}\oplus\C^{2q}
	=
	(\C^p\oplus\C^p)\oplus(\C^q\oplus \C^q).
	\label{eq:VVVV}
\end{multline}
Equip $V[\cA]$ with an orthosymplectic form $\Lambda_V(\cdot,\cdot)$ defined by the matrix
$$
\Lambda=
\left(
\begin{array}{cc|cc}
	0&1&0&0\\
	-1&0&0&0\\
	\hline 
	0&0&0&1\\
	0&0&1&0
\end{array}
\right).
$$

We save that a subspace $L\subset V[\cA]$ is {\it isotropic} if the form $\Lambda_V$ is zero on $L$.
A {\it Lagrangian subspace} is a maximal isotropic subspace. Equivalently,
it is an isotropic subspace of dimension $p|q$. 

Denote by $\Lagr(V;\cA)$ the {\it Lagrangian Grassmannian}, i.e., the set of all Lagrangian subspaces in $V$.
We have the map $\pi_\downarrow$ from $\Lagr(V;\cA)$
to the product of  Lagrangian Grassmannians
$\Lagr(V_\0)\times \Lagr(V_\1)$. The first factor is the usual Lagrangian Grassmannian
in $\C^{2p}$. The second factor is the set of maximal isotropic subspaces in the space $\C^{2q}$
equipped with the nondegenerate bilinear form 
$\begin{pmatrix}
	0&1\\1&0
\end{pmatrix}$. The second factor $\Lagr(V_\1)$ consists of two connected components
and the first factor is connected.
So the Grassmannian $\Lagr(V,\cA)$ consists of two components.

\sm 

Consider the following complementary Lagrangian subspaces 
$$
V^+=V_\0^+\oplus  V^+_\1, \qquad V^-=V_\0^- \oplus V^-_\1
$$
Consider an operator 
$$
S=\begin{pmatrix}
	A&B\\C&D
\end{pmatrix}:V^+[\cA]\to V^-[\cA],
$$
where $A$, $D$ are composed of elements $\in \cA_\0$ and $B$, $C$ are composed of elements $\in\cA_\1$.
Consider its graph  $\graph(S)$, which is a subspace in $V[\cA]$.
Then (see \cite{Ner-super}, Proposition 7.3)

\sm 

{\it  The subspace $\graph(S)$ is Lagrangian if and only if the matrix $S$ satisfies the following conditions}
$$
A=A^t, \qquad D=-D^t, \qquad B+C^t=0.
$$ 

\sm

Denote by $\cO(V,\cA)\subset \Lagr(V;\cA)$ the set consisting of subspaces $\graph(S)$.
Clearly, it is an open dense subset in one of connected components $\Lagr(V;\cA)$.

Consider the natural basis
$$
e_1,\dots,e_p,f_1,\dots,f_p\in V_\0,\qquad g_1,\dots,g_q,h_1,\dots,h_q\in V_\1.
$$
So 
$$
\Lambda_V(e_j,f_j)=1=-\Lambda_V(f_j,e_j),\qquad \Lambda_V(g_k,h_k)=1=\Lambda_V(h_k,g_k),
$$
other pairs of vectors are orthogonal.

We define the following operators $s_j$, $\sigma_k\in \OSp(V,\cA)$ of transposition of basis elements:
\begin{align}
	s_j^V\, e_j=-f_j,\quad s_j^V\, f_j=e_j;
	\label{eq:sef}
	\\
	\sigma^V_k\, g_k=h_k,\quad \sigma^V_k\, h_k=g_k,
	\label{eq:sigmasf}
\end{align}
other basis elements are fixed. Then the collection of charts 
\begin{multline*}
	s_{\alpha_1}^V\dots s_{\alpha_m}^V \sigma_{\beta_1}^V\dots \sigma_{\beta_n}^V \cO(V;\cA),
	\\
	\qquad \text{where $m\ge 0$, $n\ge 0$ and $1\le \alpha_1<\dots<\alpha_m\le p$, $1\le \beta_1<\dots <\beta_n\le q$,}
\end{multline*}
cover the whole Lagrangian Grassmannian $\Lagr(V;\cA)$.

\sm

{\bf \punct Superlinear relations.}
Let $V$, $W$ be superspaces. A superlinear relation $P:V[\cA]\tto W[\cA]$ is a subspace $P\subset V[\cA]\oplus W[\cA]$.

\sm 

{\sc Remark.}
The graph of a superlinear operator $V[\cA]\to W[\cA]$ is a superlinear relation
$V[\cA]\tto W[\cA]$.
\hfill $\boxtimes$

\sm

We say that super-linear relations $P:V[\cA]\tto W[\cA]$, $Q:W[\cA]\tto Y[\cA]$ are {\it transversal}
if

\sm 

--- $\pi_{\0\downarrow}(P):V_\0\tto W_\0$ is transversal to $\pi_{\0\downarrow}(Q):W_\0\tto Y_\0$;

\sm 

---  $\pi_{\1\downarrow}(P):V_\1\tto W_\1$ is transversal to $\pi_{\1\downarrow}(Q):W_\1\tto Y_\1$.

\sm

{\it Under the conditions of transversality a product of relations $P$, $Q$ is a superlinear relation and}
\begin{equation}
	\dim Q\circ P=\dim Q+\dim P-\dim W.
	\label{eq:dimQP}
\end{equation}
See \cite{Ner-super}, Theorem 8.2.

%\sm 
%
%{\sc Remark.} Without conditions of transversality
%a product of superlinear relations can be not a superlinear relation.
%\hfill $\square$

\sm

--- {\it  Under the conditions of transversality a product of Lagrangian superlinear relations
	$L:V[\cA]\tto W[\cA]$, $M:W[\cA]\tto Y[\cA]$ 
	is a Lagrangian lenear relation} see \cite{Ner-super},
	 Theorem 8.4.

\sm 

{\bf \punct  Superaffine relations.}
We define a {\it super-affine relation} $P:V[\cA]\tto W[\cA]$ as a
relation of the form $\xi+ P$, where $L$ (the {\it directrix} of $P$) is
a
superlinear relation $V[\cA]\tto W[\cA]$ and 
$p\in (V[\cA]\oplus W[\cA])_\0$. If directrices
of $P:V[\cA]\tto W[\cA]$, $Q:W[\cA]\tto Y[\cA]$ are transversal,
then a product $Q\circ P$ is a superaffine relation,  its  dimension is
\eqref{eq:dimQP}.

\sm 

{\bf\punct The Potapov transform of superlinear relations.}
For each superspace $V$ we fix
decomposition $V_\0$, $V_\1$ into direct sums
\begin{equation}
	V_\0=V_\0^+\oplus V_\0^-,\qquad V_\1=V_\1^+\oplus V_\1^-
	\label{eq:VVplus}
\end{equation}
Set
$$
V^+:=V_\0^+\oplus V_\1^+,\qquad V^-:=V_\0^-\oplus V_\1^-.
$$
Denote by $V^\diamond$ the space $V$ equipped with this decomposition.

Consider an operator 
\begin{multline}
	\Sigma:=\begin{pmatrix}A&B\\C&D\end{pmatrix}
	= \begin{pmatrix}
		A_{11}&A_{12}&B_{11}&B_{12}\\
		A_{21}&A_{22}&B_{21}&B_{22}\\	
		C_{11}&C_{12}&D_{11}&D_{12}\\
		C_{21}&C_{22}&D_{21}&D_{22}
	\end{pmatrix}:\\
	V^+[\cA]\oplus W^-[\cA]\to V^-[\cA]\oplus W^+[\cA]
	\label{eq:potapov}
\end{multline}
such that blocks with superscripts $11$, $22$ are composed of elements $\cA_\0$ and blocks 
with superscripts $12$ and $21$ are composed of elements of $\cA_\1$.

Let $L:V[\cA]\tto W[\cA]$  be the graph of $S$; by the definition,
we have
\begin{equation}
	\dim L=(\dim V_\0^+ +\dim W_\0^-)\,\bigr|\, (\dim V_\1^+ +\dim W_\1^-).
	\label{eq:dim}
\end{equation}
Then we say that $\Sigma=\Sigma[L]$ is the 
{\it Potapov tranform} of $L$.  
Denote by $\cO(V^\diamond,W^\diamond)$ the set of superlinear relations $V[\cA]\tto W[\cA]$,
for which a Potapov transform is well defined. It is an open dense subset in 
in the set of superlinear relations of dimension \eqref{eq:dim}

\sm

--- {\it Let $L\in \cO(V^\diamond,W^\diamond)$, $M\in \cA(W^\diamond,Y^\diamond)$.
	Let
	$$
	\Sigma[L]=\begin{pmatrix}
		A&B\\C&D
	\end{pmatrix}, \qquad \Sigma[M]=\begin{pmatrix}
		P&Q\\R&T
	\end{pmatrix}
	$$
	Let   $(1-PD)$, $(1-DP)$ be  invertible.
	Then $L$, $M$ are transversal, $\Sigma(M\circ L)$ is well defined
	and is given by the formula
	\begin{multline}
		\Sigma(M\circ L)= \begin{pmatrix}
			A&B\\C&D
		\end{pmatrix} \odot 	\begin{pmatrix}
			P&Q\\R&T
		\end{pmatrix}
		= \\ =
		\begin{pmatrix}
			A+B(1-PD)^{-1}PC& B(1-PD)^{-1}Q\\
			R(1-DP)^{-1} C& T+RD(1-PD)^{-1}Q
		\end{pmatrix}
		\label{eq:PiPi}
	\end{multline}
} See Theorem IV.3.3 in \cite{Ner-book}.

\sm

{\sc Remark.} 
Notice that invertibility of $(1-PD)$ is equivalent to the invertibility of 
$(1-P_\downarrow D_\downarrow)$. For the finite-dimensional case (which is discussed now)
inveribilities of $(1-PD)$ and $(1-DP)$ are equivalent. For an infinite-dimensional
case this is not so, but in the situation discussed below
operators $P_\downarrow$, $D_\downarrow$ are compact in the sense of
Hilbert spaces.
\hfill $\boxtimes$

\sm 

--- {\it  Let $g\in \begin{pmatrix}
		a&b\\c&d
	\end{pmatrix}$
	be an operator in $V^\diamond[\cA]$. Consider its graph
	$\graph(g)$.
	Let $a$ be invertible, then}
\begin{equation}
	\Sigma(\graph(g))=\begin{pmatrix}
		a^{-1}&-a^{-1}b\\ca^{-1}&d-ca^{-1}b
	\end{pmatrix}.
	\label{eq:Pi-matritsa}
\end{equation}

\sm

We also consider Potapov transforms for affine relations.
Namely, we consider
affine maps 	
$V^+[\cA]\oplus W^-[\cA]\to V^-[\cA]\oplus W^+[\cA]$,
then their graphs are affine relation.

\sm 

{\it 
Let $\wt L:V[\cA]\tto W[\cA]$, $\wt M:W[\cA]\tto Y[\cA]$
be affine relations, let they have Potapov transforms,
say
\begin{align}
\begin{pmatrix}
v_-& w_+
\end{pmatrix}=
\begin{pmatrix}
v_+& w_-
\end{pmatrix}
\begin{pmatrix}
A&B\\C&D\end{pmatrix}+
\begin{pmatrix}\gamma&\delta
\end{pmatrix};
\label{eq:pot1}
\\
\begin{pmatrix}
w_-& y_+
\end{pmatrix}=
\begin{pmatrix}
w_+& y_-
\end{pmatrix}
\begin{pmatrix}
P&Q\\R&T
\end{pmatrix}+
\begin{pmatrix}
\pi&\kappa
\end{pmatrix}
\label{eq:pot2}
\end{align}
If $(1-DP)$ is invertible, then
\begin{multline}
\begin{pmatrix}
v_-& y_+
\end{pmatrix}=\begin{pmatrix}
v_+& y_-
\end{pmatrix} \left[ \begin{pmatrix}
			A&B\\C&D
		\end{pmatrix} \odot 	\begin{pmatrix}
			P&Q\\R&T
		\end{pmatrix}\right]
		+\\+
		\begin{pmatrix}(\delta P+\pi)(1-DP)^{-1}C+\gamma
		&(\pi D+\delta) (1-RD)^{-1}Q+\kappa\vphantom{\Bigl|}
		\end{pmatrix},
		\label{eq:affine-product}
\end{multline}
where $\odot$-product is defined by \eqref{eq:PiPi}.}

\sm

Proof is straightforward, we eliminate $w_+$ and $w_-$ from 
equations \eqref{eq:pot1}--\eqref{eq:pot2}.

\sm 

{\bf \punct Lagrangian superlinear relations.}
Consider superspaces $V$, $W$ equiped with orthosymplectic  forms
$\Lambda_V(\cdot,\cdot)$, $\Lambda_W(\cdot,\cdot)$. 
Equip $V\oplus W$ with the orthosymplectic form
$$
\Lambda_{V\ominus W}(v\oplus w, v'\oplus w')=\Lambda_V(v,v')-\Lambda_W(w,w').
$$
We say that a {\it super-linear relation $V[\cA]\tto W[\cA]$ is Lagrangian} if 
it is  Lagrangian in $V[\cA]\oplus W[\cA]$ with respect to $\Lambda_{V\ominus W}$.

\sm

--- {\it If Lagrangian relations $P:V[\cA]\tto W[\cA]$, $Q:W[\cA]\tto Y[\cA]$
	are transversal then their product is Lagrangian,} see \cite{Ner-super}, see \cite{Ner-super}, Theorem 8.4.

\sm

For each orthosymplectic space $V$ we fix the decomposition \eqref{eq:VVVV}.
Potapov transform \eqref{eq:potapov}
 of a Lagrangian relation satisfies natural conditions of symmetry,
 $A_{11}^t=A_{11}^t$, $D_{11}^t=D_{11}^t$, $A_{22}^t=-A_{22}^t$,
 $D_{22}^t=-D_{22}^t$,  
 $A_{12}^t=-A_{21}^t$, $B_{12}^t=C_{21}^t$, $B_{21}^t=-C_{12}$

\sm

We say that an {\it affine Lagrangian relation} is an affine Lagrangian subspace
shifted by an even element of $V[\cA]\oplus W[\cA]$.
In this case we define {\it Potapov transform} is an affine map
$$
V^+[\cA]\oplus W^-[\cA]\to V^-[\cA]\oplus W^+[\cA].
$$

\sm

As above denote by $\cO^{\aff}(V^\diamond,W^\diamond;\cA)$ the set of affine superlinear relations having well-defined Potapov transforms.
The whole Lagrangian Grassmannian is the union of charts 
$$
s_{\alpha_1}^W\dots s_{\alpha_m}^W \sigma_{\beta_1}^W\dots \sigma_{\beta_n}^W 
L
s_{\gamma_1}^V\dots s_{\gamma_k}^V \sigma_{\delta_1}^V\dots \sigma_{\delta_l}^V,
\quad\text{ where $L$ ranges in  $\cO^\aff(V^\diamond,W^\diamond;\cA)$,} 
$$
and 
$\alpha_1<\dots<\alpha_m$, $\beta_1<\dots<\beta_n$, $\gamma_1<\dots<\gamma_k$,
$\delta_1<\dots<\delta_l$ are fixed. The involutions $s$ and $\sigma$ are
defined by \eqref{eq:sef}--\eqref{eq:sigmasf}. 

%%%%%%%%%%%%%%%%%%%%%%%%%%%%%%%%%%%%%%%%%%%%%%%%%%%%%

\section{The semigroup of contact superannuli\label{s:annuli}}

\COUNTERS

{\bf \punct The semigroup $\boldsymbol{\Gamma}$ of annuli.%
\label{ss:Gamma}} We intend superize the following construction. 
Consider a one-dimensional compact complex manifold $P$
with boundary
 homeomorphic to an annuli (we call such manifolds '{\it abstract annuli'}). 
 Any abstract annulus is conformally equivalent to some {\it standard annulus} $A\le |z|\le B$,
all biholomorphic automorphisms of the standard annulus
have the form
\begin{equation}
z\mapsto e^{i\theta} z,\qquad
\text{or} \qquad z\mapsto AB e^{i\theta} z^{-1}.
\label{eq:automorphisms}
\end{equation}
 They are analytic on the boundary, 
 so there are canonical structures of  oriented real  analytic %smooth 
 one-dimensional manifolds  on components of the boundary
 of the surface $P$.
 
 Recall that $B/A>1$ is the unique conformal invariant of an annulus.
 
 \sm 
 
 The rotation  of the standard annulus of $\phi$ is a unique biholomorphic
 involution sending  each component of the boundary to itself. So, this 
 transformation is canonically defined to any abstract annulus.
 In particular, we have a canonically defined rotation $\kappa$ of
 180 degrees.

\sm 
 
 We define the semigroup $\Gamma$ (the complexification of the group
 $\SDiff(S^1)$)
 of diffeomorphisms of the circle, see \cite{Ner-semigroup}, 
 \cite{Ner-holom}, \cite{Seg2}) in the following way.
 An element of $\Gamma$ is a triple $\cP=(P,p_-,p_-)$,
 where: 
 
 \sm 
 
 --- $P$ is a one-dimensional complex manifold homeomorphic to an annuli, we say that one component $(\partial P)_+$ of
 the boundary is an {\it entry} and another component $(\partial P)_-$ is an
  {\it exit};
 
 \sm 
 
 --- $p_+$ and $p_-$ are smooth parametrizations $S^1\to P$ of 
 $(\partial P)_+$ and $(\partial P)_-$ such that $P$
  lies on the left side from the path $p_+$ and the right  side  side from $p_-$.

\sm

Two elements $\cP=(P,p_-,p_-)$, $\cP'=(P',p'_-,p_-)$ coincide iff there exists a biholomorphic map $h:P\to P'$
such that $p'_\pm=h\circ p_\pm$.

For two elements $\cP=(P,p_-,p_-)$, $\cQ=(Q,q_-,q_-)\in\Gamma$ we define their product $\cQ \cP$
by gluing ({\it welding}) of the exit $(\partial P)_+$ and the entry $(\partial Q)_-$.
Precisely, for each $\phi\in S^1$ we identify $p_-(\phi)\in P$ with $q_+(\phi)\in Q$.
We extend the complex structure to the curve of gluing%
\footnote{Let us glue two annuli by a homomorphism $\gamma$. A complex structure
admits an extension under very weak conditions (see Bers \cite{Bers}),
for smooth $\gamma$ this is more-or-less obvious, see \cite{Ner-book}, Subsect, VII.4.4, VII.6.4.}
 and
 get a one-dimensional complex manifold with parametrized entry $p_+(\phi)$ and exit $q_-(\phi)$).

Any irreducible highest weight representation of the Virasoro algebra admits a canonical
integration to a representation of the semigroup $\Gamma$, see \cite{Ner-holom}, \cite{Ner-fermion}.

\sm

{\bf\punct   Contact super-annuli.} We say that a {\it super-annulus} $P[\theta]$
is  an annulus $P$  on a complex plane  equipped with  an additional Grassmann coordinate $\theta$,
$\theta^2=1$, commuting with $\fra_j$.
We allow changes of variables of the type \eqref{eq:diff-1}--\eqref{eq:diff-2}.

 Consider an even 1-form
\begin{equation}
U(z;\fra)\,dz+ V(z;\fra)\,\theta\,d\theta,
\label{eq:UzVtheta}
\end{equation}
there $U$, $W$ are even in  $\fra$. We say that it is non-degenerate
if the $\C$-valued functions $U_\downarrow$, $V_\downarrow$ are non-vanishing. A {\it contact form}%
\footnote{A general definition (see \cite{Schw}). Let $x_1$, \dots, $x_{2p+1}$ be even variables, $\theta_1$, \dots, $\theta_q$
be odd variables. Let $\ell$ be an even one-form in $x$, $\theta$. Let $t>0$ (or $t\ne 0$) be an additional even variable.
We say that $\ell$ is a {\it contact form} if $\omega:=d(t\,\ell)$  is a symplectic form, i.e., its matrix is nondegenerate
and $d\omega=0$ (in our case $d\omega=dd(t\ell)=0$  automatically).
} 
is a nondegenerate 1-form defined up to a multiplication by a function $A(z;\fra)$, which is even in  $\fra$.

To be definite, assume that the annulus surrounds the point $z=0$. The first statement of the following lemma is well known,
but we need its proof.

\begin{lemma}
	\label{l:canonical}
{\rm a)} Locally any contact form can be reduced
to the expression 
\begin{equation}
\frac{dz}{iz}+ \theta d\theta.
\label{eq:canonical}
\end{equation}

{\rm b)}	Any contact transformation preserving the form \eqref{eq:canonical}
and fixing the coordinate $z$ 
is
$(z,\theta)\mapsto (z,\theta)$ or $(z,\theta)\mapsto (z,-\theta)$.
%
%Let us allow to change  the variable $\theta$ fixing $z$ (as \ref{eq:diff-2}), 
%$$
%\theta\mapsto R(z;\fra) \theta + r(z;\fra)
%$$
\end{lemma}

{\sc Proof.}
a)
Multiplying \eqref{eq:UzVtheta} by $A=(iUz)^{-1}$ we get an expression
\begin{equation}
	\frac{dz}{iz}+ W(z;\fra)\,\theta\,d\theta.
	\label{eq:zW}
\end{equation}
Applying a transformation
$
\theta\mapsto R(z;\fra) \theta
$, 
we come to
$$
\frac{dz}{iz}
+ W(z;\fra) R(z;\fra)^2 \,\theta\,d\theta
$$
and choose 
\begin{equation}
R=W^{-1/2}.
\label{eq:root}
\end{equation}

b) So, we consider transformations \eqref{eq:diff-1}--\eqref{eq:diff-2} with $Q(z;\fra)=z$,
$q(z,\fra)=0$. The conditions \eqref{eq:solution}  reduce to
$$
R:=(1+\partial_z r\cdot r)^{1/2},\qquad 0=-r (1+\partial_z r\cdot r)^{1/2}.
$$
But $(1+\partial_z r\cdot r)^{1/2}=1+\frac 12 \partial_z r\cdot r+\dots\ne 0$, therefore
$r=0$.
\hfill $\square$

\sm 

For a contact structure on the whole supper-annulus we have two possibilities:

\sm

1. The square root \eqref{eq:root} is a single-valued function on $P$. Then a contact form is reduced 
to \eqref{eq:canonical}.

\sm 

2. The square root has two branches on $P$ (actually, we are interested in this case). 

\sm 

We have two equivalent ways to describe the contact manifold in the case 2.

\sm 

One way is to consider the M\"obious bundle over the annulus; our supermanifold is the
bundle of Grassmann algebras on fibres.

\sm 

We prefer another language.
%Then we cover the annulus $P$
% by two simply connected domains. In both domains we write the contact structure in the form \eqref{eq:canonical} but the transition 
%maps on intersection of these domains are $(z,\theta)\mapsto (z,-\theta)$.
We  consider
the 
two-sheeted covering $P^\sim$ of $P$, it is equipped with the automorphism $\pi$ of permutation of branches. Next consider
the super-annulus $\wh P$ equipped with the standard contact form and the automorphism  $\wh \pi$
sending $(u,\theta)\mapsto (\pi(u),-\theta)$.

\sm

{\bf\punct     The semigroup $\boldsymbol{\Gamma(\cA)}$  of super-annuli.}
In this  subsection we describe supersemigroup
corresponding to the Ramond
superalgebra, in the next subsection to the  Neveu--Schwarz superalgebra.

An element  $\wh \cP=(\wh P,\frp,\wh p_+, \wh p_-)$
 of $\Gamma(\cA)$
  is an element 
 $\cP=(P,p_+,p_-)\in \Gamma$ equipped with the following additional structures:

\sm

--- $\wh P$ is a one-dimensional complex manifold $P$ equivalent to an annulus equipped with
an additional Grassmann coordinate $\theta$;

\sm

--- $\frp$ is a contact structure on $\wh P$ equivalent to
 $\frac{dz}{iz}+\theta\, d\theta$ smooth up to the boundary of $P$;

\sm

---  restrictions $\frp_\pm$ of the contact structure $\frp$ to
$(\partial P)_\pm$
 determine structures of $(1|1)$-dimensional contact
 manifolds, say $(\partial \wh P)_\pm[\theta]$.
We consider the supercircle $S^{1|1}$ equipped with the standard
 contact form $d\phi+\theta\,d\theta$.
 Then  $p_\pm$  are $\cA$-contactomorphisms
$S^{1|1}\to (\partial \wh P)_\pm[\theta]$.

% $p_\pm$ are contactomorphisms from the supercircle $S^{1|1}$ equipped with  the form $d\phi+\theta d\theta$
%to components of the boundary of $P$ equipped with restrictions of the contact structure $\frp$.
%  such that $P$ lies on the left side from $p_+$ and right side from $p_-$.
 
 \sm 
 
Two elements $\wh \cP=(\wh P,\frp,p_+, p_-)$, 
$\wh \cP'=(\wh P',\frp',p'_+, p'_-)$ 
coincide if there is a contactomorphism $h:(\wh P,\frp)\to (\wh P',\frp')$
such that $h\circ p_\pm=p'_\pm$. 

\sm 

Notice that an element $\cP$ determines a trivial bundle over the annulus whose fiber is the algebra $\C(\theta)\otimes \cA$.
The contactomorphisms $p_\pm$ identify fibers over points of the boundaries with fibers over points of $S^1$.
Now, for $\cP=(\wh P,\frp,p_+, p_-)$, $\cQ=(\wh Q,\frk,q_+, q_-)$ we define their product
gluing annuli as above and the corresponding bundles fiber-wise.

\sm 

{\bf\punct The supersemigroup   $\boldsymbol{\Gamma_\bullet(\cA)}$.}
Consider an element $\cQ=(\wh Q,\frq,q_+, q_-)\in \Gamma(\cA)$. 
As we have seen in Subsect. \ref{ss:Gamma}, the annulus $Q$ admits a unique
holomorphic involution $\kappa$, $\kappa^2=1$ preserving the entry and exit.
By Lemma \ref{l:canonical}, the map 
\begin{equation}
\iota:(z,\theta)\mapsto (z,-\theta)
\label{eq:kappa-iota}
\end{equation}
is canonically defined 
for an superannulus equipped with a contact structure. Denote
by $\frq$ the composition of involutions $\kappa$ and $\iota$.

On the other hand, we have an involutive element 
$$
\sigma: (\phi,\theta)\mapsto (\phi+\pi,-\theta)\in \SCont(S^{1|1};\cA),
$$ 
 see  \eqref{eq:cont-bullet}.
 
The supersemigroup $\Gamma_\bullet(\cA)$
consists of elements  $(\wh Q,\frq,q_+, q_-)\in \Gamma(\cA)$
satisfying the condition
$$
\kappa\circ \frq_\pm=\frq_\pm \circ \sigma.
$$

%%%%%%%%%%%%%%%%%%%%%%%%%%%%%%%%%%%%%%%%%%%%%%%%%%%%%%%%%%%
%%%%%%%%%%%%%%%%%%%%%%%%%%%%%%%%%%%%%%%%%%%%%%%%%%%%%%%%%%%
%%%%%%%%%%%%%%%%%%%%%%%%%%%%%%%%%%%%%%%%%%%%%%%%%%%%%%%%%%%
%%%%%%%%%%%%%%%%%%%%%%%%%%%%%%%%%%%%%%%%%%%%%%%%%%%%%%%%%%%
%%%%%%%%%%%%%%%%%%%%%%%%%%%%%%%%%%%%%%%%%%%%%%%%%%%%%%%%%%%

\sm

{\bf \punct Representations of the semigroup $\boldsymbol{\Gamma_\bullet(\cA)}$.} Consider a standard ring $C_t:e^{-t}\le |z|\le 1$ with standard
contact form $dz+\theta d\theta$ and the involutive contactomorphism
\eqref{eq:kappa-iota}. Consider  $\cA$-contactomorphisms 
$c_\pm: S^{1|1}_\bullet\to (\partial P)_\pm$ defined by
$$c_+:(e^{i\phi},\theta)\to (e^{-t}e^{i\phi},\theta),
 \qquad  c_-:(e^{i\phi},\theta)\to (e^{i\phi},\theta).$$
This determines an element  
$\cC_t=(C_t,\frc, c_\pm)\in \Gamma_\bullet(\cA)$.

Let $\wh\cP=(P,\frp,\pi_\pm)\in\Gamma_\bullet(\cA)$. Without loss of generality we can assume that $P$ is an annulus $C_t$.
Then we have   $\cA$-contactomorphisms $p_\pm$
 from
$S^{1|1} \to (\partial C)_\pm$.
We can think that a contactomorphism of the circle is a limit element of $\Gamma_\bullet(\cA)$ corresponding to infinitely narrow 
annulus.

So we can decompose any element of $\Gamma_\bullet(\cA)$ as a product 
\begin{equation}
\wh\cP=
p_-\circ \cC_t\circ p_+
\label{eq:canonical-decomposition}
\end{equation}
 of a standard annulus and two 
$\cA$-contactomorphisms of the supercircle $S^{1|1}_\bullet$. 

\begin{conjecture}
Consider a highest weight representation $\rho_{h,c}$
 of $\ns$ and the corresponding space $\cH_\infty[\cA]$. Assign:
 
\sm 
 
 --- to any $ g\in \Diff^{(2)}(S^1)$ the operator $\rho_{h,c}(g)$,
 
 \sm 
 
 --- to $1+\mu Z$, where $Z\in \ns_{\ov j}$ and
 $\mu\in\cM_+$ {\rm(}$p(\mu{\rm)}=\ov j${\rm)}, the operator $1 +\mu \rho_{h,c}(Z)$
 
 \sm 
 
 --- to $\cC_t$ the operator $\exp\{t L_0\}$.
 
 \sm
 
 Then we get a projective representation of the semigroup 
 $\Gamma_\bullet(\cA)$.
\end{conjecture} 

\begin{theorem}
	\label{th:integration-gamma}
The statement above holds for the representation of $\ns$ defined in Subsect.  \eqref{ss:NS}
by \eqref{eq:fock-L}--\eqref{eq:fock-M} with $\mu$, $\nu\in\R$. This means that modules $L_\ns(h,c)$
for $c\ge \frac32$, $h\ge \frac1{24}(c-\frac32)$ can be integrated to $\Gamma_\bullet(\cA)$.
\end{theorem}

The proof of  Theorem \ref{th:integration-gamma} is contained in   Sect. 
\ref{s:last}.

%%%%%%%%%%%%%%%%%%%%%%%%%%%%%%%%%%%%%%%%%%%%%%%%%%%%%

%%%%%%%%%%%%%%%%%%%%%%%%%%%%%%%%%%%%%%%%%%%%%%%%%%%%%

%%%%%%%%%%%%%%%%%%%%%%%%%%%%%%%%%%%%%%%%%%%%%%%%%%%%%

%%%%%%%%%%%%%%%%%%%%%%%%%%%%%%%%%%%%%%%%%%%%%%%%%%%%%

%%%%%%%%%%%%%%%%%%%%%%%%%%%%%%%%%%%%%%%%%%%%%%%%%%%%%

\section{Logarithmic densities on super-annuli%
\label{s:logarithmic}} 

\COUNTERS

{\bf \punct  Logarithmic densities on superannuli.%
	\label{ss:logarithmic}} Fix $\mu$, $\nu\in \C$.
Let $U\subset\C$ be a topological annulus
 surrounding 0 (which is not contained in $U$).
We say that a logarithmic density of type $(\mu,\nu)$ in $U$ is a formal expression
$$
f(z,\theta)-i \mu \ln (z)+ \nu \ln(dz\,d\theta)
$$ 
defined up an addition of a complex valued constant function.
Under a contact  change of variables 
$$
z=\bfQ(u,\xi;\fra)=Q(u;\fra)+ q(u;\fra)\xi , \qquad 
\theta=\bfQ(R,\xi;\fra)=R(u,\fra)\xi+r (u;\fra),
$$  
see \eqref{eq:diff-1}-\eqref{eq:diff-2}, the expression 
changes to
\begin{multline}
	\Bigl(f\bigl((Q(u;\fra),q(u;\fra)\xi\bigr)
	- i \mu\ln(Q(u;\fra)/u)+
	\nu \ln \Ber 
	\begin{pmatrix}
		\partial_u \bfQ&\partial_\theta \bfQ\\
		 \partial_u \bfR&\partial_\theta \bfR
	\end{pmatrix}\Bigr)
	- \\- i \mu \ln (u)+ \nu \ln(dz\,d\theta)
	\label{eq:annulus-logarithmic}
\end{multline}
(according the usual manipulations with symbols $d$).

\sm
 
 {\bf\punct Logarithmic densities and affine relations.}
 Consider an element $\cQ=(\wh Q, \frq, q_+,q_-)\in \Gamma_\bullet(\cA)$. 
 We wish to construct an affine relation $\Delta_{\mu,\nu}(\cQ)$ between
 two copies of the space 
 $\cW[\cA]$.
 %\simeq 
 %C_-^\infty(S^1;\cA)\oplus C_+^\infty(S^1;\cA)$.
 
 Consider a holomorphic logarithmic density $F$ of weight $(\mu,\nu)$ on the contact ring $(\wh Q, \frq)$,
 which is smooth up to a boundary. We consider its inverse images
 $$
 f_\pm(\phi,\theta)+\frac{1}{2\pi}\phi+ \mu\ln (d\phi\,d\theta)
 $$ 
 under the $\cA$-contactomorphisms $q_\pm:S^{1|1}_\bullet\to \wh Q$.
 So we get an element
 $$
 f_+\oplus f_-\in \cW[\cA]\oplus \cW[\cA]
 % C^\infty(S^{1|1}_\bullet; \cA)\oplus C^\infty(S^{1|1}_\bullet; \cA).
 $$ 
 Considering all density we obtain a relation 
 $$\Delta_{\mu,\nu}(\cQ):\cW[\cA]\tto \cW[\cA].
 % \,C^\infty(S^{1|1}_\bullet; \cA)\looparrowright
 % C^\infty(S^{1|1}_\bullet; \cA).
 $$
 
 \begin{theorem}
 	For two elements $\cP$, $\cQ\in \Gamma_\bullet(\cA)$ we have
 	$$\Delta_{\mu,\nu}(\cP)\circ \Delta_{\mu,\nu}(\cQ)=\Delta_{\mu,\nu}(\cP\cQ).$$
 \end{theorem}
 
 {\sc Proof.} Embedding $\supset$ is obvious.
 
 On other hand, after gluing of two logarithmic forms we get a  a logarithmic form
 we get an expression, which is holomorphic outside the line on curve and continuous 
 on this curve. This singularity is removable and we get a holomorphic form.
 \hfill $\square$

\sm

{\bf\punct  Properties of relations $\boldsymbol{\Delta_{\mu,\nu}(\cP)}$.}
Thus, we consider superaffine relations $\Delta_{\mu,\nu}(\cP)$ in the 
\Frechet space $\cW[\cA]$.
We intend to show that they  satisfy certain good properties,
which allow assign a Gauss--Berezin operator for each $\Delta_{\mu,\nu}(\cP)$.

Recall that elements  of $\cW_\0$ are  series 
$\sum_{n\in 2\Z\setminus 0} a_n e^{i n\phi}$,
elements of $\cW_\1$ are series $\sum_{n\in 2\Z+1} b_k e^{i k\phi}$.
In both cases Fourier coefficients rapidly decrease
$$
a_n=o(|n|^{-M}),\quad b_k=o(|k|^{-M}),\quad 
\text{as $n$, $k\to \infty$ for all $M>0$.}
$$
We equip $\cW$ with the natural $C^\infty$-topology, 
on the other hand we define the inner 
\eqref{eq:inner}. Norms in formulas below are norms with respect to this inner product.

We decompose $\cW_{\ov j}=\cW_{\ov j}^+\oplus \cW_{\ov j}^-$, 
 where $\cW_{\ov j}^+$ (resp. $\cW_{\ov j}^+$) consists of 
Fourier series
$\sum_{n>0}$ (resp. $\sum_{n<0}$). Denote by $\cO^\aff(\cW, \cW;\cA)$
 the set
of closed affine relations $\cW[\cA]\tto\cW[\cA]$, whose Potapov transform
exists  and is continuous in the topology of $\cW[\cA]$.

\begin{proposition}
	\label{pr:long}
	 Let $\nu$, $\nu\in\R$.
	
	\sm 
	
	{\rm a)}
	Any $\Delta_{\mu,\nu}(\cP)\subset \cW[\cA]\tto \cW[\cA]$
	is a shifted Lagrangian  subspace.
	
	\sm
	
{\rm b)}   
Any $\Delta_{\mu,\nu}(\cP)$
is contained in some chart	
$$
\sigma_{\beta_1}^\cW\dots \sigma_{\beta_n}^\cW 
L
\sigma_{\delta_1}^\cW\dots \sigma_{\delta_l}^\cW,
\quad\text{ where $L$ ranges in  $\cO^\aff(\cW,\cW;\cA)$,
} 
$$
and 
$\beta_1<\dots<\beta_n$, 
$\delta_1<\dots<\delta_l$ are fixed.

\sm 

{\rm c)} The matrix $\Sigma$ of the Potapov transform,
 see \eqref{eq:potapov},
satisfies conditions:

\sm 

--- {\rm (c1)} $\left\|\begin{pmatrix}
	A_{11}&B_{11}\\ C_{11}& D_{11}
\end{pmatrix}_\downarrow \right\|\le 1$;

\sm

--- {\rm (c2)} $\|(B_{11})_\downarrow \|<1$,  $\|(C_{11})_\downarrow \|<1$;

\sm 

--- {\rm (c3)} Matrices  $(B_{11})_\downarrow$,  $(C_{11})_\downarrow$
have rapidly decreasing matrix elements%
\footnote{We say that coefficients of a matrix $Q]\{q_{ij}\}$, where $i$, $j>0$
if for any $M>0$ we have $q_{ij}=o(|i|+|j|)^M$ as
$|i|+|j|\to \infty$.};

\sm

--- {\rm  (c4)} $\left\|\begin{pmatrix}
	A_{22}&B_{22}\\ C_{22}& D_{22}
\end{pmatrix}_\downarrow\right\|<\infty $.

\sm 

--- {\rm (c5)} $(B_{22})_\downarrow $,  $(C_{22})_\downarrow$
have rapidly decreasing matrix elements.

\sm

--- {\rm d)} The vector of shift 
$\begin{pmatrix}\gamma&\delta \end{pmatrix}$ in
\eqref{eq:pot1} has rapidly decreasing coefficients. 
\end{proposition} 

We establish this statement in the rest of this section.

\sm

{\bf \punct Potapov transforms for surplace contactomorphisms.}
We refine decomposition \eqref{eq:canonical-decomposition}
and represent elements $\wh\cP\in\Gamma_\bullet(\cA)$
as 
\begin{equation}
	\wh\cP=
\pi_-\circ (p_+)_\downarrow\circ \cC_t\circ 
 (p_+)_\downarrow\circ \pi_+,
 \label{eq:canonical-decomposition-1}
\end{equation}
where $\pi_\pm\in \Cont_\vel(S^{1|1}_\bullet;\cA)$,
$(p_\pm)_\downarrow\in\Diff^{(2)}(S^1)$, and
$$ \pi_-\circ (p_-)_\downarrow=p_-, \qquad (p_+)_\downarrow\circ \pi_+=p_+.$$ 

\begin{lemma}
	\label{l:potapov-surplace}
	For any $\pi\in \Cont_\vel(S^{1|1}_\bullet;\cA)$
	its Potapov transform is  well defined and continuous,
	and $\Xi(\pi)_\downarrow=
	\begin{pmatrix}1&0\\0&1\end{pmatrix}$.	
\end{lemma}

{\sc Proof.} Let $Z\in \cont_{\ov j}(S^{1|1};\cA)$. Consider its block decomposition 
$\begin{pmatrix}\alpha&\beta\\\gamma&\delta \end{pmatrix}$
with respect to the decomposition $\cW^+[\cA]\oplus \cW^-[\cA]$.
Notice that projections $\cW[\cA]\mapsto \cW^\pm[\cA]$ 
are continuous in both $C^\infty$ and Hilbert topologies. So, blocks are continuous operators 
$\cW^\pm[\cA]\to \cW^\pm[\cA]$. Therefore blocks
of the matrix 
$$
\begin{pmatrix}
	1&0\\0&1
\end{pmatrix}+\lambda
 \begin{pmatrix}\alpha&\beta\\\gamma&\delta \end{pmatrix},
 \qquad \lambda \in \cM_+, 
$$
are bounded. So we can evaluate Potapov transform  by formula
\eqref{eq:Pi-matritsa}. After this, we can evaluate Potapov transform
of a product $\prod (1+\lambda_j Z_j)$ using the formula 
\eqref{eq:PiPi}. 
\hfill $\square$

\sm

We also need the following remark:

\begin{lemma}
	\label{l:1001}
Let $P\in \cO^\aff(\cW,\cW;\cA)$ and $\Sigma(P)_\downarrow=\begin{pmatrix}
	1&0\\0&1
\end{pmatrix}$. Then 
$$\sigma^\cW_j P \sigma^\cW_j\in \cO^\aff(\cW,\cW;\cA).$$	
\end{lemma}

\sm 

{\bf \punct Potapov transform for elements of 
$\boldsymbol{\Gamma_\bullet(\cA)}$.%
\label{ss:standard}}
Now consider the middle part
$$
\cP:=(p_-)_\downarrow\circ \cC_t\circ 
(p_+)_\downarrow \in \Gamma_\bullet
$$
of the product \eqref{eq:canonical-decomposition-1}.

In this case we have affine relations $\Delta_{\mu,\nu}(\cP):\cW\tto \cW$.
Such relation splits into a direct sum of an affine  relation
$$
\Delta^\0_{\mu,\nu}(\cP):\cW_\0\tto \cW_\0,
$$
and a linear relation
$$
\Delta^\1(\cP):\cW_\1\tto \cW_\1.
$$
The first relation is Lagrangian with respect to the symplectic form
$$
\{f_1,f_2\}_\0:=\frac{1}{8\pi}\int_{0}^{2\pi} 
\bigl(
f_1(\phi)\,d f_2(\phi)- f_2(\phi)\,d f_1(\phi)
\bigr)
$$
in $\cW_\0$.
The second relation is Lagrangian (maximal isotropic)
with respect to the orthogonal form
$$
\la g_1,g_2\ra_\1= \int_0^{2\pi} g_1(\phi)(\phi)\, g_2(\phi)\,d\phi.
$$
Properties of relations $\Delta^\0_{\mu,\nu}(\cP)$ and $\Delta^\1$
were investigated in \cite{Ner-boson}, \cite{Ner-fermion}, see also exposition
in \cite{Ner-book}, Sect. VII.5.

The relations $\Delta^\0_{\mu,\nu}(\cP)$ are morphisms of the category
$\ov {\mathrm{ASp}}$ (see \cite{Ner-book}, Sect. VI.4), its Potapov transform
$$
\left(
\begin{array}{cc|c}
P&Q&u\\R&T&v
\end{array}
\right)
$$
satisfies the following conditions:

\sm 

--- $\left\|\begin{pmatrix} P&Q\\R&T\end{pmatrix} \right\|\le 1$,
this matrix is symmetric;

\sm

--- $\|P\|<1$, $\|T\|<1$ and matrix elements of these 
matrices rapidly decrease.

\sm 

--- coordinates of $u$, $v$ rapidly decrease.

\sm 

The relations $\Delta^\1 (\cP)$ are morphisms of the category
$\ov {\mathbf{GD}}$, see \cite{Ner-fermion}.
 In this case, for some $\alpha_1<\dots<\alpha_m$,
$\gamma_1<\dots <\gamma_l$ the relation
\begin{equation*}
\sigma^{\cW_\0}_{\alpha_1}\dots \sigma^{\cW_\0}_{\alpha_m}
\,
\Delta^\0 (\cP)
\,
 \sigma^{\cW_\0}_{\gamma_1}\dots \sigma^{\cW_\0}_{\gamma_l}
 \label{eq:atlas}
\end{equation*}
has a well-defined Potapov transform 
$S=\begin{pmatrix}A&B\\C&D \end{pmatrix}$,
it satisfies the conditions (see \cite{Ner-book})

\sm 

--- $S$ is skew-symmetric; it is bounded in the both  Hilbert  and $C^\infty$-topologies ;

\sm 

--- matrix elements of $A$ and $D$ rapidly decrease.

\sm 

{\bf\punct Proof of Proposition \ref{pr:long}.}
Clearly,
\begin{multline}
\Delta_{\mu,\nu}\bigl(
\pi_-\circ (p_-)_\downarrow\circ \cC_t\circ 
(p_+)_\downarrow\circ \pi_+,\bigr)
=\\=
\Delta_{\mu,\nu}(\pi_-)\,
\Delta_{\mu,\nu}
\bigl( 
(p_+)_\downarrow\circ \cC_t\circ 
(p_+)_\downarrow 
\bigr)\,
\Delta_{\mu,\nu}(\pi_+).
\label{eq:DeltaDelta}
\end{multline}
Potapov transforms of 3 factors were discussed in the previous 
two subsections.
If
 a Potapov transform for a middle factor is  well defined, 
 then we simply multiply super-affine relations by   formula
 \eqref{eq:PiPi} and get the desired list of properties. 

 Otherwise, a Potapov transform of $\Delta^\1(\bigl( 
 (p_-)_\downarrow\circ \cC_t\circ 
 (p_+)_\downarrow 
 \bigr))$ is not well defined. Then we choose 
 $\sigma_{\alpha_i}$, $\sigma_{\gamma_j}$ such that
 \eqref{eq:atlas} has a Potapov transform and 
 represent \eqref{eq:DeltaDelta} in the form
 \begin{multline*}
\Bigl(\prod_{i} \sigma_{\alpha_i}\cdot 
 	\Delta_{\mu,\nu}(\pi_-)\cdot \prod_{i} \sigma_{\alpha_i}   \Bigr)\,
\Bigl(\prod_{i} \sigma_{\alpha_i}\cdot  	\Delta_{\mu,\nu}
 	\bigl( 
 	(p_+)_\downarrow\circ \cC_t\circ 
 	(p_+)_\downarrow 
 	\bigr) \cdot \prod_{j} \sigma_{\gamma_j}  \Bigr)
 	\times \\ \times	
 	\Bigl( \prod_{j} \sigma_{\gamma_j} \cdot 
 	\Delta_{\mu,\nu}(\pi_+)\cdot \prod_{j} \sigma_{\gamma_j} \Bigr).
 \end{multline*}
 We apply Lemma \ref{l:1001} and multiply our affine relations using formula
 \eqref{eq:PiPi}.
  \hfill $\square$

\section{Super-Fock space and Gauss--Berezin integral operators%
\label{s:super-fock}}

\COUNTERS

Here we define superanalogs of Gaussian integral  operators 
and show that they are enumerated by superaffine Lagrangian relations. 

%Consider unitary representation \eqref{eq:NS1}--\eqref{eq:NS3} of
%the superalgebra $\ns$. Let us extend it to a representation of
%$\cont^\diamond(S^{1|1};\cA)$. By Theorem \ref{th:3}, it can be integrated to

  \sm

{\bf \punct The super-Fock space.%
\label{ss:super-Fock}}
Let $z_j$,  $\ov z_j$, $u_j$, $\ov u_j$ be even variables.
Consider  monomials $z^\alpha:=z_1^{\alpha_1} z_2^{\alpha_2}\dots$, where products are actually finite.
We define   the {\it bosonic Fock space} $F$
 as a Hilbert space  with orthogonal basis $z^\alpha$,
  their norms are defined by
$$
\|z^\alpha\|^2=\alpha!:=\prod \alpha_j!
$$ 
Denote by $\ov F$ the space of formal series in $z^\alpha$, by $F^0$ the space of polynomials.
So, $F^0\subset F\subset \ov F$.

Next, let $\xi_j$, $\ov \xi_j$, $\eta_j$, $\ov\eta_j$
 be Grassmann variables, all these variables anticommute
 and they commute with even variables. We define the {\it fermionic Fock space} $\Lambda$ as a Hibert space  
with an orthonormal basis consisting of monomials 
$\dot\xi^I:=\xi_{i_1}\dots\xi_{i_k}$, where $i_1<\dots<i_k$.
Denote by $\ov\Lambda$ the space of formal series in these monomials, by $\Lambda$ the space of polynomials.
So, $\Lambda^0\subset \Lambda\subset \ov \Lambda$.

We consider spaces 
$$
\F^0:=F^0\otimes \Lambda^0, 
\quad \F:=F\otimes \Lambda,\quad \ov\F:=\ov F\otimes \ov\Lambda,
$$
we have respectively the tensor product of countable-dimensional
spaces, of Hilbert spaces, in the last case we consider the space of formal series in the variables $z_k$, $\xi_j$.

Consider  
 the corresponding variants
of 'super-Fock' space: 
$$\F^0[\cA]\subset \F[\cA]\subset \ov\F[\cA],$$
%(our main object below is the space $\F[\cA]$ consisting
%of smooth vectors of the Neveu--Schwartz algebra).
Their elements are sums
$$
f(z,\xi;\fra)=
\sum c_{\alpha,I, K} z^\alpha \dot\xi^I \fra^K
$$
We assume 
$$\xi_i\fra_j=-\fra_j \xi_k.$$
So, we can multiply elements of our spaces by elements of $\cA$
from the left and from the right.

By $\bfF[\cA]_\0$ and $\bfF[\cA]_\1$  etc. we denote the subspace of 
functions satisfying
$$
f(z,-\xi;-\fra)=f(z,\xi;\fra)\quad\text{and}\quad
f(z,-\xi;-\fra)=-f(z,\xi;\fra)\quad\text{respectively}.
$$

We define the following map 
$$
\sfS f(z,\xi;\fra):=f(z,\xi;-\fra).
$$ 
So,
$$
\sfS^2=1,\qquad
\sfS \lambda f=\lambda^\circ \sfS f \qquad\text{for any $\lambda\in \cA$.}
$$

\sm

{\bf\punct Operators and Berezin symbols.}
Let  $\xi_m$, $\ov\xi_m$, $\eta_j$, $\ov\eta_j$ be
pairwise anti-commuting variables, let they anticommute with $\fra_j$.
Let $I:\,i_1<\dots <i_m$. We denote
$$
\dot\xi^I:= \xi_{i_1}\dots \xi_{i_m},\qquad
\ddot {\ov \eta}^I:=
\ov\eta_m \dots \ov \eta_1=(-1)^{m(m-1)/2}\, \dot{\ov\eta}^I. 
$$

We say that an {\it operator} in a super-Fock space
 is a morphism of right $\cA$-modules. 
We define the {\it Berezin symbol} (cf. \cite{Ber-second}, Sect. 2-3, 
see also \cite{Ner-book}, Subset. V.3.3, II.1.9) of an operator
$R:\F^0[\cA]\to \ov\F[\cA]$
as the formal series 
$$
K_R(z,\xi;\ov u, \ov \eta;\fra):=\sum_{\alpha,I} R(z^\alpha\dot\xi^{I}) 
\cdot
\frac 1{\alpha!}
\ov  u^{\alpha}\ddot{\ov \eta}^{I}.
$$

In fact, we are interested in expressions, which are even or odd with respect to the total collection
$\{\xi_m, \ov\eta_j, \fra_k\}$. The following statement is obvious

\sm

--- {\it The map $R\mapsto K_R$ is a one-to-one correspondence between the space of 
	operators $R:\F^0[\cA]\to \ov\F[\cA]$
and the space of formal series of the form
$$
\sum_{\alpha,\beta,I,J} c_{\alpha,\beta,I,J}(\fra) z^\beta \xi^J \ov u^\alpha \ov \eta^I,
$$
where coefficients $c_{\alpha,\beta,I,J}(\fra)$ are contained in $\cA$.} 
%Formally, for $f(z,\xi;\fra)=\sum b_{\beta,I} z^\beta\xi^I$
%we have
%$$
%R f(z,\xi;\fra)=\la K_R(z,\xi,\ov u, \ov \eta),f(u, \eta)
%$$

\sm 

{\bf \punct Remark. Berezin symbols as an imitation of kernels
of integral operators.} Consider the probabilistic measure $d\mu(z)$ on $\C$
with the coordinate $z=x+iy$ given by 
$d\mu(z)=\frac1\pi e^{-|z|^2}dx\,dy$.

Next, consider `fermionic Gaussian measure'
$\exp(-\eta\ov\eta)\,d\dot\eta\,d\ddot {\ov \eta}$
defined from the following conditions:
% on the 
%space of functions $\sum_{I,J} \dot \eta^I \dot {\ov \eta}^J$
%by
$$
\int \dot\eta^I \ddot {\ov \eta}^I  \exp(-\eta\ov\eta)\,d\dot\eta\,
d\ddot {\ov \eta}=1,
$$
and integrals of all other monomials are 0.

On physical level of rigor, we have 
\begin{equation}
R f(z,\xi;\fra)=\int\limits_{\eta,\ov\eta} \int\limits_{u,\ov u}
	 K_{R}(z,\xi;\ov u, \ov\eta;\fra)\,
	 f(u,\eta;\fra)
	 \,\exp(-\eta\ov\eta)\,d\dot \eta\,
	 d\ddot {\ov \eta} \prod d\mu(u_j).
	 \label{eq:operator-vector}
\end{equation}

The symbol of products of operators is 
\begin{multline}
K_{R_1 R_2}(z,\xi;\ov w,\ov \zeta)
= \\ =
\int\limits_{u,\ov u}\int\limits_{\eta,\ov\eta}
 K_{R_1}(z,\xi;\ov u, \ov\eta;\fra)\, K_{R_2}(u,\eta; \ov w, \ov\zeta;\fra)
 	 \,\exp(-\eta\ov\eta)\,d\dot\eta\,
	 d\ddot {\ov \eta} \prod d\mu(u_j).
 %d\mu(u,\ov u, \eta,	 \ov\eta).
\label{eq:convolution}
\end{multline}
 This formula
is a precise counterpart of the usual formula for a product of
integral operators.
 For a   super-Fock with finite  degrees of freedom,
 see \cite{Ner-super}. 
 \hfill $\square$
   
\sm

{\bf \punct $\boldsymbol l$-Operators.} We say that an {\it $l$-operator} $\F^0[\cA]\to \ov\F[\cA]$ 
is a morphism of left $\cA$-modules.
If a  kernel $K_R$ satisfies the condition
\begin{equation}
K_R(z,-\xi;\ov u,-\ov \eta;-\fra)= K_R(z,\xi;\ov u,\ov \eta;\fra)
\label{eq:K-even}
\end{equation}
then $R$ is an $l$-operator sending $\F^0[\cA]_\0\to \ov\F[\cA]_\0$, $\F^0[\cA]_\1\to \ov\F[\cA]_\1$.
%and
%$$
%R (\lambda f)=\lambda Rf \quad \text{where $\lambda\in \cA$.}
%$$

If a kernel $K_Q$ satisfies
\begin{equation}
K_Q(z,-\xi;\ov u,-\ov \eta;-\fra)=- K_Q(z,\xi;\ov u,\ov \eta;\fra),
\label{eq:K-odd}
\end{equation}
then $Q$ transposes parities, and 
$$
Q (\lambda f)=\lambda^\circ Qf, \quad \text{where $\lambda\in \cA$.}
$$
So, $\sfS Q$ is an $l$-operator.
%We say an {\it  in a super-Fock space}
%is an operator  whose kernel satisfies \eqref{eq:K-even}
%or transformation $Q\sfS$, if a kernel $K_Q$ satisfies \eqref{eq:K-odd}.

\sm

{\bf\punct Creation-annihilation operators 
and the Heisenberg Lie superalgebra.%
\label{ss:creation-annihilation}}
We define (left) {\it partial derivatives} is Grassmann variables $\xi$, $\ov\xi$, $\eta$, $\ov\eta$
\dots\vphantom{.} in the usual way: if $f$ and $g$ do not depend on $\xi_j$, then
$$
\frac{\partial}{\partial \xi_j}(\xi_j f+g)=f.
$$
We have 
$$
\frac{\partial}{\partial \xi_j} \fra_k F=- \fra_k \frac{\partial}{\partial \xi_j} F.
$$
In few formulas we use {\it right partial derivatives} defined by
$$
( f \xi_j+g)\frac{\eth}{\eth \xi_j}=f.
$$
So,
$$
h(\xi, \ov\eta, \dots;\fra) \frac{\eth}{\eth \xi_j}=
-\frac{\partial}{\partial \xi_j}h(-\xi, -\ov\eta, \dots;-\fra).
$$

Consider the following differential operators ({\it creation and annihilation operators}) in the super-Fock space:
 operators $ z_j$, $\frac\partial{\partial z_j}$ of parity $\0$ and  operators $\xi_k$ and $\partial_{\xi_k}$ of parity $\1$.
Their anticommutators are
$$
[z_j,\tfrac\partial{\partial z_j}]_s=-1,\qquad [\xi_k, 
\tfrac\partial{\partial\xi_k}]_s=1,
$$ 
other supercommutators are zero. Adding scalar operators 
to this collection we get  the Heisenberg Lie superalgebra
$\frheis$.

Denote
\begin{multline}
\wh a(v)=\wh a(v^+,v^-)=
\wh a(v_\0^{+},v_\1^{+},v_\0^{-},v_\1^{-})=
\\=
\sum  v_\0^{j+} z_j+   \sum v^{k+}_\1 \xi_k   +
\sum   v_\0^{j-}\frac\partial{\partial z_j} +
\sum   v_\1^{k-} \frac\partial{\partial\xi_k},
\label{eq:wh-a}
\end{multline}
where $v_{\ov i}^{j \pm }\in \cA$.
So the  space $V[\cA]$ of possible parameters 
 splits as
$$
V=V_\0^+\oplus V_\1^+ \oplus V_\0^-\oplus V_\1^- 
$$

If sums in \eqref{eq:wh-a} are actually finite, then
the operator $\wh a(v)$ is well defined in both $\bfF^0$ and%
\footnote{But operators $z_j$ and $\frac{\partial}{\partial z_j}$
are unbounded in the Hilbert space $\bfF$.}
$\ov\bfF$.

\begin{lemma}
If $\wh a(v)$ and $\wh a(w)$ are finite, then
\begin{equation}
[\wh a(v),\wh a(w)]_s= J(v,w),
\label{eq:aaJ}
\end{equation}
where $J$ is the orthosymplectic form determined by
\begin{multline}
J\Bigl((v_\0^{+},v_\1^{+},v_\0^{-},v_\1^{-}) , (w_\0^{+},w_\1^{+},w_\0^{-},w_\1^{-}) \Bigr)
	=\\=
\frac12\begin{pmatrix}
v_\0^{+}&v_\0^{-}&v_\1^{+}&v_\1^{-}
\end{pmatrix}
\begin{pmatrix}
	0&1&0&0\\
	-1&0&0&0\\
	0&0&0&1\\
	0&0&1&0
\end{pmatrix}
\begin{pmatrix}
	(w_\0^{+})^{st}\\(w_\0^{-})^{st}\\(w_\1^{+})^{st}\\(w_\1^{-})^{st}
\end{pmatrix}.
\label{eq:J}
\end{multline}
\end{lemma}

{\sc Proof.} We partially present a straightforward calculation since its result is not a priory
obvious. Let $a=a_\0+a_\1$, $b=\b_\0+b_\1\in \cA$. Then
\begin{equation*}
\Bigl[a_\0\, z_k,b_\0\, \frac\partial{\partial z_k}\Bigr]=-a_\0\, b_\0,\quad
\Bigl[a_\0\, z_k,b_\1\, \frac\partial{\partial z_k}\Bigr]=-a_\0\, b_\1,
\quad
\Bigl[a_\1\, z_k,b_\0\, \frac\partial{\partial z_k}\Bigr]=-a_\1\, b_\0,
\end{equation*}
\begin{multline*}
\Bigl[a_\1\, z_k,b_\1\, \frac\partial{\partial z_k}\Bigr]_s
=a_\1\, z_k\, b_\1  \frac\partial{\partial z_k}+ 
b_\1 \frac\partial{\partial z_k} a_\1\, z_k
=\\=
a_\1\, b_\1\, z_k   \frac\partial{\partial z_k}+b_\1\, a_\1 
\Bigl(1+z_k   \frac\partial{\partial z_k}\Bigr)=
-a_\1\, b_\1. 	
\end{multline*}	
So,
$$
\Bigl[a\, z_k,b\, \frac\partial{\partial z_k}\Bigr]_s=-a\, b,
\quad 
\Bigl[b'z_k, a' \frac\partial{\partial z_k}\Bigr]_s=b'_\0 a'_\0+ b'_\1 a'_1+b'_\0 a'_\1- a'_\1b'_\1=b'\, a'.
$$
Next,
\begin{align*}
	\Bigl[a_\0\, \xi_k , b_\0\, \frac\partial{\partial \xi_k}\Bigr]_s&=
a_\0\, b_\0 \Bigl[\xi_k, \frac\partial{\partial \xi_k} \Bigr]_s=
a_\0\, b_\0,
\\
\Bigl[a_\0\, \xi_k , b_\1\, \frac\partial{\partial \xi_k}\Bigr]&=
a_\0\, \xi_k  b_\1\, \frac\partial{\partial \xi_k}-
 b_\1\, \frac\partial{\partial \xi_k} a_\0\, \xi_k=-a_\0\,b_\1 ,
 \\
\Bigl[a_\1\, \xi_k , b_\0\, \frac\partial{\partial \xi_k}\Bigr]&=
a_\1\, \xi_k  b_\0\, \frac\partial{\partial \xi_k}-
b_\0\, \frac\partial{\partial \xi_k} a_\1\, \xi_k=a_\1\,b_\0, \\
\Bigl[a_\1\, \xi_k , b_\1\, \frac\partial{\partial \xi_k}\Bigr]&=
a_\1\, \xi_k  b_\1\, \frac\partial{\partial \xi_k}-
b_\1\, \frac\partial{\partial \xi_k} a_\1\, \xi_k=-a_\1\,b_\1, 
\end{align*}
and we get 
$$
\Bigl[a\, \xi_k , b\, \frac\partial{\partial \xi_k}\Bigr]_s=a\,b^\circ,\quad
\Bigl[b'\, \frac\partial{\partial \xi_k},a'\, \xi_k \Bigr]_s=a_\0' b_0'+a_\0 b_\1 -a_\1 b_\0+a_\1 b_1
=b' (a')^\circ.
$$
Other supercommutators are zero.
\hfill $\square$
 
 \sm  

We need some topological versions of the Heisenberg Lie superalgebra.

\sm

1. We consider the space $\sheis_\infty$
consisting  series \eqref{eq:wh-a}
with rapidly decreasing complex  coefficients  and 
the Lie superalgebra 
$\frheis^\infty:= \sheis_\infty\oplus \C$ obtained by addition
of scalar operators%
\footnote{This Lie superalgebra is our tool below (see Sect. \ref{s:last}).}.

\sm 

2. The space $\sheis_{\bfF^0}$ consisting 
of  series \eqref{eq:wh-a} with complex coefficients, where sequences $v_{\ov i}^-$
are arbitrary and $v_{\ov i}^+$ are finitary. 
The
Lie superalgebra $\frheis_{\bfF^0}:=\sheis_{\bfF^0}\oplus \C$ acts in $\bfF^0$.

\sm

3. The space $\sheis_{\ov\bfF}$ consisting 
of series \eqref{eq:wh-a}, where sequences $v_{\ov i}^+$
are arbitrary and $v_{\ov i}^-$ are finitary. The 
Lie superalgebra $\frheis_{\ov\bfF}:=\sheis_{\ov\bfF}\oplus \C$  act in $\ov \bfF$.

\sm

{\bf \punct Superspinors and orthosymplectic Lie superalgebra.}
It is easy to see that all polynomial expressions
$$
\cH\Bigl(z_j, \frac{\partial}{\partial z_j},
 \xi_k, \frac{\partial}{\partial \xi_k}\Bigr)
$$
with summands of degree $\le 2$ determine a projective
representation of an affine orthosymplectic Lie superalgebra. 

More precisely, consider infinite-dimensional affine orthosymplectic
Lie superalgebra $\mathfrak{aosp}(2\infty,2\infty)$ consisting of block matrices of
size $(\infty+\infty+1)+(\infty+\infty)$ having the
 the form
\begin{equation}
\left(
\begin{array}{ccc|cc}
P_{11}&P_{12}&0&Q_{11}&Q_{12}\\
P_{21}&P_{22}&0&Q_{21}&Q_{22}\\
V_1&V_2&0&W_1&W_2\\
\hline
R_{11}&R_{12}&0&T_{11}&T_{12}\\
T_{21}&T_{22}&0&T_{21}&T_{22}
\end{array}\right),
\label{eq:aosp}
\end{equation}
and satisfying natural conditions of symmetry, see Subsect. \ref{ss:lie-superalgebras}.
For such a matrix we assign the following operator
in  $\F^0$ (or in $\ov\F$) 
\begin{align}
&\frac12 \sum p_{11}^{kl}\, z_k z_l+ 
\sum p_{12}^{kl}\, z_k\frac\partial{\partial z_l}+
\frac12 \sum p_{22}^{kl}\, \frac{\partial^2}{\partial z_k\partial z_l}
+
\label{eq:osp-zz}
 \\+
&\sum v_1^k\, z_k+\sum v_2^k \frac\partial{\partial z_k}
+
\label{eq:osp-z}
\\+
&\frac12\sum t_{11}^{kl}\, \xi_k \xi_l+
\sum t_{12}^{kl}\, \xi_k\frac{\partial}{\partial \xi_l}+
\frac12 \sum t_{22}^{kl}\, 
\frac{\partial}{\partial \xi_k} \frac{\partial}{\partial \xi_l}
+
\label{eq:osp-xixi}
\\+
&\sum v_1^k\,\xi_k +
\sum v_2^k \frac{\partial}{\partial \xi_k}
+
\label{eq:osp-xi}
\\+
&\sum q_{11}^{kl}\, z_k\xi_l +
\sum q_{12}^{kl}\, z_k  \frac{\partial}{\partial \xi_l}+
\sum q_{21}^{kl}\,\xi_l \frac{\partial}{\partial z_k}+
\sum q_{22}^{kl}\,\frac{\partial}{\partial z_k}\frac{\partial}{\partial \xi_l}
\label{eq:osp-zxi}
\end{align}
(summands \eqref{eq:osp-zz}--\eqref{eq:osp-xixi})
have parity $\0$, summands \eqref{eq:osp-xi}--\eqref{eq:osp-zxi} have parity
$\1$). It is easy to verify that we get a projective representation of
$\mathfrak{aosp}(2\infty,2\infty)$.

\sm

{\sc Remark.}
The operators \eqref{eq:osp-zz}--\eqref{eq:osp-z} form a projective representation of the affine symplectic Lie algebra in
the bosonic Fock space $F^0$;
the operators \eqref{eq:osp-xixi} form a projective representation
of orthogonal Lie algebra in the fermionic Fock space $\Lambda^0$.
\hfill $\boxtimes$

\sm

{\sm Remark.}
We  can write the expression \eqref{eq:osp-zz}--\eqref{eq:osp-zxi}
in the form 
$$
\begin{pmatrix}
	z&\frac{\partial}{\partial z}&1&\xi&\frac{\partial}{\partial \xi}
\end{pmatrix}
\left(
\begin{array}{ccc|cc}
	P_{11}&P_{12}&0&Q_{11}&Q_{12}\\
	P_{21}&P_{22}&0&Q_{21}&Q_{22}\\
	V_1&V_2&0&W_1&W_2\\
	\hline
	R_{11}&R_{12}&0&T_{11}&T_{12}\\
	T_{21}&T_{22}&0&T_{21}&T_{22}
\end{array}\right)
\begin{pmatrix}
	z^t\\ \bigl[\frac{\partial}{\partial z}\bigr]^t\\1\\ \xi^t\\ \bigl[\frac{\partial}{\partial \xi}\bigr]^t
\end{pmatrix},
$$
where $\bigl[\frac{\partial}{\partial z}\bigr]^t$, 
$\bigl[\frac{\partial}{\partial z}\bigr]^t$
denote columns consisting of partial differentiations.
\hfill $\boxtimes$

\sm

We apply this construction to the affine orthosymplectic Lie superalgebra
of the space $\cW$
defined in Subsect. \ref{ss:cont-osp}. We write matrices
\eqref{eq:aosp} in the basis 
$$\frac 1{\sqrt{2\pi\cdot |2k|}}e^{2ik\phi}\in \cW_\0,\qquad
\frac 1{\sqrt{2\pi}}
e^{(2l+1)i\phi}(d\phi)^{1/2}\in \cW_\1.$$ 

Consider the subalgebra   
$\cont_{\mathrm{fin}}(S^{1|1}_\bullet)\subset \cont(S^{1|1}_\bullet)$
generated by vector fields $L_n$, $M_r$
defined by \eqref{eq:LMbasis}.
It is not contained in $\mathfrak{aosp}(\cW)$, but 
for $N\in \cont_{\mathrm{fin}}(S^{1|1}_\bullet)$,
$X\in \mathfrak{aosp}(\cW)$ we have $[N,X]\in \mathfrak{aosp}(\cW)$.
So we can consider the semidirect product, say 
$\wh{\mathfrak{aosp}}(\cW)$, of these 
algebras. For  elements $L_n$, $M_r$ of 
$\cont_{\mathrm{fin}}(S^{1|1}_\bullet)$ we consider the corresponding
infinitesimal
affine  operators 
and write
$\tau_{\mu,\nu}(N)$ by
\eqref{eq:osp-zz}--\eqref{eq:osp-zxi}.
This leads to formulas \eqref{eq:fock-L}--\eqref{eq:fock-M}
for the representation of the Neveu--Schwarz algebra
in $\F^0$. Clearly, these series are well-defined
operators in $\F^0$ and we get a projective representation
of $\wh{\mathfrak{aosp}}(\cW)$ in  $\F^0$.

\sm

{\sc Remark.}
More generally, we can consider the Lie superalgebra of orthosymplectic matrices of size $(2\infty+2\infty)$ whose nonzero matrix elements
a located in a finite number of diagonals.
\hfill $\boxtimes$.

\sm

Therefore we have a representation of the surplace group
$\wh{\mathfrak{aosp}}_\vel(\cW,\cA)$ in $\F[\cA]$,
we change 
$$\xi_r\mapsto \sfS \xi_r;\qquad
\frac{\partial}{\partial \xi_r}\mapsto 
\sfS\,\frac{\partial}{\partial \xi_r};
\qquad M_r\mapsto \sfS M_r
.$$
Let $\wt X:v\mapsto vX+h$ be a pure element of 
$\wh{\mathfrak{aosp}}_\vel(\cW;\cA)$, 
let $\lambda\in \cM_+$
has  the same parity. Then 
\begin{equation}
(1+\lambda \wt X)^{-1} \wh a(v) 
(1+\lambda \wt X)=\wh a(v) -[\lambda X,\wh a(v)]
=\wh a(v+\lambda X)+ J(v, h)
\label{eq:XvX}
\end{equation}
 (we consider $v$ as a row of Fourier coefficients
 of elements of $\cW[\cA]$.

\sm

{\bf \punct Automorphisms of canonical super-commutation relations.}
Now we repeat well-known arguments of Friedrichs \cite{Friedrichs}
(which were a starting point for  Berezin in \cite{Ber-second}).

\begin{lemma}
	\label{l:commuting}
	Any l-operator $R:\F^0[\cA]\to \ov \F[\cA]$ commuting
	with all operators $z_j$, $\frac{\partial}{\partial z_j}$,
	$\xi_k$, $\frac{\partial}{\partial \xi_k}$ 
	has a form $Rf=b f$, where $b\in \cA_\0$ or $Rf= c\sfS f$,
	where $c\in\cA_-$. 
\end{lemma}	

{\sc Proof.} The function $\psi(z,\xi;\fra):=1$ is annihilated
by all operators $\frac{\partial}{\partial z_j}$, $\frac{\partial}{\partial \xi_k}$.
So, the function $ R \psi$ is annihilated by the same operators, i.e.,
it is a phantom constant, say $d$.
Since $R$ commutes with $z_j$, $\xi_k$, it sends monomials
$z^\alpha \xi^I\to z^\alpha \xi^I d$.
\hfill $\square$.

\sm 

Next consider an intermediate topological vector space $\cF$,
such that 
$$\F^0\subset \cF\subset \ov \bF$$
 and tautological
embeddings $\F^0\to \cF$ and $\cF\to \ov\F$ are continuous.
Denote by $\sheis_\cF$ the linear space of creation-annihilation
operators that are bounded%
\footnote{Below $\cF$ is the space of smooth vectors of the Neveu--Schwarz
algebra, and $\sheis_\cF$ coincides with $\sheis_\infty$.
We do not need any general abstractionism.} in $\cF$.

 Consider the set $\OSp^{\mathrm{res}}(\sheis_\cF;\cA)$ of all linear operators $g$ in $\sheis_\cF[\cA]$
 preserving the form $J(\cdot,\cdot)$, for which there exists 
 invertible operator $T(g)$ in $\cF[\cA]$ such that
 \begin{equation}
 T(g)^{-1}\wh a(v) T(g)= \wh a (vg), \qquad \text{for all $v\in \sheis_\cF[\cA]$.}
 \label{eq:aut-1}
 \end{equation}
 Denote by $\mathrm{AOSp}^{\mathrm{res}}(\sheis_\cF;\cA)$ the set of all
 affine transformations $v\mapsto vg+ p$ of $\sheis_\cF[\cA]$, for which
 there exist invertible operator $T(g,p)$ such that
 \begin{equation}
 T(g,p)^{-1} \wh a(v)  T(g,p)=\wh a(vg)+ J(vg, p)
  \qquad \text{for all $v\in \sheis_\cF[\cA]$.}
  \label{eq:aut-2}
\end{equation}
 
 \begin{proposition}
 {\rm a)}	The set
$\OSp^{\mathrm{res}}(\sheis_\cF;\cA)$ is a group and the map
$g\mapsto T(g)$ is a projective representation of this group.

\sm

{\rm b)} The set $\mathrm{AOSp}^{\mathrm{res}}(\sheis_\cF;\cA)$
is a group  and $g\mapsto T(g,p)$ is a projective representation
of this group.
\end{proposition}

{\sc Proof.} a) First, we notice that an operator $T(g)$ satisfying
\eqref{eq:aut-1} is unique up to a scalar
factor $\in\cA_\0$, if it exists. Indeed, let $T'(g)$
be another such operator, $ T'(g)^{-1}\wh a(v) T'(g)= \wh a (vg)$.
Then
$$
 T'(g)^{-1}\wh a(v) T'(g)=  T(g)^{-1}\wh a(v) T(g).
$$ 
So $T(g)T'(g)^{-1}$ commutes with all operators $\wh a(v)$
and we apply Lemma \ref{l:commuting}.

 Let $g_1$, $g_2\in \OSp^{\mathrm{res}}(\sheis_\cF;\cA)$.
Then
\begin{equation*}
T(g_2)^{-1}T(g_1)^{-1}\wh a (v) T(g_1)T(g_2)=\\=
T(g_1)^{-1}\wh a (vg_1) T(g_1)=\wh a(vg_1g_2).
\end{equation*}
and we can set $T(g_1 g_2):=T(g_1)T(g_2)$. 

\sm 

b) Let maps $v\mapsto v g_1+p_1$, $v\mapsto v g_2+p_2$
be contained in $\mathrm{AOSp}(\sheis_\cF;\cA)$, their product
is $v\mapsto vg_1g_2 + p_1g_2+p_2$. We have
\begin{multline*}
T(g_2,p_2)^{-1}T(g_1,p_1)^{-1}\wh a (v) T(g_1,p_1)=\\=
T(g_2,p_2)^{-1}\bigl(\wh a(vg_1)+J(vg_1,p_1)\bigr)T(g_1,p_1)
=\\=
\wh a(vg_1g_2)+ \Bigl[J(vg_1g_2,p_2)+J(vg_1,p_1)\Bigr].
\end{multline*}
We have $J(v g_1,p_1)=J(v g_1g_2,p_1g_2)$. Therefore the expression
is square brackets equals to $J(vg_1g_2,\,p_2+p_1 g_1)$.
\hfill $\square$

\sm

{\bf\punct Creation-annihilation operators and relations.%
\label{ss:creation-relation}}

\begin{lemma}
	Let $L$, $M:\sheis_\cF[\cA]\tto \sheis_\cF[\cA]$ be  relations.
	Let $l$-operators $P$, $R$ satisfy conditions
	\begin{align*}
\wh a(v) P=\wh a(w) \,P,\,\text{\rm for all $w\oplus v\in L$};
\\
\wh a(w)R=\wh a(y)\,R	,\,\text{\rm for all $y\oplus w\in M$}
		\end{align*}
	Then $PR$ satisfies the condition 
	$$
	\wh a(v) PR= PR \,\wh a (y) ,\,\text{\rm for all $y\oplus v\in L\circ M$.}
$$
	\end{lemma}

{\sc Proof.} Indeed, let $y\oplus v\in L\circ M$.
Then there exists $w\in \sheis_\cF[\cA]$ such that $y\oplus w\in M$, $w\oplus v\in L$.
Therefore, 
$$
\qquad \qquad\qquad\wh a(v) PR=P\,\wh a(w) R=PR \,\wh a(y).\qquad\qquad\qquad \square
$$

\sm 

Let $R:\F^0(\cA)\to \ov\F(\cA)$ be an $l$-operator.
Consider the equation
\begin{equation}
	\wh a(v) R=R\, \wh a(w)+ \chi R
	%	\wh a(w_\0^+,w_\1^+ , w_\0^-, w_\1^-;)
	%	\,R=R\,
	%	\wh a(v_\0^+,v_\1^+ , v_\0^-, v_\1^-),
	\label{eq:ca-relations} 	
\end{equation}
where 
$$\wh a(v)\in \sheis_{\ov\F}[\cA], \qquad
\wh a(w)\in \sheis_{\F^0}[\cA],\quad \chi\in \cA.
$$

For a  nonzero $l$-operator $R$ denote by $\frL(R)$ the set of all 
$w\oplus v\in \sheis_{\F^0}\oplus \sheis_{\ov F}$,	
for which there exist $\chi$ satisfying \eqref{eq:ca-relations}.

\begin{proposition}
	\label{pr:cLP-isotropic}
	Let  $R$ be an $l$-operator such that $R_\downarrow\ne 0$.
	Then	the linear relation $\frL(R)$ is $J$-isotropic, i.e.,
	for $w_1\oplus v_1$, $w_2\oplus v_2\in \frL(R)$
	we have 
	$$
	J(v_1,v_2)-J(w_1,w_2)=0.
	$$
\end{proposition}

{\sc Proof.} Consider homogeneous $v_j$, where $j=1,2$ and $\ov p(v_j)=:\ov \zeta_j$.  
We write
$$
\wh a(v_j) R= R\,\wh a(w_j)+\chi_j R.
$$
Here $\chi_j$ depends on $w_j\oplus v_j$, and $\ov p(w_j)=\ov p(\chi_j)=\ov p(v_j)=\ov \zeta_j$.
We have
\begin{multline*}
	\wh a(v_1)\wh a (v_2)R= \wh a(v_1)\bigl(R\wh a(w_2)+\chi_2 R \bigr)
	=\\=
	R\,\wh a(w_1) \wh a(w_2)+\Bigl(\chi_1 R\,\wh a(w_2)+(-1)^{\ov \zeta_1 \ov\zeta_2}
	\chi_2 R\, \wh a (w_1)+ (-1)^{\ov \zeta_1 \ov \zeta_2} \chi_1 \chi_2	R\Bigr),
\end{multline*}
\begin{multline*}
	\wh a(v_2)\wh a (v_1)R
	=\\=
	R\,\wh a(w_2) \wh a(w_1)+\Bigl(\chi_2 R\,\wh a(w_1)+(-1)^{\ov \zeta_1 \ov\zeta_2}
	\chi_1 R\, \wh a (w_2)+ (-1)^{\ov \zeta_1 \ov \zeta_2} \chi_1 \chi_2	R\Bigr).
\end{multline*}
If $\ov \zeta_1 \ov \zeta_2=\0$, then the expressions in big brackets coincide.
If $\ov \zeta_1 =\ov \zeta_2=\1$, then they differs by the sign.
Therefore,
$$
[\wh a(v_1),\wh a (v_2)]_s\, R= R\, [\wh a(w_1),\wh a (w_2)]_s.
$$
Since $R$ is $l$-linear, we get
$$
\Bigl([\wh a(v_1),\wh a (v_2)]_s-[\wh a(w_1),\wh a (w_2)]_s\Bigr)\, R= 0.
$$
Since $R_\downarrow\ne 0$, the expression in big brackets is 0.
\hfill $\square$

\sm

{\bf \punct Multiplications by creation-annihilation operators.}
For an operator $R$ with kernel $K(z,\xi,\ov u, \ov \eta;\fra)$ we have the following correspondences
between operators  and their  kernels:
\begin{gather}
	z_i\cdot R\,\longleftrightarrow\, z_i K(z,\xi,\ov u, \ov \eta;\fra),\quad
	\frac\partial{\partial z_i}\cdot R\,\longleftrightarrow\, 
	\frac\partial{\partial z_i} K(z,\xi,\ov u, \ov \eta;\fra), 
	\label{eq:corr-1}
	\\ 
	\xi_j\cdot R\,\longleftrightarrow\, \xi_j K(z,\xi,\ov u, \ov \eta;\fra),
	\quad \frac\partial{\partial\xi_j}\cdot R\,\longleftrightarrow\, \frac\partial{\partial \xi_j} K(z,\xi,\ov u, \ov \eta;\fra),
		\label{eq:corr-1.1}
	\\
	R\cdot u_i \,\longleftrightarrow\, \frac\partial{\partial\ov u_i} K(z,\xi,\ov u, \ov \eta;\fra),
	\quad R\cdot \frac\partial{\partial u_i} \,\longleftrightarrow\, {\ov u_i} K(z,\xi,\ov u, \ov \eta;\fra),
		\label{eq:corr-1.2}
	\\
	R\cdot \eta_j \,\longleftrightarrow\,  K(z,\xi,\ov u, \ov \eta;\fra) \frac\eth{\eth\ov \eta_j},\quad R\cdot \frac\partial{\partial\eta_j} \,\longleftrightarrow\,  K(z,\xi,\ov u, \ov \eta;\fra)\,{\ov  \eta_j},
	\label{eq:corr-2}
	\\
	\sfS \cdot R\cdot \sfS \,\longleftrightarrow\,  K(z,\xi,\ov u, \ov \eta;-\fra).
		\label{eq:corr-3}
\end{gather}

{\bf \punct Gauss--Berezin operators.%
\label{ss:Gauss-Berezin}}
Here we define a superhybrid of Gaussian operators 
in bosonic Fock space (see, \cite{Ner-boson}, \cite{Ner-book}, Sect.V.4.1)
and Berezin (fermionic Gaussian) operators in fermionic Fock space
\cite{Ner-fermion}, \cite{Ner-book}, Sect. II.4).
This topic was partially discussed in \cite{Ner-super}.

We define a {\it Gauss--Berezin operator in the narrow sense}
$$
\lambda \cdot \cB\left[\begin{array}{cc|c}
A&B&a^t\\
C&D&b^t
\end{array}\right]:\,
\F^0[\cA]\to \ov\F[\cA]$$
as an operator whose kernel is a formal series of the form
\begin{multline}
K(z,\xi;\ov u,\ov\eta;\fra)
=\\=
	\lambda \cdot \exp\left\{\frac12
	 \begin{pmatrix}
		z&\xi&\ov u&\ov\eta
	\end{pmatrix} 
\begin{pmatrix}
	A_{11}&A_{12}&B_{11}&B_{12}\\
	A_{21}&A_{22}&B_{21}&B_{22}\\
	C_{11}&C_{12}&D_{11}&D_{12}\\
	C_{21}&C_{22}&D_{21}&D_{22}
\end{pmatrix}
\begin{pmatrix}
	z^t\\ \xi^t\\ \ov u^t\\ \ov\eta^t
\end{pmatrix}
\right.%\},
\times\\\times
\left.
	 \begin{pmatrix}
	z&\xi&\ov u&\ov\eta
\end{pmatrix} 
\begin{pmatrix}
	a_1^t\\ a_2^t\\ b_1^t\\ b_2^t
\end{pmatrix}
\right\}
\label{eq:GB}
\end{multline}
where

\sm

---  $\lambda\in \cA_\0$ is invertible;

\sm 

--- the expression has parity $\0$ in the set $\xi_1$, $\xi_2$ \dots,
$\eta_1$, $\eta_2$, \dots, $\fra_1$, $\fra_2$, \dots;

\sm

--- the first summand in the exponential 
is 
 a qudratic form in variables $z$, $\xi$, $\ov u$, $\ov\eta$;
 this matrix must satisfy the natural condition of symmetry.
 
 \sm
 
 {\sc Remark. } 
In more details:    

\sm 

---
 matrices $A_{ij}$, $B_{ij}$, $C_{ij}$, $D_{ij}$ are composed of elements of $\cA$; if $i+j$ is even, then their elements are $\in \cA_\0$,
 if $i+j$ is odd, then their elements are $\in\cA_\1$; elements of vectors $a_1$, $b_1$
 are $\in \cA_\0$, elements of $a_2$, $b_2$ are $\in \cA_\1$.

 \sm
 
--- the matrix 
$\begin{pmatrix} A_{11}&B_{11}\\C_{11}&D_{11} \end{pmatrix}$
is symmetric; the matrix 
$\begin{pmatrix} A_{22}&B_{22}\\C_{22}&D_{22} \end{pmatrix}$
is skew symmetric and 
$\begin{pmatrix} A_{12}&B_{12}\\C_{12}&D_{12} \end{pmatrix}^t
=-\begin{pmatrix} A_{21}&B_{21}\\C_{21}&D_{21} \end{pmatrix}$.
\hfill $\boxtimes$
 
\sm

\sm

Next, we define $l$-operators
$$\frD_j=\Bigl(\xi_j+\frac\partial{\partial\xi_j}\Bigr)\sfS.$$
If $f_1(z,\xi;\fra)$, $f_2(z,\xi;\fra)$ do not depend
on a variable $\xi_j$, then
\begin{equation}
\frD_j (f_1(z,\xi;\fra) +\xi_j f_2(z,\xi;\fra))=\xi_j f_1(z,\xi;\fra) +f_1(z,\xi;\fra).
\label{eq:D-changes}
\end{equation}
So, 
$$\frD_j^2=1, \qquad \frD_k\frD_l=-\frD_l\frD_k \quad \text{for $k\ne l$.}$$
Notice that  operators $\frD_j$ send $\F^0[\cA]\to \F^0[\cA]$,
 $\F[\cA]\to \F[\cA]$, $\ov\F[\cA]\to \ov\F[\cA]$.

A {\it Gauss--Berezin operator} is an operator of the form
\begin{equation}
Q=\frD_{i_1}\dots \frD_{i_\alpha} R\,\frD_{j_1}\dots \frD_{j_\beta},  
\label{eq:GB-general}
\end{equation}
where $R$ is a Gauss--Berezin operator in the narrow sense.

\sm

{\sc Remarks.}
a) Usually, Gauss--Berezin operators have many different representations 
in the form \eqref{eq:GB-general}. See the next subsection.

\sm

b) For Gaussian operators in bosonic Fock spaces
 and Berezin operators in fermionic Fock spaces
 conditions of boundedness are well investigated
 (see, respectively, \cite{Ner-boson}, \cite{Ner-book},
Theorems IV.4.10, IV.4.14, V.3.1, V.3.2, V.4.6, 
 \cite{Olsh}and  \cite{Ner-fermion}, \cite{Ner-book}, 
 Theorems IV.2.1-4). In  cases, which can be interesting
for representation theory, Gauss--Berezin operators are unbounded 
as operators  $\bfF[\cA]\to \bfF[\cA]$
and are not defined as operators $\bfF^0[\cA]\to \bfF^0[\cA]$
and $\ov\bfF[\cA]\to \ov \bfF[\cA]$. I have no sufficient
experimental data for  understanding which topologies in super-Fock space are
natural. For this reason, I do not try to find general conditions
for boundedness.

\sm

c) We can evaluate a product of two Gauss--Berezin operators in the narrow sense
$R_1$, $R_2$ by formula \eqref{eq:convolution}. We come to a super-Gaussian
integral in the right hand side. On evaluation
of such integrals, see \cite{Ner-super}, Sect.~4.  It may happen that
a kernel of a product satisfies property 
$K_{R_1R_2}(z,\xi,\ov u, \ov \eta;\fra)_\downarrow=0$,
then $R_1R_2$ is not a Gauss--Berezin operator
(at least in the sense of our definition). I am not sure that this
obstacle is unique.
\hfill $\boxtimes$

\sm

{\bf\punct Intersections of charts.}

\begin{proposition}
\label{pr:atlas-GB}
	A Gauss--Berezin operator $T$ with a kernel
	$L(z,\xi;\ov u, \ov\eta;\fra)$ can be represented in
	the form \eqref{eq:GB-general} with given $i_1<\dots<i_\alpha$,
	$j_1<\dots<j_\beta$ if and only if 
	$
	L_T(z,\xi;\ov u, \ov\eta;\fra)_\downarrow
	$
	contains
	a non-zero  term of the form
	$$
	c\cdot
	\xi_{i_1}\dots \xi_{i_\alpha} \ov\eta_{j_1} \dots \ov\eta_{j_\beta} 
	%z^\alpha \ov u^\beta, 
	\qquad\text{where $c\in \C$}.
	$$
\end{proposition}

{\sc Remark.} For a kernel \eqref{eq:GB} we have
\begin{align}
K(z,\xi;&\ov u,\ov \xi;\fra)_\downarrow=\lambda_\downarrow\times
\notag
\\
\times
\exp&\left\{\frac12\begin{pmatrix}z&\ov u \end{pmatrix}
\begin{pmatrix}
A_{11}&B_{11}\\ C_{11}&D_{11}
\end{pmatrix}_\downarrow
\begin{pmatrix}z^t\\\ov u^t \end{pmatrix}+
\begin{pmatrix}z&\ov u \end{pmatrix}
\begin{pmatrix}a_1^t\\ b_1^t \end{pmatrix}_\downarrow
\right\}
\times 
\label{eq:Kdown1}
\\ 
&\times
\exp\left\{\frac12\begin{pmatrix}\xi&\ov \eta \end{pmatrix}
\begin{pmatrix}
A_{22}&B_{22}\\ C_{22}&D_{22}
\end{pmatrix}_\downarrow
\begin{pmatrix}\xi^t\\\ov \eta^t \end{pmatrix}
\right\}.
\label{eq:Kdown2}
\end{align}
The expansion of the  factor \eqref{eq:Kdown1}
contains the summand 1. Therefore
a presence of the summand $\dot \xi^I\dot {\ov\eta}^J$
depends only on the factor \eqref{eq:Kdown2}.
\hfill $\boxtimes$

\sm

It is sufficient to prove our statement for Gauss--Berezin operators
in the narrow sense (so $\alpha+\beta$ is even).
\begin{lemma}
Let a Gauss--Berezin operator has a kernel \eqref{eq:GB},
\begin{equation}
K(z,\xi;\ov u,\ov\eta;\fra)=\lambda\cdot 
\exp\bigl\{U(z,\xi;\ov u,\ov\eta;\fra)\bigr\}
\label{eq:K-U}
\end{equation}
Let the decomposition of $K(z,\xi;\ov u,\ov\eta;\fra)_\downarrow$
contains a nonzero term 	$c\cdot
	\xi_{i_1}\dots \xi_{i_\alpha} \ov\eta_{j_1} \dots \ov\eta_{j_\beta}$
	with $\alpha+\beta\ne 0$.
	Then  at least one of the following statements holds:
	
\sm 	
	
{\rm (i)} There are $i_p<i_q$ such that
the kernel 
$$
\Bigl(\xi_{i_p}+\frac\partial{\partial \xi_{i_p}}\Bigr)
\Bigl(\xi_{i_q}+\frac\partial{\partial \xi_{i_q}}\Bigr)
K(z,\xi;\ov u,\ov\eta;\fra)
$$
has the form 
\begin{equation}
\sigma\cdot \exp\Bigl\{\text{\rm quadratic expression in $z$, $\xi$, $\ov u$,
$\ov\eta$ }
\Bigr\}.
\label{eq:for-lemma}
\end{equation}

{\rm (ii)} There are $i_p$, $j_r$ such that
the following expression has the form \eqref{eq:for-lemma}
$$
\Bigl(\xi_{i_p}+\frac\partial{\partial \xi_{i_p}}\Bigr)
K(z,\xi;\ov u,\ov\eta;\fra)
\Bigl(\ov\eta_{j_r}+\frac{\eth}{\eth \ov\eta_{j_r}} \Bigr)
.$$

{\rm (iii)} There are  $j_r<j_s$ such that
the following expression has the form \eqref{eq:for-lemma}
$$
K(z,\xi;\ov u,\ov\eta;\fra)
\Bigl(\ov\eta_{j_r}+\frac{\eth}{\eth \ov\eta_{j_r}} \Bigr)
\Bigl(\ov\eta_{j_s}+\frac{\eth}{\eth \ov\eta_{j_s}} \Bigr)
.$$
	\end{lemma}

{\sc Proof of the lemma.} Consider the sum  in curly brackets in
\eqref{eq:Kdown2} and its subsum $S$ consisting of 
all summand with $\xi_{i_p}\xi_{i_q}$, $\xi_{i_p}\ov \eta_{j_r}$,
$\ov \eta_{j_r} \ov \eta_{j_s}$ with $p$, $q\le\alpha$
and $r$, $s\le\beta$. Clearly, this sum is nonzero. We take its
nonzero summand. There are 3 variants. To be definite, let the 
summand with $\xi_{i_p}\xi_{i_q}$ be nonzero.

To simplify notation, we can assume $\xi_{i_p}\xi_{i_q}=\xi_1\xi_2$.
So, the sum $U(z,\xi;\ov u, \ov \eta;\fra)$
contains a summand $\sigma \xi_1\xi_2$ with $\sigma_\downarrow\ne 0$.
We split the sum $U$ in \eqref{eq:K-U} as
$$U_{12}+U':=
\bigl( \sigma \xi_1\xi_2+ \xi_l \ell_1 +\xi_2\ell_2\bigr) + U' 
$$
where $U'$ does not depend on $\xi_1$, $\xi_2$ and
$\ell_{1}$, $\ell_{2}$ are expression of parity $\1$ depending
on $\fra$, $z$, $\ov u$, $\ov \eta$ and $\xi_3$, $\xi_4$, \dots.
So,
$$
\Bigl(\xi_2+\frac\partial{\partial \xi_2}\Bigr)
\Bigl(\xi_1+\frac\partial{\partial \xi_1}\Bigr)
K=
\Bigl(\xi_2+\frac\partial{\partial \xi_2}\Bigr)
\Bigl(\xi_1+\frac\partial{\partial \xi_1}\Bigr)\exp\{U_{12}\}
\cdot \exp\{U'\}
$$
We have
$$
\exp\{U_{12}\}=1+\sigma \xi_1\xi_2+ \xi_l \ell_1 +\xi_2\ell_2-\boxed{\xi_1\xi_2\ell_1\ell_2}.
$$
So, 
\begin{multline*}
\Bigl(\xi_2+\frac\partial{\partial \xi_2}\Bigr)
\Bigl(\xi_1+\frac\partial{\partial \xi_1}\Bigr)
\exp\{U_{12}\}=\xi_2\xi_1+\sigma +\xi_2\ell_1-\xi_1\ell_2-\boxed{\ell_1\ell_2}
=\\\sigma\cdot\Bigl(1+\sigma^{-1}\,
\bigl[\xi_2\xi_1+\xi_2\ell_1-\xi_1\ell_2-\ell_1\ell_2 \bigr]\Bigr)=
\sigma\cdot\exp\Bigl\{ \sigma^{-1}\,
\bigl[\xi_2\xi_1+\xi_2\ell_1-\xi_1\ell_2-\ell_1\ell_2 \bigr]\Bigr\}.
\end{multline*}
We emphasize that 
$$
\frac12\bigl[\dots\bigr]^2=
-\xi_2\xi_1\cdot \ell_1\ell_2-\xi_2\ell_1\cdot \xi_1\ell_2=0. 
$$ 

The cases of non-zero summands $\xi_{i_p}\ov\eta_{j_r}$,
or $\ov\eta_{j_r}\ov\eta_{j_s}$
 are similar.
\hfill $\square$

\sm 

{\sc Proof of Proposition \ref{pr:atlas-GB}.}
 By \eqref{eq:D-changes}, the lemma justifies the inductive step $\alpha+\beta\mapsto \alpha+\beta+2$.
\hfill $\square$

\sm

Proposition \ref{pr:atlas-GB} implies the following corollary.

\begin{corollary}
	\label{cor:unit-not-in-kernel}
{\rm a)} Let a Gauss--Berezin operator satisfies $(R\cdot 1)_\downarrow\ne 0$. Then it can be represented in the form
$\frD_{i_1}\dots \frD_{i_\alpha} R$, where $R$ is a Gauss--Berezin operator
in the narrow sense.

\sm

{\rm b)} Let for  
a  Gauss--Berezin operator the projection of $(\im R)_\downarrow$ to the vacuum vector $1$
is not zero.
Then it can be represented in in the form
$R \frD_{j_1}\dots \frD_{j_\beta}$, where
$R$ is a Gauss--Berezin operator in the narrow sense.
\end{corollary}

\sm

{\bf \punct Gauss--Berezin operators and linear relations.}
In Subsect. \ref{ss:creation-relation}, we assigned a linear relation
$\frL(R):\sheis_{\F^0}[\cA]\tto \sheis_{\ov F}[\cA]$ for any operator
$R: \F^0[\cA]\to \ov\F[\cA]$. By Proposition  \ref{pr:cLP-isotropic},
 this relation is isotropic. We intend to show that
for Gauss--Berezin operators these relations are Lagrangian. 

\sm

\begin{theorem}
	\label{th:operator-relation}
For a Gauss--Berezin operator $R=\cB\left[\begin{array}{cc|c}
	A&B&p^t\\
	C&D&q^t
\end{array}\right]:\,
\F^0[\cA]\to \ov\F[\cA]$  the linear relation $\frL(R)$ consists of $v\oplus w$
satisfying the equation
\begin{multline}
		\begin{pmatrix}
	 v^+_\0& v^+_\1	&	w^-_\0& w^-_\1 
	\end{pmatrix}=\\=
			\begin{pmatrix}
		v^-_\0& v^-_\1	&	w^+_\0& w^+_\1 
	\end{pmatrix}
	\begin{pmatrix}
		-A_{11}&-A_{12}&B_{11}&B_{12}\\
		-A_{21}&-A_{22}&B_{21}&B_{22}\\
		C_{11}&C_{12}&-D_{11}&-D_{12}\\
		-C_{21}&-C_{22}&D_{21}&D_{22}
	\end{pmatrix}
	\label{eq:gauss-relation}
.\end{multline}
For a given $v\oplus w\in\frL(R)$ we have 
\begin{equation}
\chi= v_\0^- p_1^t+ v_\1^- p_2^t - w_\0^+ q_1^t+w_\1^+ q_2^t. 	
	\label{eq:gauss-constant}
\end{equation}

\end{theorem}

{\sc Proof.} We evaluate
$\wh a(v)R$ and $R\wh a(w)$ using correspondences 
\eqref{eq:corr-1}--\eqref{eq:corr-2} and
come to the equation
\begin{multline}
	\Bigl( v_\0^+ z^t+v_\1^+ \xi^t+ v_\0^- \Bigl[\frac{\partial}{\partial z}\Bigr]^t+
	 v_\1^- \Bigl[\frac{\partial}{\partial \xi}\Bigr]^t
	\Bigr)	 K(z,\xi;\ov u,\ov\eta;\fra)-
	\\-\Bigl( w_\0^+ \Bigl[\frac{\partial}{\partial \ov u}\Bigr]^t
	-w_\0^+ \Bigl[\frac{\partial}{\partial \ov\eta}\Bigr]^t
	+w_\1^-\, \ov u+ w_\1^-\,\ov \eta
	\Bigr)	 K(z,\xi;\ov u,\ov\eta;\fra)+\chi
	K(z,\xi;\ov u,\ov\eta;\fra)
	=0.
\end{multline}
Here a term $v_\0^- \bigl[\frac{\partial}{\partial z}\bigr]^t$
denotes the product of a matrix-row $v_\0^-$ and a matrix column
consisting of partial derivatives $ \frac{\partial}{\partial z_j}$.
The minus in brackets  in the second line arises from a right derivative
in formula \eqref{eq:corr-2}.

We  differentiate an  exponential
of a quadratic form  in the usual way,
\begin{align*}
	\frac{\partial}{\partial \xi_1} \exp\Bigl\{\frac 12 \xi A_{22}\xi^t\Bigr\}
	&=\Bigl(\sum {a_{22}^{1j} \xi_j^t}\Bigr) \exp\Bigl\{\frac 12 \xi A_{22}\xi^t\Bigr\},
	\\	
	\frac{\partial}{\partial \xi_1} \exp\Bigl\{\frac 12 (\xi B_{12}\ov u^t+
	\ov u C_{12} \xi^t)\Bigr\}	
	&=\Bigl(\sum b_{12}^{1j}\ov u_j^t \Bigr)
	 \exp\Bigl\{\frac 12 (\xi B_{12}\ov u^t+
\ov u C_{12} \xi^t)\Bigr\},	
\end{align*}
etc.  We have
$$
\Bigl(v_\0^-\Bigl[\frac{\partial}{\partial z}\Bigr]^t\Bigr) K=
\Bigl( v_\0^- \bigl(A_{11} z^t +A_{12}\xi^t+B_{11} \ov u^t+B_{12}\ov\eta^t+
p_\0^+
 \bigr)\Bigr) K
$$
and  corresponding terms with $v_\0^-$, $w_\0^+$, $w_\1^+$.
We get an equation of the type $(\dots)K=0$ with    20 summands in the brackets.
Reducing similar terms in brackets with $z$, $\xi$, $\ov u$, $\ov\eta$ we
come to \eqref{eq:gauss-relation}.  Constant terms give \eqref{eq:gauss-constant}.
Cf. \cite{Ber-second}, Sect. 4-5, \cite{Ner-book}, Subsect. II.6.4,
Subsect. IV.4.6.
\hfill $\square$

\sm

{\sc Remark.} So, we see that
$$
\chi=J(v,\pi)-J(w,\kappa),
$$
where $\pi=0\oplus 0\oplus p_1\oplus (0-p_2)$, $\kappa=q_1\oplus q_2\oplus 0\oplus 0$.
So we have a super-affine relation 
$$
\frL^{\rm{aff}}(R):\,\,\pi\oplus \kappa +\frL(R):\sheis_{\F^0}[\cA]\tto \sheis_{\ov\F}[\cA].
$$
Heuristically, we claim that
$$
\frL^{\rm{aff}}(R)\circ \frL^{\rm{aff}}(Q)=\frL^{\rm{aff}}(RQ).
$$
But here we have  formal difficulties with boundedness of 
Gausss--Berezin operators and good properties of products
of super-affine relations.
\hfill $\boxtimes$

\begin{lemma}
$\frL(\frD_j)$ is the graph of the reflection $\sigma_j$
defined be \eqref{eq:sigmasf}. 
\end{lemma}

This statement is obvious.

\begin{theorem}
{\rm a)}	For each Gauss--Berezin operator $R:\F^0[\cA]\to \ov\F[\cA]$ 
the set $\frL(R)$
	is a super-Lagrangian relation $\sheis_{\F^0[\cA]}\tto \sheis_{\ov\F^0[\cA]}$.
	
\sm 

{\rm b)} For an operator $Q$ given by \eqref{eq:GB-general} we have
$$
v\oplus w\in \frL(Q)\qquad \leftrightarrow
  \qquad (\sigma_{j_1}\dots \sigma_{j_\beta} v)\oplus (\sigma_{j_1} \dots \sigma_{j_\beta}w)\in \frL(R),
$$
where $\sigma_j$ are the reflections \eqref{eq:sigmasf}.
 
 \sm 
 
{\rm c)}    Let $R$, $Q:\F^0(\cA)\to \ov\F(\cA)$ be 
$l$-operators,  for which the sets 
of solutions  of \eqref{eq:ca-relations}
coincide. Let $R$ be a Gauss--Berezin operator.
Then $Q=\mu R$ for some $\mu\in\cA$. 
\end{theorem}

This statement in a super-copy of similar fermionic
 and   bosonic statements
   (\cite{Ner-fermion}, Theorem 3, \cite{Ner-book}, Theorem II.6.4, 
   Subsect. V.4.6).

\sm

{\sc Proof.}  Let $R$ be a Gauss--Berezin operator in the narrow
sense. Then $\frL(\R)$ is defined by  equation \eqref{eq:gauss-relation}.
A straightforward calculation shows  that such matrices are precisely  Potapov transforms 
of super-Lagrangian relations.

Next, the statement holds for the operators $\frD_j$. The corresponding 
super-linear relations are graphs of reflections $\sigma_j$. This implies
statements a) and b).

Statement c) reduces to the case then $R$ is a Gauss--Berezin operator
in narrow sense. We denote by $L$ and $K$ %(z,\xi;\ov u,\ov\eta;\fra)
the kernels of $Q$ and $R$ and set
$S:=L K^{-1}$. Consider the differential operator
$$
\cD=v_\0^- \Bigl[\frac{\partial}{\partial z}\Bigr]^t+v_\0^- \Bigl[\frac{\partial}{\partial \xi}\Bigr]^t
-
w_\0^+ \Bigl[\frac{\partial}{\partial \ov u}\Bigr]^t
+w_\0^+ \Bigl[\frac{\partial}{\partial \ov\eta}\Bigr]^t
$$
Let $v$, $w$, $\chi$  conditions \eqref{eq:ca-relations}.
Then
\begin{multline*}
0=\Bigl( v_\0^+ z^t+v_\1^+ \xi^t 
+\cD
-w_\1^-\, \ov u- w_\1^-\,\ov \eta-\chi
\Bigr)\,\bigl( S \cdot K\bigr)=\\=
(\cD S)\cdot K+S\cdot \Bigl( v_\0^+ z^t+v_\1^+ \xi^t 
+\cD
-w_\1^-\, \ov u- w_\1^-\,\ov \eta-\chi
\Bigr)K= (\cD S)\cdot K+0.
\end{multline*}
But coefficients $v_\0^-$, $v_\1^-$, $w_\0^+$, $v_\1^+$
can be arbitrary. Therefore all partial derivatives of $S$ are 0.
\hfill $\square$

\sm

\begin{corollary}
	\label{cor:one-side}
For a Gauss--Berezin operator in the narrow sense 
we have 
\begin{align}
&\dom \frL(R)\supset \sheis_{\F^0}^+[\cA],
\qquad  \im \frL(R)\supset \sheis_{\ov\F}^-[\cA],\\
&\ker \frL(R)\subset \sheis_{\F^0}^-[\cA],
\qquad  \indef \frL(R)\subset \sheis_{\ov\F}^+[\cA].
\end{align}
\end{corollary}

These statements are immediate consequences
of Theorem \ref{th:operator-relation}.

\section{Representation of the semigroup of superannuli%
\label{s:last}}

\COUNTERS 
  
{\bf \punct The subspace $\boldsymbol{\cF^\infty}$.}
Fix complex parameters $\mu$, $\nu$.
Consider the representation \eqref{eq:fock-L}-\eqref{eq:fock-M}
of $\ns$  in 
$\F^0=F^0\otimes \Lambda^0$ (see \cite{IK1}).
The `vacuum vector' $1\in \F^0$ satisfies the conditions
$$
L_0 1= h \cdot 1, \qquad \zeta 1= c\cdot 1,
$$
where 
$$
h=\frac12(\mu^2+\nu^2), \qquad c=1+12 \nu^2.
$$
%Recall that for $h$, $c$ in general position
% we get an irreducible $\ns$-module $L(h,c)$;
% the exceptional set is a union
% of a countable family of quadrics, the explicit condition was 
% obtained by Kac \cite{Kac}. 

% For the correspondence with Subsects \ref{ss:super-Fock}, \ref{ss:Gauss-Berezin} 
We shift subscripts of the fermionic variables and redenote
$$\xi_{j}\mapsto \xi_{-1/2+j}.$$
Consider the operator 
$$\cL=\boxed{1}+\sum_{j>0}j z_j\frac{\partial}{\partial z_j}+ 
\sum_{l>0, l-1/2\in \Z} l \xi_l \frac{\partial}{\partial \xi_l}. 
$$
Denote by $\ccF^p$ the set of all vectors  $f\in \F=F\otimes\Lambda$
such that 
$$
\la \cL^p f,f\ra_{\F}<\infty.
$$
So $\ccF^{2p}$ is the domain of the operator $\cL^p$ in $\F$.
%% as in Subsect. \ref{ss:integration},
%$\cH_p$ is the domain of $\cL^p$, and
Denote
 $$\ccF^\infty:=\cap\, \ccF^p.$$
 It is the space $\cH_\infty$ as in Subsect. \ref{ss:integration},
i.e.,   the space of smooth vectors of the Neveu--Schwarz algebra $\ns$.
 
The monomials 
$$
z^\alpha \xi^\epsilon=z_1^{\alpha_1} z_2^{\alpha_2}\dots
\xi_{1/2}^{\epsilon_{1/2}} \xi_{3/2}^{\epsilon_{3/2}}
\dots, \text{where $\alpha_j=0$, $\epsilon_l=0,1$,} 
$$
 form an eigenbasis of $\cL$,
$$
\cL z^\alpha \xi^\epsilon=\bigl(1+\sum j\alpha_j+ \sum l \epsilon_l\bigr) z^\alpha \xi^\epsilon,
$$
So $\ccF^p$ is a Hilbert space with norm
$$
\|\sum_{\alpha,\epsilon} c_{\alpha,\epsilon} z^\alpha \xi^\epsilon\|^2_p=
\sum_{\alpha,\epsilon} |c_{\alpha,\epsilon}|^2 \, 
\bigl(1+\sum j\alpha_j+ \sum l\epsilon_l\bigr)^{p}\, \alpha!
$$
%\begin{align}
% L_\alpha&:= \frac 12 \sum_{i,j:\,i+j=\alpha}\sgrt{ij}: T_i T_j:+\frac 12 %\sum_{r,s:\,r+s=\alpha} :A_r A_s:,
% \label{eq:fock-L}
% \\
%M_r:&= \sum_{\alpha,s: \, \alpha+s=r} T_\alpha A_s, 
%\label{eq:fock-M}
%\end{align}

{\sc Remark.} Denote by $\cL_A$ the operator $\cL$, where the boxed unit is replaced
by $A>0$. Clearly, for all such operators subspaces $\ccF^p$ coincide. Indeed,
for $A>B$ we have $\cL_A\ge \cL_B\ge  \frac B A \cL_A$.
% In particular, this means
%that for unitary case ($h\ge 0$, $c\ge 3/2$) these subspaces $\ccF^p$
%coincide with subspaces $\cH_p$ defined in Subsect. \ref{ss:integration}.
\hfill $\square$

\sm

In Subsect. \ref{ss:creation-annihilation} we defined the space
$\sheis_\infty$ of creation-annihilation operators with rapidly decreasing
coefficients.

\begin{proposition}
{\rm a)} Any operator $\wh a(v)\in \sheis_\infty$  is continuous in
  the space $\ccF^\infty$. 	

\sm

{\rm b)} Any operator $\wh a(v)\in \sheis_\infty[\cA]$ 
is continuous in the space $\ccF^\infty[\cA]$. 
	\end{proposition}

The proposition follows from the lemma: 

\begin{lemma}
	\label{l:norms-creation}
	{\rm a)}  The norm of operator $z_m:\ccF^{2p}\to\ccF^{2p-2}$
	is $\le  (m+1)^{p-1}$
	
	\sm
	
	{\rm b)} The norm of operator $\frac\partial{\partial {z_m}}:\ccF^{2p}\to\ccF^{2p-2}$
	is $\le 1$
	
	\sm 
	
{\rm c)} The norm of operator $\xi_m:\ccF^{2p}\to\ccF^{2p-2}$	is $\le (1+s)^{2p-2}$.

\sm 

{\rm d)} The norm of operator $\frac\partial{\partial \xi_m}:\ccF^{2p}\to\ccF^{2p-2}$	
is $\le 1$.	
\end{lemma}	

{\sc Proof Lemma \ref{l:norms-creation}.} All these operators $z_m$, $\partial/\partial z_m$,
$\xi_m$, $\partial/\partial \xi_m$  send elements of the eigenbasis $z^\alpha \xi^\epsilon$ of  $\cL$
to elements the eigenbasis. Therefore it is sufficient to examine
$\sup_{\alpha,\epsilon}\|z_m z^\alpha \xi^\epsilon\|_{p-1}/ 
\|z^\alpha \xi^\epsilon\|_p$, etc.

\sm

{\it The statement  a)}.
\begin{multline*}
\frac{\|\cL^{p-1}z_m z^\alpha \xi^\epsilon\|^2}
{\|\cL^p z^\alpha \xi^\epsilon\|^2}=
\frac{(1+\sum j\alpha_j+m+\sum l\epsilon_l)^{2p-2}\,\alpha!\cdot(\alpha_m+1)}
{(1+\sum j\alpha_j+\sum l\epsilon_l)^{2p}\,\alpha!}
=\\=
\Bigl(\frac{1+\sum j\alpha_j+m+\sum l\epsilon_l}
{1+\sum j\alpha_j+\sum l\epsilon_l} \Bigr)^{2p-2} 
\cdot \frac{\alpha_m+1}{1+\sum j\alpha_j+\sum l\epsilon_l}\cdot
\frac{1}{1+\sum j\alpha_j+\sum l\epsilon_l}.
\end{multline*} 
We estimate the first and the second factors (the last factor is $\le 1$)
$$
1+\frac{\sum j\alpha_j+m+\sum l\epsilon_l}
{1+\sum j\alpha_j+\sum l\epsilon_l}
=1+\frac{m}{{1+\sum j\alpha_j+\sum l\epsilon_l}}\le 1+m,
$$
$$
\frac{\alpha_m+1}{1+\sum j\alpha_j+\sum l\epsilon_l}\le \frac{\alpha_m+1}{1+m\alpha_m}\le 1,
$$
and we get our estimate.

\sm

{\it The statement  b)}. If $\alpha_m> 0$ (otherwise $\frac{\partial}{\partial z_m} z^\alpha \xi^\epsilon=0$),
then
\begin{multline*}
	\frac{\bigl\|\cL^{p-1} 
	\frac{\partial}{\partial z_m} z^\alpha \xi^\epsilon\bigr\|^2}
	{\|\cL^p z^\alpha \xi^\epsilon\|^2}=
	\frac{(1+\sum j\alpha_j-m+\sum l\epsilon_l)^{2p-2}\,\alpha!\cdot\alpha_m^2/\alpha_m }
	{(1+\sum j\alpha_j+\sum l\epsilon_l)^{2p}\,\alpha!}
	=\\=
	\Bigl(\frac{1+\sum j\alpha_j-m+\sum l\epsilon_l}{1+\sum j\alpha_j+\sum l\epsilon_l} \Bigr)^{2p-2} 
	\cdot 
	\frac{\alpha_m}{1+\sum j\alpha_j+\sum l\epsilon_l}\cdot \frac{1}{1+\sum j\alpha_j+\sum l\epsilon_l}
\end{multline*}
The middle factor 
$$
\frac{\alpha_m}{1+\sum j\alpha_j+\sum l\epsilon_l}\le
 \frac{\alpha_m}{1+\alpha_m}\le 1,
$$
the first and the last factor are $\le 1$.

\sm

{\it The statement  c)}. Let $\epsilon_s=0$ (otherwise $\xi_s z^\alpha \xi^\epsilon=0$). Then
\begin{multline*}
	\frac{\|\cL^{p-1}\xi_s z^\alpha \xi^\epsilon\|^2}
	{\|\cL^p z^\alpha \xi^\epsilon\|^2}=	
		\frac{(1+\sum j\alpha_j+\sum l\epsilon_l+s)^{2p-2}\,\alpha!}
	{(1+\sum j\alpha_j+\sum l\epsilon_l)^{2p}\,\alpha!}
	=\\=
	\Bigl(\frac{1+\sum j\alpha_j+\sum l\epsilon_l+s}
	{1+\sum j\alpha_j+\sum l\epsilon_l} \Bigr)^{2p-2} 
	\cdot \frac{1}{(1+\sum j\alpha_j+\sum l\epsilon_l)^2}
	\le (s+1)^{2p-2}\cdot 1.
\end{multline*}

{\it The statement d).} Let $\epsilon_s=1$,
otherwise $\frac\partial{\partial\xi_s}
	 z^\alpha \xi^\epsilon=0$
\begin{equation*}
	\frac{\|\cL^{p-1}\frac\partial{\partial\xi_s}
	 z^\alpha \xi^\epsilon\|^2}
	{\|\cL^p z^\alpha \xi^\epsilon\|^2}=	
		\frac{(1+\sum j\alpha_j+\sum l\epsilon_l-s)^{2p-2}\,\alpha!}
	{(1+\sum j\alpha_j+\sum l\epsilon_l)^{2p}\,\alpha!}.
	\end{equation*}
Clearly, the expression is $\le 1$.
\hfill $\square$

\sm

{\bf \punct Gauss--Berezin operators and products of affine relations.}
Let $P$ be a Gauss--Berezin operator, which is bounded in 
$\ccF^\infty[\cA]$.
Then we have 
$$
\wh a(v) P=P\wh a(w)+ \chi P %J(v\oplus w, p\oplus q)
$$
for $v\oplus w$ ranging in a certain Lagrangian relation 
$\frL(P):\sheis_{\ccF^\infty}[\cA]\tto \sheis_{\ccF^\infty}[\cA]$.
and $\chi\in \cA_\0$ depends on $w\oplus v$.
We say that $P$ is {\it sufficiently good} if
$$
\chi(w\oplus v)=J(w,\pi)-J(v,\rho)\qquad
 \text{for some $\pi$, $\rho\in \sheis_{\ccF^\infty}[\cA]$.}
$$
Notice that $\pi\oplus \rho$ is defined upto
addition of elements of $\frL(P)$. We assign to $P$
the super-affine relation 
$$\wt \frL(P)=\pi\oplus\rho+\frL(P): \, \cW[\cA]\tto\cW[\cA].$$

Let $R$ be another Gauss--Berezin operator in $\ccF^\infty[\cA]$,
let $\wt \frL(R)=\phi\oplus\psi+ \frL(R)$ be the corresponding super-affine relation,
$$
\wh a(w) R=R\wh a(y)+ J(y,\phi)P-J(w,  \psi)R
$$

\begin{proposition}
	\label{pr:product-good-relations}
	Now let $P$, $R$, $Q$ be sufficiently good.
	Let $\frL(P)$, $\frL(R)$ be transversal and
	$\wt \frL(Q)=\wt\frL(P)\circ \wt\frL(R)$ 
	Then $Q=PR$.
\end{proposition}

{\sc Proof.} Since $\im R+\dom P=\sheis_{\ccF^\infty}[\cA]$,
So, $\psi-\pi$ can be represented as $w'-w''$ (in a unique way), where
$w'\in \im \frL(Q)$, $w'' \in \dom \frL(P)$.
So there is $y'$ such that $y'\oplus w'\in \frL(Q)$ and 
$v''$ such that $w''\oplus V''\in \frL(P)$.
Now we can correct $\phi\oplus \psi$ and 
get new equation of the type
\begin{align*}
\wh a(v) P=P\wh a(w)+J(w,\pi^\circ)P-J(v,\rho^\circ)P,
\\
 \wh a(w)R=R\wh a(y)+J(y,\phi^\circ)R- J(w,\pi^\circ)R.
\end{align*}
Now let $\daleth\oplus \beth\in \wt\frL(P)\circ\wt\frL(R)$.
Then there is $\gimel$ such that $\daleth\oplus\gimel\in \wh\frL(R)$,
$\gimel\oplus \daleth\in \wh\frL(P)$.
So we have 
\begin{align*}
\gimel\oplus \daleth-\pi^\circ\oplus \rho^\circ
=:w\oplus v\in \frL(P);
\\
\daleth\oplus\gimel-\phi^\circ\oplus\pi^\circ  =:y\oplus w\in \frL(R)
\end{align*} 
(two $w$ in the right hand part are the same, $w=\gimel-\pi^\circ$).
So we have
\begin{multline*}
\wh a (v)PR= \bigl(\wh a(w) P+ J(w,\pi^\circ)P-J(v,\rho^\circ)P \bigr)R
=\\=
PR \wh a(y)+ J(y,\phi^\circ)PR- J(w,\pi^\circ)PR+
J(w,\pi^\circ)PR-J(v,\rho^\circ)PR
\end{multline*}

\begin{corollary}
Let $P$, $R:\ccF^\infty[\cA]\to \ccF^\infty[\cA]$ be sufficiently good Gauss--Berezin operators in the narrow sense. Then $PR$
is sufficiently good and  $\wt\frL(PR)=\wt\frL(P)\wt\frL(R)$. 
\end{corollary}

{\sc Proof.} In this case the transversality
holds automatically, see Corollary \ref{cor:one-side}.  \hfill $\square$

 \sm

{\bf \punct Construction of representations of the semigroup $\Gamma(\cA)$.}
Consider an element $\wh\cP\in\Gamma_\bullet[\cA]$ and the corresponding
super-affine relation $\Delta_{\mu,\nu}(\wh\cP)$, see Proposition \ref{pr:long}.
Comparing it with condition of Theorem \ref{th:operator-relation}
we see that this linear relation determines a
Gauss--Berezin operator $\frB[\wh\cP]:\F^0[\cA]\to \ov\F[\cA]$.

We can construct this $\frB[\wh\cP]$ as a bounded operator
in $\ccF^\infty$ in another way.
We decompose $\wh \cP$ as 
$$
	\wh\cP=
\pi_-\circ \bigl[ (p_-)_\downarrow\circ \cC_t\circ 
(p_+)_\downarrow \bigr] \pi_+,
$$
The middle factor is an element of $\Gamma_\bullet$.
We  have its representation in the Hilbert boson Fock space $\F$
and the Hilbert fermion Fock space $\Lambda$. These  operators have trivial kernels and dense images.
%Notice, that in the both cases the corresponding linear relation
%has a trivial kernel and trivial indefinity and dense image and domain
Left and right multiplications by surplace contactomorphisms preserves
these properties. 

By Proposition, it is sufficient to show that 
for any $\wh \cP$, $\wh\cQ\in \Gamma_\bullet(\cA)$
$R=\frL(\frB[\wh\cP])$ and $T=\frL(\frB[\wh\cQ])$ are transversal.
By Corollary \ref{cor:unit-not-in-kernel}, we can represent $R$, $T$ in the form
$$
R=\frD_{i_1}\dots \frD_{i_\alpha} R', \quad
T=T' \frD_{j_1}\dots \frD_{j_\beta}.
$$
where $R'$, $T'$ are Gauss--Berezin operators in the narrow sense.
By Corollary \ref{cor:one-side},
 $\frL(P')$, $\frL(Q')$ the transversality take place.
and therefore $\frL(P)$, $\frL(Q)$ also are transversal.
Now we can refer to Proposition  \ref{pr:product-good-relations}.

\section{Final remarks%
\label{s:final}}

\COUNTERS

{\bf \punct Notion of a supergroup.%
\label{ss:formalities}}
There are different ways to formalize the notions of  supergroups and supermanifolds.
For instance, definitions in Berezin, Leites  \cite{BerL},
DeWitt \cite{DeW}, Schwarz \cite{Schw} are not equivalent, 
see comparison in Schmitt \cite{Schm} and Rogers \cite{Rog} (Introduction).
 Berezin and Leites \cite{BerL} consider  supergroups
as  functors from the category of supercommutative algebras to
the category of groups (in my opinion, the paper by Berezin, Kats \cite{BK} 
demands this point of view).
 Of course, different approaches are not antipodes, and  translations
(or partial translation) are possible.

Let us try to replace our Grassmann algebra $\cA$ to an arbitrary
supercommutative algebra $\cB$. First, we need a body map $\downarrow$
from $\cB$ to $\C$. Let $b\in\cB_+$ satisfies
$b_\downarrow=0$. We need $\ln (1+b)$ for a definition of logarithmic densities,
see \eqref{eq:annulus-logarithmic}. We need $\exp(b)$ in the definition of Gauss--Berezin operators.
We need $(1+b)^{-1/2}$ for an evaluation of a super-Gaussian integrals
(see \cite{Ner-super}, Theorem 4.3). In numerous places, we need $(1+b)^{-1}$.

All such operations are well defined for nilpotent algebras $\cB$ 
or for pro-nilpotent $\cB$ (for instance, we can easily replace
the Grassmann algebra $\cA$ by the algebra of formal series
in the variables $\fra_j$, we also can add a countable 
family of commuting generators $\epsilon_1$, $\epsilon_2$, \dots).  
Choosing one of these variants,
we can consider our supergroups as functors from 
a certain category of supecommutative algebras to the category
of groups.

In this work, we discuss certain explicit questions (which are relatively far
from usual considerations of super-mathematics). In our context,
 I try to reduce a super-formalism to a minimally necessary
 level.

\sm

Some authors (see \cite{VV}, \cite{Mol}) consider certain
Banach or Fr\'echet 
completions of $\cA$. In Sections \ref{s:cont}--\ref{s:embedding},
our $\fra_j$ are kind of actual infinitesimals (and a unique natural topology here
is the topology of formal series). In Sections 
\ref{s:integration}--\ref{s:proof}, \ref{s:last},
an attempt to introduce  Banach or Fr\'echet topology produces series 
of unbounded operators and overtaking of artificial difficulties%
\footnote{The author of the present work has no ideas about infinite-dimensional supermanifolds (which are the topic of \cite{Mol}).
However, in representation theory of  `infinite-dimensional groups',
a `Lie group' is an important heuristic notion but `infinite-dimensional groups' usually are not Lie groups in any formal sense. On the other  hand,
there is a nontrivial theory of `infinite dimensional Lie groups', which
is interesting for PDE and global analysis.}. 

\sm 

{\bf \punct Other super-Virasoro algebras.}
Apparently, our statements can be easyly extended to
the Ramond Lie superalgebra. We prefer to work with spinors over
$\mathfrak{aosp}(2p|2q)$ for $p=\infty$ and $q=\infty$ and the Ramond superalgebra
requires spinors over $\mathfrak{aosp}(2p|2q+1)$ for $p=\infty$ and $q=\infty$
(these spinors are slightly different).

May be, other variants of Virasoro superalgebra related to 
contact supercicles $S^{1|k}$ are more interesting,
see \cite{KvL}, \cite{GLS}, \cite{IK2}. 

\sm

{\bf\punct Super-Riemann surfaces.} Our paper is a partial superization
of work \cite{Ner-holom}. But this work contains also constructions
of operators corresponding to more complicated Riemann surfaces, which
apparently can be superized (on the formal level such superization
exists, see Barron \cite{Bar1}--\cite{Bar2}).

Also, a question about superization of affine Lie superalgebras can be interesting (but our approach based on superspinors does not work).

\tt 

University of Graz,
\\
\phantom{.}
\hfill Department of Mathematics and Scientific computing;

Higher School of Modern Mathematics MIPT,

%\phantom{.}
%\hfill 1 Klimentovskiy per., Moscow; 

Moscow State University, MechMath. Dept;

 University of Vienna, Faculty of Mathematics.
 
 \sm

e-mail:yurii.neretin(dog)univie.ac.at

URL: https://www.mat.univie.ac.at/$\sim$neretin/ 

\phantom{URL:} https://imsc.uni-graz.at/neretin/index.html

%       {\tt Math.Dept., University of Vienna;\\
%Insitute for Information Transmission Problems;\\
% Moscow State University, MechMath Department.\\
%URL:www.mat.univie.ac.at/$\sim$neretin

\end{document}